\documentclass{article}
\usepackage{amssymb,graphicx,enumitem}
\usepackage[utf8]{inputenc}
\usepackage{color}
\usepackage{amsmath,amsthm}
\usepackage[hidelinks]{hyperref}
\usepackage{mathrsfs}
\RequirePackage[left=0.9in,right=0.9in,top=1in,bottom=1in]{geometry}
\usepackage{titlesec}
\usepackage{fancyhdr}
\usepackage[title]{appendix}
\numberwithin{equation}{section}

\theoremstyle{plain}
\newtheorem{thm}{Theorem}[section]

\newtheorem{lem}{Lemma}[section]
\newtheorem{prop}{Proposition}[section]
\newtheorem{cor}{Corollary}[section]
\newtheorem{defs}{Definition}[section]
\theoremstyle{definition}
\newtheorem{rmk}{Remark}[section]

\usepackage{bm}         

\usepackage{mathtools}  
\usepackage{xparse}     

\NewDocumentCommand{\mybar}{ O{0.60} O{3pt} m }{
    \mathrlap{\hspace{#2}\overline{\scalebox{#1}[1]{\phantom{\ensuremath{#3}}}}}\ensuremath{#3}
}

\newcommand{\rd}{\mathbb{R}^d}
\newcommand{\rone}{\mathbb{R}}
\newcommand{\intsol}{\int_{-\infty}^{\infty}}
\newcommand{\cv}{\mathcal{V}}
\newcommand{\cd}{\mathcal{D}}
\newcommand{\cw}{\mathcal{W}}
\newcommand{\ca}{\mathcal{A}}
\newcommand{\cb}{\mathcal{B}}

\newcommand{\cq}{\mathcal{Q}}

\newcommand{\cf}{\mathcal{F}}
\newcommand{\cp}{\mathcal{P}}

\begin{document}

\title{Bypassing H\"older super-critcality barriers in viscous, incompressible fluids}
\date{\today\\ \vspace{0.3in} \small{\emph{Dedicated to Daoud (David) Araki on the occasion of his 80th birthday: a father, friend, and teacher.}}}
\author{Hussain Ibdah\footnote{University of Maryland, College Park, MD 20742, USA (hibdah@umd.edu)}}
{\let\newpage\relax\maketitle}
\abstract{This is the second in a series of papers where we analyze the incompressible Navier-Stokes equations in H\"older spaces. We obtain, to our knowledge, the very first genuinely super-critical regularity criterion for this system of equations in any dimension $d\geq3$ and in the absence of physical boundaries. For \emph{any} $\beta\in(0,1)$, we show that $L_t^1C_x^{0,\beta}$ solutions emanating from smooth initial data do not develop any singularities. The novelty stems from obtaining new bounds on the fundamental solution associated with a one-dimensional drift-diffusion equation in the presence of destabilizing singular lower order terms. Such a bound relies heavily on the symmetry and pointwise structure of the problem, where the drift term is shown to ``enhance'' the parabolic nature of the equation, allowing us to break the criticality barrier. Coupled with a subtle regularity estimate for the pressure courtesy of Silvestre, we are able to treat the (incompressible) Navier-Stokes equation as a perturbation of the classical drift-diffusion problem. This is achieved by propagating moduli of continuity as was done in our previous work, based on the elegant ideas introduced by Kiselev, Nazarov, Volberg and Shterenberg.\\

\noindent\textbf{2010 MSC:} 35Q30, 76D03, 35B65, 35B50\\
\textbf{Keywords:} Incompressible Navier-Stokes; regularity; maximum principle
\section{Introduction and main results}
In 1958, John Nash published one of his cornerstone results \cite{Nash1958}: he proved that (bounded) solutions to 
\begin{equation}\label{paraeqnash}
	\partial_t u-\nabla\cdot[C\nabla u]=0,
\end{equation}
where $C$ is a symmetric matrix, gain H\"older continuity in space and time provided that there exists some $\lambda,\Lambda>0$ for which $\lambda |z|^2\leq Cz\cdot z\leq\Lambda |z|^2$, a condition commonly referred to as uniform ellipticity. The remarkable fact is that the a-priori bound obtained \emph{is independent of any continuity assumption on $C$}. In particular, the drift term is essentially a distribution. To be precise, he showed that there exists two constants, $\alpha\in(0,1)$ and $A>0$, \emph{depending only on the dimension $d$ and the uniform ellipticity constants $\lambda,\Lambda$}, such that 
\begin{equation}\label{Nashest}
	|u(t,x)-u(t_0,y)|\leq A\|u_0\|_{L^{\infty}}\left[\frac{|x-y|^{\alpha}}{t_0^{\alpha/2}}+\frac{(t-t_0)^{\frac{\alpha}{2(\alpha+1)}}}{t_0^{\frac{\alpha}{2(\alpha+1)}}}\right],
\end{equation}
holds true for every $t\geq t_0>0$ and every $x,y$. At the same time (and independently) De Giorgi \cite{DeGiorgi1957} proved an analogous result for the elliptic system, which was followed by a different proof due to Moser \cite{Moser1960} a couple of years later. The resulting technique/theorem is commonly referred to as the De Giorgi-Nash-Moser iteration, and remains to this day one of the fundamental tools used in analyzing parabolic/elliptic PDEs. 

A natural question to ask is this: can we obtain an analogous result for a parabolic equation that isn't in divergence form? Let us look at this from the perspective of a simple drift-diffusion equation:
\begin{equation}\label{ddintro}
	\partial_tu-\nu\Delta u+b\cdot\nabla u=0,
\end{equation}
 where $b$ is given, and say $u(0,x)=u_0(x)$ is bounded. Is it possible to prove an estimate analogous to \eqref{Nashest}? That is, can we show a smoothing effect, or can we at least guarantee that the solution will not ``lose'' some degree of regularity in finite time, while assuming the drift velocity is merely measurable? The answer is no, and there have been a number of counter-examples. We focus on the (2-dimensional) counter-example presented in \cite{SVZ2013} as it is the most relevant to our discussion. The authors show that even for $C^2$ initial data, one can always construct a divergence-free drift velocity such that the solution will break \emph{any} modulus of continuity in finite time. In particular, it simply is not possible, in general, to extend Nash's result to parabolic equations that aren't in divergence form. 

So the question now should be modified to read: how much regularity do we need to impose on the drift velocity in order to make sure that the solution to \eqref{ddintro} either experiences a smoothing effect, or at least propagates whatever regularity it started with? Let us start by noting that \eqref{ddintro} has a natural scale invariance: for $\epsilon>0$, $u_{\epsilon}(t,x):=\epsilon u(\epsilon^2t,\epsilon x)$ solves \eqref{ddintro} with drift $b_\epsilon(t,x):=\epsilon b(\epsilon^2t,\epsilon x)$ provided $u$ solves \eqref{ddintro} with drift $b$ (and rescaled initial data). The general ``folklore''\footnote{We say ``folklore'' because, to our knowledge, there dose not exist a concrete theorem that rigorously justifies this argument for any space.} is that if $b$ lies in some normed space $X$ such that $\|b_{\epsilon}\|_X=\epsilon ^{r}\|b\|_X$, for some $r\geq0$, then one would ``expect'' the solution $u$ to be unique, to propagate the regularity it started with, and/or to experience smoothing effects. The same idea applies to nonlinear PDEs: in general, one would expect good local well-posedness theory in critical ($r=0$) or subcritical ($r>0$) spaces. In particular, it is perfectly reasonable to conjecture that the solution will not lose regularity provided a critical or sub-critical norm is under control. This has been demonstrated throughout the years for various PDEs and in various spaces; even if we specialize to drift-diffusion type models, the list is far too long to include here. Nevertheless, to the best of our knowledge, all results regarding the drift-diffusion equation \eqref{ddintro} fall into one of two categories: (i) assume $b$ lies in some critical or sub-critical space, then $u$ is unique and has some degree of regularity, or (ii) there exists some drift $b$ in a super-critical space for which the solution loses some degree of regularity in finite time and/or is not unique. 

On the other hand, we point out that there \emph{does not exist any rigorously justified result that says super-critical drift implies loss in regularity, or critical/sub-critical drift implies propagation of regularity}, at least we are not aware of any such results. On the contrary, one of the results in \cite{SVZ2013} says that if we specialize to two dimensions and assume the drift is independent of time and is in $L^1_{loc}$, then the solution obeys a logarithmic modulus of continuity. Similarly, the aforementioned result of Nash does not require the coefficients to be in scale invariant spaces. In fact, in the simple one-dimensional setting, for instance, \eqref{paraeqnash} reads $\partial_tu-c\partial^2_xu-\partial_xc\partial_xu=0$, and bound \eqref{Nashest} requires zero regularity assumption to be made on the drift $\partial_xc$. It was also shown in our previous work \cite{Ibdah2022} that if we assume the existence of an arbitrary $\beta\in[0,1)$ such that $b\in L^{1/(1+\beta)}_tC_x^{0,\beta}$ with the H\"older semi-norm being non-decreasing as a function of time, then $\nabla u\in L^1_tL_x^{\infty}$ (see also the heuristics in \S\ref{heur1} below). Notice that the $L_t^1\dot{W}^{1,\infty}$ (semi)-norm is at the critical level, while the drift velocity $b$ is assumed to be at the super-critical level (the space $L_t^pC_x^{0,\beta}$ is critical when $p=2/(1+\beta)$). As far as we can tell, the special example mentioned earlier in this paragraph discovered in \cite{SVZ2013} together with this previous result of ours is, at the time, the closest one could get to extending Nash's argument to an equation that is not in divergence form. 

The purpose of this work is to push in the direction of extending such a-priori bound to the incompressible Navier-Stokes system, which from a scaling point of view is indistinguishable from the drift-diffusion equation \eqref{ddintro}, once the pressure is rescaled appropriately. Indeed, this system of equations read as 
\begin{equation}\label{NSE}
\begin{cases}
	\partial_t u(t,x)-\nu\Delta u(t,x)+(u\cdot\nabla)u(t,x)+\nabla p(t,x)=F(t,x), \\
	\nabla\cdot u(t,x)=0, \\
	u(0,x)=u_0(x),
\end{cases}
\end{equation}
where $F$ is a given forcing term. Setting $F=0$, we see that if we rescale the pressure according to $p_{\epsilon}(t,x)=\epsilon^2p(\epsilon^2 t,\epsilon x)$ (while $u$ is rescaled as before), then we get the same scale-invariance as the drift-diffusion equation. The main idea is to track the evolution of moduli of continuity (adapt the elegant method introduced in \cite{KNS2008, KNV2007}). As was shown in our previous work \cite{Ibdah2022}, very loosely speaking one has the bound $|u(t,x)-u(t,y)|\leq\Omega(t,|x-y|)$ provided $\Omega$ is a non-negative, non-decreasing solution to 
\begin{equation}\label{modcondintro}
\partial_t\Omega(t,\xi)-4\nu\partial^2_\xi\Omega(t,\xi)- g(t)\xi^{\beta}\partial_\xi\Omega(t,\xi)\geq C_{d,\beta} g(t)\int_0^{\xi}\partial_\eta\Omega(t,\eta)\eta^{\beta-1}d\eta,\quad (t,\xi)\in(0,T)\times(0,\infty),
\end{equation}
where $[u(t,\cdot)]_{C_x^{0,\beta}}\leq g(t)$, $\Omega(t,0)=0$, $C_{d,\beta}\geq0$ is some constant depending on the dimension $d$ and $\beta$, and $|u_0(x)-u_0(y)|<\Omega(0,|x-y|)$ for $x\neq y$, see \S\ref{pfmainthm}, below. The fact that $u$ has modulus of continuity $\Omega$ means that we can control $\|\nabla u(t,\cdot)\|_{L^{\infty}}$ via $\partial_\xi\Omega(t,0)$, as well as control $\|u(t,\cdot)\|_{L^{\infty}}$ by $\|\Omega(t,\cdot)\|_{L^{\infty}}$ (pointwise in $t$). Such reduction relies heavily on the point-wise structure of \eqref{NSE}, as will be demonstrated in section \ref{secheurprel}. Although non-local, \eqref{modcondintro} is a much more simpler problem than \eqref{NSE}. It is one-dimensional, and possesses an extra degree of freedom: it is an inequality rather than an equality. This degree of freedom allows us to rigorously treat the incompressible Navier-Stokes equation as a perturbation of drift-diffusion, as we shall now demonstrate: utilizing the fact that $\partial_\xi\Omega\geq0$ means that our task reduces to obtaining bounds on solutions to
\begin{equation}\label{eqmodintro2}
	\partial_t\Omega(t,\xi)-4\nu\partial^2_\xi\Omega(t,\xi)-\mu_1 g(t)\xi^{\beta}\partial_\xi\Omega(t,\xi)= C_{d,\beta} g(t)\int_0^{\xi}\partial_\eta\Omega(t,\eta)\eta^{\beta-1}d\eta,
\end{equation}
with any $\mu_1\geq1$. When $\mu_1$ is large, the dynamics of the equation is dominated by the simple, one-dimensional, drift-diffusion equation: $\partial_t\Omega(t,\xi)-4\nu\partial^2_\xi\Omega(t,\xi)-\mu_1 g(t)\xi^{\beta}\partial_\xi\Omega(t,\xi)=0$. This is nothing but a sufficient condition for $\Omega$ to be a modulus of continuity for solutions to \eqref{ddintro}. That is, the non-local term in \eqref{modcondintro} corresponds to the pressure. Setting $\cv:=\partial_\xi\Omega$ and extending the drift (and initial data) in an odd fashion about $\xi=0$, we see that $\partial_t\cv-4\nu\partial^2_\xi\cv-\mu_1 g\partial_\xi[h_0\cv]=0$. The difficulty in controlling $\|\cv(t,\cdot)\|_{L^{\infty}}$ stems from the singular lower-order coefficient $h_0'$. Due to the symmetries involved, the drift term has a good sign and helps enhance the classical parabolic regularity, allowing us to bound $\cv\in L_t^1L_x^{\infty}$ in terms of the $L^1$ norm of $g$ only. The trick is to notice that the adjoint equation has no singular lower order term, a fact that we exploit to obtain bounds on the fundamental solution; such heuristics are explained in more details in section \ref{heur2}, below. Thus, for drift-diffusion, one can control the time-average of $\|\nabla u(t,\cdot)\|_{L^{\infty}}$ in terms of $\|g\|_{L^1(0,T)}$. When $C_{d,\beta}>0$ (which is the case for the NSE), no matter how large $\mu_1$ is chosen, the adjoint equation associated to the derivative of \eqref{eqmodintro2} will always have a singular order term: $\cv:=\partial_\xi\Omega$ solves $\partial_t\cv-4\nu\partial^2_\xi\cv-\mu_1 gh_0\partial_\xi\cv=(\mu_1+C_{d,\beta})gh_0'\cv$. This restricts us from having good control over $\|\nabla u(t,\cdot)\|_{L^{\infty}}$. However, the bounds on the case when $C_{d,\beta}=0$ allows us to gain good control over the expected value of the derivative of the associated stochastic flow map, which then translates to good bounds on $\|\cv\|_{L_t^qL_x^1}$ for \emph{any} $q\geq1$ simply by employing the classical Feynman-Kac formula after choosing $\mu_1$ large enough depending only on $\beta$, $q$, and the dimension $d$. More detailed heuristics are provided in section \ref{heur3}. This in turn allows us to bound $\|u\|_{L_t^qL_x^{\infty}}$, any $q\geq1$, in terms of the $L^1$ norm of $g$, from which regularity follows by the classical Ladyzhenskaja-Prodi-Serrin criteria by choosing $q>2$.

Before stating our result, let us mention that Silvestre and Vicol \cite{SV2012} previously studied both the drift-diffusion equation \eqref{ddintro} and the incompressible NSE \eqref{NSE} in H\"older, and more generally in Morrey, classes. What they were able to show is that if the drift-velocity (or the solution to NSE) lies in the \emph{critical} H\"older class $L^{2/(1+\beta)}_tC_x^{0,\beta}$ (or the respective critical Morrey class), then the solution propagates $C_x^{0,\alpha}$ regularity from the initial data $L^{\infty}$ in time, for any $\alpha\in(0,1)$ (but \emph{not} $\alpha=1$). We showed in our previous work \cite{Ibdah2022} that the solution to NSE propagates Lipschitz regularity from the initial data $L^1$ in time, provided it is in the same critical H\"older class $L^{2/(1+\beta)}_tC_x^{0,\beta}$ studied in \cite{SV2012}, thus providing a different regularity proof, and complementing the argument in \cite{SV2012} by obtaining estimates for the case when $\alpha=1$. As explained above, here we improve the regularity criterion down to $L_t^1C_x^{0,\beta}$, by controlling the $L_t^qL_x^{\infty}$ norm in terms of the $L_t^1C_x^{0,\beta}$ quantity rather than trying to control the (stronger) entity $\|\nabla u\|_{L_t^1L_x^{\infty}}$. The main result of this work is the following: 
\begin{thm}\label{mainthm}
Let $d\geq3$, let $T_*\in(0,\infty)$, suppose $F:[0,T_*)\times\rd\rightarrow\rd$ is smooth, divergence-free, and assume that $(u,p)$ is a solution to \eqref{NSE} with regularity 
\begin{equation}\label{regu}
	u\in C_t^1C_x^2((0,T_*)\times\rd)\cap C([0,T_*);W^{1,\infty}(\rd)).
\end{equation}
Assume further that the forcing term $F$ and the initial data $u_0$ are such that one of the following hypotheses is true:
	\begin{enumerate}[label=\Alph*)]
	\item $u(t,\cdot)$ is periodic in space with arbitrary period $L>0$ in every direction and has zero average or
	\item There exists an $r\in(1,\infty)$ such that $u(t,\cdot)\in L^r(\rd)\cap W^{2,\infty}(\rd)$, and it vanishes together with its gradient at infinity, i.e.
	\[
	\lim_{|x|\rightarrow\infty}|u(t,x)|+\lim_{|x|\rightarrow\infty}|\nabla u(t,x)|=0,\quad \forall t\in[0,T_*).
	\] 
	\end{enumerate}
Let $\beta\in(0,1)$ be arbitrary, and suppose there exists two continuous functions $g:[0,T_*)\rightarrow[0,\infty)$ and $f:[0,T_*)\rightarrow[0,\infty)$ such that, 
\begin{equation}\label{hypoth}
	\sup_{x\neq y}\frac{|u(t,x)-u(t,y)|}{|x-y|^{\beta}}\leq g(t),\quad \sup_{x\in\rd}|F(t,x)|\leq f(t),\quad \forall t\in[0,T_*).
\end{equation}
Then there exists a constant $C_{d}\geq 1$ depending only on the dimension $d$, such that for any $\epsilon\in(0,1]$, if we set $\mu_{d,\beta,\epsilon}:=C_d/[\beta(1-\beta)\epsilon]$,
\begin{align}
&\Gamma(t):=\int_0^tg(r)dr,\label{deffGamma}\\
&G(t):=1+\mu_{d,\beta,\epsilon}e^{\mu_{d,\beta,\epsilon}\Gamma(t)}\left[\nu^{\frac{\beta-1}{2}}\int_0^t(t-r)^{\frac{\beta-1}{2}}g(r)dr+\nu^{\frac{\beta-2}{2}}\mu_{d,\beta,\epsilon}\Gamma(t)\int_0^t(t-r)^{\frac{\beta-2}{2}}g(r)dr\right]\label{defG},
\end{align}
then the following bound holds true for every $t\in[0,T_*)$
\begin{align}
	\| u(t,\cdot)\|_{L^{\infty}}\leq \left[3\|u_0\|_{L^{\infty}}+\int_0^tf(s)ds\right][G(t)]^{\epsilon}\label{mainbd}
\end{align}
In particular, given any $q\in[1,\infty)$, if we set $\epsilon:=1/q$, then there exists a constant $M\geq 1$ depending only on $\nu, \|u_0\|_{L^{\infty}}$, $T_*$, $\beta,\ d,\ \|g\|_{L^1(0,T_*)}$, $q$, and $\|f\|_{L^1(0,T_*)}$ (which blows up as $T_*\rightarrow\infty$ but is finite for any $T_*<\infty$) such that 
\begin{equation}\label{mainbd2}
	\int_0^{T_*}\| u(t,\cdot)\|^q_{L^{\infty}}dt\leq M.
\end{equation}
\end{thm}
\begin{rmk}\label{rmkscalednorm}
 Bound \eqref{mainbd} is not sharp: one can avoid the exponential growth in $\|g\|_{L^1}$. Moreover, it is not dimensionless. The culprit behind such ``inhomogeneity'' is a technicality: it results from bounding the $L^{1}$ norm of the fundamental solution associated to the one-dimensional parabolic  equation in terms of the $L^{\infty}$ norm in a suboptimal way. By optimizing such control, we get a sharper, dimensionless bound (that is more cumbersome to deal with): for any $\epsilon>0$, if we set
 \[
 \tilde{G}(t):=\left[1+\frac{2C_d(1-\beta)^{\beta-1}}{\epsilon\beta}\int_0^tg(s)\left(\nu^{\frac{-1}{2}}(t-s)^{\frac{-1}{2}}+\frac{C_d\nu^{\frac{\beta-2}{2}}}{\epsilon\beta(1-\beta)}\int_s^t(t-r)^{\frac{\beta-2}{2}}g(r)dr\right)^{1-\beta}ds\right]^{\frac{1}{1-\beta}}
 \]
 and assume $F=0$ for simplicity, then $\|u(t,\cdot)\|_{L^{\infty}}\leq 3\|u_0\|_{L^{\infty}}\left[\tilde{G}(0,t)\right]^\epsilon$. Straightforward, yet slightly tedious, applications of Minkowski's integral inequality and Fubini-Tonelli yield that $\tilde{G}$ is bounded in $L^1$ in terms of $\|g\|_{L^1}$, as is the case for $G$. Apart from avoiding exponential growth,  $\tilde{G}$ is dimensionless: in the periodic setting, $g$ being a H\"older semi-norm, has units $L^{1-\beta} T^{-1}$, kinematic viscosity $\nu$ has units $L^2T^{-1}$, while $C_d$ is an absolute constant. This makes the units of the quantity in parenthesis under the integral $L^{-1}$, and as it is raised to the power of $1-\beta$ before being multiplied by $gds$ (which has units $L^{1-\beta}$), $\tilde{G}$ has no physical units. Consequently, it scales appropriately (with respect to the natural scaling of \eqref{NSE}): notice that $\tilde{G}=\tilde{G}(t,[u(t,\cdot)]_{C_x^{0,\beta}})$, so that a change in variables yields $\tilde{G}(\lambda^2t,[u(t,\cdot)]_{C_x^{0,\beta}})=\tilde{G}(t,[u_\lambda(t,\cdot)]_{C_x^{0,\beta}})$. Here, $\lambda>0$ and $u_\lambda(t,x):=\lambda u(\lambda^2t,\lambda x)$, from which $[u_\lambda(t,\cdot)]_{C_x^{0,\beta}}=\lambda^{1+\beta}[u(\lambda^2t,\cdot)]_{C_x^{0,\beta}}$. Moreover, no matter how small $\epsilon$ is chosen, the $L^1$ norm of $g$ will always show up regardless whether we use $G$ or $\tilde{G}$ to bound $\|u\|_{L^q_tL_x^{\infty}}$.
 \end{rmk}
 \begin{rmk}
If one would like to get uniform in time bounds on $\| u(t,\cdot)\|_{L^{\infty}}$ by means of estimate \eqref{mainbd}, then one has to make a sub-critical assumption on the solution $u$. As is well known, to rule out the formation of a singularity at $T_*$, one only needs to control $\| u(t,\cdot)\|_{L^{\infty}}$ $L^q$ in time ($q>2$), so that it is enough to make a super-critical assumption. For the sake of convenience and completeness, we give a brief proof of this fact in Appendix \ref{applps} below. 
\end{rmk}
\begin{rmk}
We also show that one can control certain subcritical H\"older ($L_t^qC_x^{0,\gamma}$) semi-norms in terms of the super-critical $L_t^1C_x^{0,\beta}$ norm, see Appendix \ref{applpbds}.	 Unfortunately, the approach presented herein restricts $\gamma\in(0,1/2)$; although one can control $\|\nabla u(t,\cdot)\|_{L^{\infty}}$ uniformly in time by bootstrapping \eqref{mainbd2}, it is an open question whether one is able to replicate the regularity $\nabla u\in L_t^1L_x^{\infty}$ available for drift-diffusion directly (without bootstrapping). We think the bound $\|\nabla u(t,\cdot)\|_{L^{\infty}}\lesssim[G(t)]^{1+\epsilon}$ should hold true for any $\epsilon\in(0,1)$ (one could also replace $G$ by $\tilde{G}$ as defined in Remark \ref{rmkscalednorm}). We do not know how to prove this, see the heuristics in \S\ref{heur3}, below, for more details.
\end{rmk}
\begin{rmk}
	The reader may wonder why the bounds blow up as $\beta\rightarrow0^+$ or $\beta\rightarrow 1^{-}$. This is a consequence of using the \emph{optimal} (homogenous) pressure regularity estimates first observed by Silvestre in \cite{Silvestre2010unpub}: $p\in C_x^{0,2\beta}$ provided $u\in C_x^{0,\beta}$. Loosely speaking, the blowup happens due to the fact that Riesz transforms are unbounded on integer spaces: as $\beta\rightarrow0^+$, one expects the pressure to be $L^{\infty}$ (a bad space for the Riesz transform). Analogously, as $\beta\rightarrow 1^{-}$, one expects the pressure to be $C^2$, also a bad space for the Riesz transform. 
\end{rmk}

As a direct corollary of the above a-priori bound, we get the following super-critical regularity criterion: 
\begin{cor}\label{supercritreg}
Let $u$ be a smooth solution to \eqref{NSE} on $(0,T_*)\times\rd$ emanating from smooth initial data $u_0$, suppose it is either periodic in space or vanishes at infinity (satisfies one of the conditions A) or B) from Theorem \ref{mainthm}), and assume $F=0$. If there exists a $\beta\in(0,1)$ such that 
\[
\int_0^{T_*}[u(t,\cdot)]_{C_x^{0,\beta}}dt<\infty,
\]
then there exists a $\delta>0$ such that $u$ is smooth on $(0,T_*+\delta]$. Equivalently, let $T_*$ be the maximal time of existence for the unique solution $u$ emanating from such initial data. Then $T_*<\infty$ if and only if for any $\beta\in(0,1)$, 
\[
	\int_0^{T_*}[u(t,\cdot)]_{C_x^{0,\beta}}dt=+\infty.
\]
\end{cor}
\begin{rmk}
We emphasize: \emph{this is a statement about classical (strong) solutions, and not Leray-Hopf weak solutions}. We do \emph{not} say that Leray-Hopf solution that are assumed to be $L_t^1C_x^{0,\beta}$ are regular (or unique): that remains open. To our knowledge, the state of the art regularity statement regarding (suitable) Leray-Hopf solutions remains to this day the celebrated Caffarelli-Kohn-Nirenberg result \cite{CKN1982}. What we are saying here is that given any solution that emanates from smooth initial data (of arbitrary size), if there exists a $\beta\in(0,1)$ such that $u\in L^1_tC_x^{0,\beta}((0,T_*)\times\rd)$, then $u$ cannot develop a singularity at time $T_*$.
\end{rmk}
\begin{proof}[Proof of Corollary \ref{supercritreg}]
Apply Theorem \ref{mainthm} with $g(t):=[u(t,\cdot)]_{C_x^{0,\beta}}$, $t\in[0,T_*)$, and $F=0$. It is well known that $u$ will not develop a singularity at $T_*$ provided $\| u(t,\cdot)\|_{L^{\infty}}$ is $L^q(0,T_*)$, $q>2$ (see Appendix \ref{applps} below for a simple proof of this fact), from which the claim follows. The second statement follows analogously using a straightforward proof by contradiction and an application of Theorem \ref{mainthm}.
\end{proof}
As the experienced reader may know, all known a-priori bounds for \eqref{NSE} are at the supercritical level (e.g.: $L_t^{\infty}L_x^2,\ L_t^2\dot{H}_x^1,\ L_t^1L_x^{\infty}$ and $L_t^{\zeta_m}\dot{H}_x^m$ where $\zeta_m=2/(2m-1)$ \cite{Constantin2001, Constantin2014, Duff1990, FGT1981, Vasseur2010}), while all previously known regularity criteria are (at best) at the critical level (e.g.: \cite{ESS2003, Ibdah2022, Ladyzhenskaya1967,Prodi1959,Serrin1963, SV2012}), save a few ``slightly'' super-critical criteria where one goes below the critical level by a logarithmic factor. See the introduction of \cite{Ibdah2022} for a brief survey. This is usually referred to as the ``scaling'' gap, and is one of the reasons why the global well-posedness question is very hard to address. Corollary \ref{supercritreg}, above, provides, to our knowledge, the very first \emph{genuinely} super-critical regularity criterion. We do not know whether the a-priori bound assumed in Theorem \ref{mainthm} is satisfied by solutions to \eqref{NSE}, and so \emph{this is not a solution to the global regularity problem}. We do, however, remind the reader that the $L_t^1L_x^{\infty}$ norm of Leray-Hopf solutions to \eqref{NSE} is under control (in 3 dimensions), as was first shown in \cite{FGT1981} and later on in \cite{Constantin2001, Constantin2014, Tao2013}. In particular, any upgrade of such an a-priori bound to $L_t^1C_x^{0,\beta}$, no matter how small $\beta$ is, would rule out the formation of a singularity in finite time. It may be worthwhile noting that as a consequence of the a-priori bounds proven in \cite{FGT1981}, one also has, for instance, $u\in L_t^{2/3}C_x^{0,1/2}$, see \S\ref{heur1} below for more details. We also want to point out that it is highly unlikely one can obtain a similar regularity criteria for the averaged Navier-Stokes model introduced by Tao \cite{Tao2016}, or any of the related shell-models. This is because we \emph{heavily rely on and exploit the point-wise structure} of the advective non-linearity and incompressibility constraint, in particular, we do not rely on the energy inequality, nor do we employ any energy estimates. The argument presented herein is in some sense a maximum principle type of result (we are \emph{directly} measuring the evolution of $\| u(t,\cdot)\|_{L^{\infty}}$, and so the point-wise structure is crucial, as will be seen later).

At the heart of the proof of Theorem \ref{mainthm} is a rather unexpected $L_t^1L_x^{\infty}$ a-priori bound to solutions of a one-dimensional parabolic equation with a singular lower order term. Such a-priori bound is, as far as we can tell, far from being obvious and is highly non-trivial. It implies an $L_t^1\dot{W}_x^{1,\infty}$ a-priori bound on solutions to a one-dimensional drift-diffusion equation in terms of supercritical quantities of the drift. The inspiration is drawn in part from the result of Nash \cite{Nash1958}, in the sense that it is based on a-priori bounds on the fundamental solution in terms of supercritical quantities. Of course, as explained earlier, it is impossible to extend Nash's result to an abstract parabolic equation without making any regularity assumptions on the coefficients, and we do not do so here: to our knowledge, the following Theorem supersedes any of the available results so far. We state it in an informal manner here, the rigorous statement is Lemma \ref{thmbuildmod}, below.
\begin{thm}\label{secondthm}
Let $\beta\in(0,1)$, $T>0$, and let $h(\xi)$ be the odd extension about $\xi=0$	 of $\xi^{\beta}$. Let $\mu_1\geq0$ and $g:[0,T]\rightarrow[0,\infty)$ be given, and let $\cv:[0,T]\times\rone$ solve
\begin{equation}\label{scndthmeq}
	\partial_t\cv(t,\xi)-\nu\partial^2_\xi\cv(t,\xi)-\mu_1g(t)h(\xi)\partial_\xi\cv(t,\xi)-\mu_1g(t)h'(\xi)\cv(t,\xi)=0,\quad (t,\xi)\in(0,T]\times\rone,
\end{equation}
with initial data $\cv_0$. If the initial data $\cv_0$ happens to be
\begin{itemize}[label=\textbf{(S)}]
	\item non-negative, even about $\xi=0$, non-increasing on $[0,\infty)$ (and hence non-decreasing on $(-\infty,0]$) so that it is maximized at $\xi=0$,
\end{itemize}
then so is the solution $\cv(t,\cdot)$ for every $t\in[0,T]$, and we have the bound for every $t\in[0,T]$.
\[
\|\cv(t,\cdot)\|_{L^{\infty}}\leq\|\cv_0\|_{L^{\infty}}\left[1+8\mu_1e^{\mu_1\|g\|_{L^1}}\left[\nu^{\frac{\beta-1}{2}}\int_0^t(t-r)^{\frac{\beta-1}{2}}g(r)dr+\nu^{\frac{\beta-2}{2}}\mu_1\|g\|_{L^1}\int_0^t(t-r)^{\frac{\beta-2}{2}}g(r)dr\right]\right].
\]
In particular, $\cv\in L_t^1L_x^{\infty}$ provided $g\in L^1(0,T)$ and $\beta\in(0,1)$.
\end{thm}
\begin{rmk}
	Notice that the assumption (S) on the initial data (and hence, the solution) makes $\cv(t,\cdot)$ achieves its maximum at $\xi=0$, so that the singularity in $h'$ at $\xi=0$ makes it very difficult to obtain $L^{\infty}$ bounds.
\end{rmk}
\begin{rmk}
Theorem \ref{secondthm} is false if one omits (or reverses the sign of) the drift term while keeping the singular lower order term in \eqref{scndthmeq}, see \S\ref{heur3} for more details, and Appendix \ref{appce} for a counter-example essentially due to Elgindi \cite{ElgindiPrivateComm}.
\end{rmk}

The connection between Theorem \ref{secondthm} and Theorem \ref{mainthm} is explained in section \ref{secheurprel} below, and is based on extending the elegant ideas introduced in \cite{KNS2008, KNV2007} to the incompressible NSE as was done in our previous work \cite{Ibdah2022}, taking advantage of a subtle continuity estimate for the pressure first discovered by Silvestre \cite{Silvestre2010unpub}. Very loosely speaking, Theorem \ref{secondthm} is used to construct $L_t^qL_x^{\infty}$ solutions to 
\[
	\partial_t\Omega(t,\xi)-4\nu\partial^2_\xi\Omega(t,\xi)- g(t)\xi^{\beta}\partial_\xi\Omega(t,\xi)\geq C_d g(t)\left[\int_0^{\xi}\Omega(t,\eta)\eta^{\beta-2}d\eta+\frac{\Omega(t,\xi)}{(1-\beta)\xi^{1-\beta}}\right],
\]
for any $q\in[1,\infty)$, under the assumption that $\beta\in(0,1)$ and $g\in L^1$: see discussion in \S\ref{secheuristic} and Theorem \ref{thmbuildmod} below. The solution to NSE \eqref{NSE} will be shown to satisfy $|u(t,x)-u(t,y)|\leq\Omega(t,|x-y|)$, from which the claimed a-priori bound follows once we bound $\Omega$. We describe the key ideas and observations made in \S\ref{secheuristic}, and list some preliminary results in \S\ref{secprel}. The heart of the matter is section \ref{mainsec} where we analyze solutions to the above inequality. Theorem \ref{mainthm} is then proven in section 4. We complement the analysis with some results in the Appendices.
\section{Heuristics and preliminaries}\label{secheurprel}
Let us start by making some basic definitions. We slightly modify the definition of a modulus of continuity utilized in our previous work \cite{Ibdah2022}, since we do not need the level of generality required there. This makes Lemmas \ref{pressest}, \ref{gradbd}, and \ref{localmaxprincp}, below, ``corollaries'' of Lemmas  4.1, 3.2, and 3.3 respectively in our previous work \cite{Ibdah2022}, and so their proofs will be omitted.
\begin{defs}\label{defmod}
 We say a function $\omega:[0,\infty)\rightarrow[0,\infty)$ is a modulus of continuity if it is continuous, non-deceasing, $\omega\in C^2(0,\infty)$, $\omega(0)=0$, and $\omega''(0^+)=-\infty$. We call $\omega$ Lipschitz if $\omega'(0)<\infty$.
\end{defs}
A typical example of a Lipschitz modulus of continuity is $\omega(\xi)=\xi/(1+\xi^{\alpha})$, for some $\alpha\in(0,1)$.
\begin{defs}\label{deftimdepmod}
Let $T>0$ be given. A function $\Omega\in C([0,T]\times[0,\infty))$ is said to be a time-dependent modulus of continuity on $[0,T]$ if $\Omega(t,\cdot)$ is a modulus of continuity for each $t\in[0,T]$ and $\Omega(\cdot,\xi)\in C^1(0,T]$ for each fixed $\xi\in[0,\infty)$.
\end{defs}
A typical example of a time-dependent modulus of continuity is $\Omega(t,\xi)=\tilde{\Omega}(t,\xi)+\delta\omega(\xi)$, where $\delta>0$, $\omega$ is a concave modulus of continuity, and $\tilde{\Omega}$ solves $\partial_t\tilde{\Omega}-\partial^2_\xi\tilde{\Omega}-g(t)\xi^{\beta}\partial_\xi\tilde{\Omega}=f(t,\xi)$ on $(0,T)\times(0,\infty)$ together with the boundary condition $\tilde{\Omega}(t,0)=0$, and some initial condition $\tilde{\Omega}(0,\xi)=\Omega_0(\xi)$. Here, $\beta\in(0,1]$, $g$ is non-negative, with $\Omega_0$ and $f(t,\cdot)$ being non-decreasing, smooth functions that vanishes at $\xi=0$.
\begin{defs}\label{defobeymod}
Let $\omega$ be a modulus of continuity and let $u:\rd\rightarrow\rd$ be a vector field. We say $u$ has modulus of continuity $\omega(\xi)$ if $|u(x)-u(y)|\leq\omega(|x-y|)$ for every $(x,y)\in\rd\times\rd$. We say $u$ strictly obeys $\omega$ if $|u(x)-u(y)|<\omega(|x-y|)$ whenever $x\neq y$.	
\end{defs}
\subsection{Heuristics}\label{secheuristic}
\subsubsection{Transferring the evolution to moduli of continuity}\label{heur1}
We now assume that $u$ is a smooth solution to the classical drift-diffusion equation \eqref{ddintro}, equipped with smooth initial data. Here $b$ is assumed to be smooth, and we focus on either the periodic or the whole space (with decay at infinity) setting. Let us assume that for some $\beta\in(0,1)$, there exists a smooth, non-negative $g:[0,T]\rightarrow[1,\infty)$ such that $|b(t,x)-b(t,y)|\leq g(t)|x-y|^{\beta}$ for every $x,y$. As was shown in our previous work \cite{Ibdah2022} (see also \S\ref{pfmainthm} below), one can adapt the elegant ideas introduced in \cite{KNS2008,KNV2007} to show that if $\Omega$ is a smooth time-dependent modulus of continuity (as in Definition \ref{deftimdepmod}) and solves $\partial_t\Omega-4\nu\partial^2_\xi\Omega-g(t)\xi^{\beta}\partial_\xi\Omega\geq0$ on $(0,T]\times(0,\infty)$, then $|u(t,x)-u(t,y)|\leq \Omega(t,|x-y|)$ for any $t\in[0,T]$ and any $x,y$. In particular, since $\Omega(t,0)=0$, we get $\|\nabla u(t,\cdot)\|_{L^{\infty}}< \partial_\xi\Omega(t,0)$ (provided $|u(0,x)-u(0,y)|<\Omega(0,|x-y|)$ for every $x\neq y$); see Lemma \ref{gradbd} below. Perhaps the easiest way to construct such an $\Omega$ is to make the ansatz $\Omega(t,\xi):=\omega(\mu\xi)$, where $\mu=\mu(t)$. A straightforward calculation tells us that if $\mu\approx g^{1/(1+\beta)}$, and if $g$ is non-decreasing, then as long as $4\nu\omega''(\sigma)+\sigma^{\beta}\omega'(\sigma)\leq 0$, the previous conditions are met, leading to the bound 
\begin{equation}\label{c200}
	\|\nabla u(t,\cdot)\|_{L^{\infty}}< C_0g^{\frac{1}{1+\beta}}(t),\quad \forall t\in[0,T],
\end{equation}
for some $C_0\geq 1$ depending on initial data and $\nu$. From here, we clearly see that if $g\in L^{1/(1+\beta)}(0,T)$, then $\|\nabla u(t,\cdot)\|_{L^{\infty}}\in L^1(0,T)$: a critical entity of the solution is guaranteed to be under control provided a super-critical quantity of the drift is finite. Here is an explicit example: setting $\nu=1$ for simplicity, one can check that 
\[
\omega(\sigma):=
\begin{cases}
	\sigma-\sigma^{3/2},& \sigma\in[0,\delta],\\
	\omega_{\mathcal{R}}(\sigma),& \sigma\geq\delta,
\end{cases}
\]
where 
\[
\omega_{\mathcal{R}}(\sigma)=\delta-\delta^{3/2}+\frac{1}{4}\int_\delta^\sigma \exp\left(\frac{-\eta^{\beta+1}}{4(\beta+1)}\right)d\eta,
\]
and $\delta\in(0,1/4)$ satisfies all the necessary requirements. One can rescale $\omega$ to take into account size of initial data (all this was made rigorous in \cite{Ibdah2022}). The importance of the previous bound (apart from providing us with ``super-critical regularity'') stems from the fact that if $u$ is a Leray-Hopf solutions to the incompressible Navier-Stokes system (in three dimensions), then the time average of $\|u(t,\cdot)\|_{\dot{H}^2}^{2/3}$ on $(0,T)$ does not blowup, i.e., $u\in L_t^{2/3}\dot{H}_x^2$, as was first proven in \cite{FGT1981} (periodic setting) followed by \cite{Duff1990} (bounded domains). Sobolev embedding then tells us that $u\in L_t^{2/3}C_x^{0,1/2}$. It follows that if we can prove a bound similar to \eqref{c200} for the solution to \eqref{NSE}, then we can pretty much rule out the formation of a singularity in finite time, save for the non-decreasing assumption we have on $g$. The difficulty, of course, lies in controlling the pressure. In order to track the evolution of a modulus of continuity by the system \eqref{NSE}, one needs to obtain continuity estimates on $\nabla p$ in terms of those known for $u$. It is not very obvious how this can be achieved: indeed the relationship between $u$ and $p$ is given by $-\Delta p=\text{div}[u\cdot\nabla u]$, or equivalently, $p=R_iR_j(u_i u_j)$, with $\{R_i\}_{i=1}^d$ being the standard Riesz transforms, and we omitted the double sum for convenience. Standard continuity estimates available for the Riesz transform as well as standard elliptic regularity tell us that we would need continuity estimates for $\nabla u$ in order to get a continuity estimate for $\nabla p$. Remarkably, one \emph{can} prove that $p\in C^{0,2\alpha}$ provided $u\in C^{0,\alpha}$. When $\alpha\in(1/2,1)$, this translates to a H\"older estimate on the gradient of the pressure: $\nabla p\in C^{0,2\alpha-1}$, \emph{without imposing a H\"older condition on $\nabla u$}. This was realized first by Silvestre in an unpublished work \cite{Silvestre2010unpub} and is based on utilizing the incompressibility constraint in a very clever and subtle way: for any $x\in\rd$, we have the identity
\[
\sum_{i,j}\partial_{z_i}\partial_{z_j}[u_i(z)u_j(z)]=\sum_{i,j}\partial_{z_i}\partial_{z_j}[(u_i(z)-u_i(x))(u_j(z)-u_j(x))],
\]
from which the result follows from the representation
\begin{equation}\label{reppress}
\nabla p(x)=\int_{\rd} \left[u_i(x-z)-u_i(x)\right][u_j(x-z)-u_j(x)]\partial_i\partial_j\nabla\phi(z)dz,
\end{equation}
where $\phi$ is the fundamental solution to the Laplace equation, $\phi(z):=C_d|z|^{2-d}$. This was extended to the periodic setting and to abstract moduli of continuity as well in \cite[Lemma 4.1]{Ibdah2022}, which we now state. 
\begin{lem}\label{pressest}
	Let $u$ and $b$ be continuous, divergence-free vector fields. Suppose further that either $u,b\in C_{per}(\rd)$ or $b\in L^{q}(\rd)$ and $u\in L^{\infty}(\rd)$ for some $q\in(1,\infty)$. Assume $u$ and $b$ have moduli of  continuity $\omega_{u}$ and $\omega_{b}$ respectively. If 
	\[
	p:=\sum_{i,j=1}^dR_iR_j(b_iu_j),
	\]
where $\{R_j\}_{j=1}^d$ are the Riesz transforms, then $p\in C^1(\rd)$ and $\nabla p$ has modulus of continuity
\begin{equation}\label{modgradp}
\widetilde\omega(\xi):=C_d\left[\int_0^\xi\frac{\omega_{b}(\eta)\omega_{u}(\eta)}{\eta^2}d\eta+\omega_{b}(\xi)\int_{\xi}^\infty\frac{\omega_{u}(\eta)}{\eta^2}\ d\eta+\omega_{u}(\xi)\int_{\xi}^\infty\frac{\omega_{b}(\eta)}{\eta^2}\ d\eta\right],
\end{equation}
where $C_d$ is a positive, absolute universal constant depending only on the spatial dimension $d\geq3$ but not on any norm of $u$, $b$ or $p$ (provided the integrals converge).
\end{lem}
Silvestre's result was reproduced by Constantin \cite{Constantin2014} (see also \cite{DS2014, Isett2013regularity, IO2016} for similar bounds in the context of convex integration). Armed with Lemma \ref{pressest}, we show in \S\ref{pfmainthm}, below, that if we require $|u(t,x)-u(t,y)|\leq g(t)|x-y|^{\beta}$, then we get $|u(t,x)-u(t,y)|\leq \Omega(t,|x-y|)$ provided $\Omega$ is a smooth, time-dependent modulus of continuity and satisfies
\begin{equation}\label{c21}
	\partial_t\Omega(t,\xi)-4\nu\partial^2_\xi\Omega(t,\xi)- g(t)\xi^{\beta}\partial_\xi\Omega(t,\xi)\geq C_d g(t)\left[\int_0^{\xi}\Omega(t,\eta)\eta^{\beta-2}d\eta+\frac{\Omega(t,\xi)}{(1-\beta)\xi^{1-\beta}}\right]
\end{equation}
for every $(t,\xi)\in(0,T]\times(0,\infty)$. One can, of course, make the same ansatz $\Omega(t,\xi)=\omega(\mu\xi)$, and realize that since \eqref{NSE} scales exactly like \eqref{ddintro}, one should require $\mu\approx g^{1/(1+\beta)}$ in order to balance diffusion with transport and the non-local contribution from the pressure. However, the stationary $\omega$ now needs to (roughly) satisfy $4\nu\omega''(\sigma)+\sigma^{\beta}\omega'(\sigma)+\sigma^{\beta-1}\omega(\sigma)\leq0$. This is bad news, since the lower order term would cause the solution $\omega$ to be decreasing for large enough $\sigma$ (and in fact become negative at some point), meaning that it no longer is a modulus of continuity. Heuristically, this tells us that diffusion (by itself) can only balance transport (at least in the context of propagating moduli of continuity), and that one has to take advantage of the full parabolic operator, not just the elliptic part, to balance out the pressure. Taking advantage of the linearity of the condition \eqref{c21} that $\Omega$ needs to satisfy, one can instead make the ansatz $\Omega(t,\xi)=\lambda(t)\omega(\mu\xi)$, but then quickly realizes that, since $\omega''(0^+)=-\infty$, dissipation balances both the transport and the pressure term when $\sigma=\mu\xi$ is small, while for $\sigma$ away from zero, one can rely on the time derivative to absorb the instabilities from the pressure term. This would, unfortunately, require the amplitude $\lambda$ to satisfy $\lambda'\approx g^{2/(1+\beta)}\lambda$, meaning that a critical assumption needs to be made. The main contribution of this work is to consider more general moduli of continuity, those that are not necessary given by the above ansatz, in order to push towards a super-critical regularity criterion for the NSE \eqref{NSE}. 

\subsubsection{Enhanced parabolic regularity}\label{heur2}
First off, by considering $\Omega=\tilde{\Omega}+\epsilon\omega$, we can for the moment ignore the condition $\partial^2_\xi\Omega(t,0^+)=-\infty$. Let us turn our attention back to drift-diffusion and focus on solutions to the initial-boundary value problem
\begin{align*}
	&\partial_t\Omega(t,\xi)-4\nu\partial^2_\xi\Omega(t,\xi)-g(t)\xi^{\beta}\partial_\xi\Omega(t,\xi)=0,\quad (t,\xi)\in(0,T]\times(0,\infty),\\
	&\Omega(t,0)=0,\\
	&\Omega(0,\xi)=\Omega_0(\xi),
\end{align*}
where $\Omega_0$ is a modulus of continuity, in particular, it is non-decreasing. Our goal is to estimate the Dirichlet-to-Neumann map. To do so, we denote the odd extension of $\xi^{\beta}$ by $h(\xi)$, extend $\Omega_0$ in an odd fashion about $\xi=0$, and consider solutions to the PDE in the whole space. Moreover, if we further assume that $\Omega_0$ is concave on $[0,\infty)$, then it is easy to show that $\cv:=\partial_\xi\Omega$ is a solution to 
\begin{equation}\label{eqcvheur}
	\partial_t\cv(t,\xi)-4\nu\partial^2_\xi\cv(t,\xi)-g(t)\partial_\xi[h(\xi)\cv(t,\xi)]=0,\quad \cv(0,\xi)=\Omega_0'(\xi)
\end{equation}
such that $\cv(t,\cdot)$ satisfies the symmetry assumption (S) stated in Theorem \ref{secondthm} for every $t\in[0,T]$; in particular, we have $\cv(t,0)=\|\cv(t,\cdot)\|_{L^{\infty}}$ for every $t\in[0,T]$, and the anti-derivative of $\cv$ is a (concave) time-dependent modulus of continuity according to Definition \ref{deftimdepmod}. This follows from the classical maximum/minimum principle, together with the fact that $h(\xi)$ is non-decreasing on all of $\rone$ and is concave on $[0,\infty)$, see Theorem \ref{classicalthm} below. The singularity in $h'$ at $\xi=0$ makes it very difficult to obtain $L^{\infty}$ bounds on $\cv$, and frankly speaking is the heart of the matter. The key observation we make in order to overcome this difficulty is the following: let us for the moment assume $h$ and $g$ are both smooth and bounded, and let $Z(t,\xi;s,\sigma)$ be the fundamental solution to the operator 
\begin{equation}\label{defopd}
\cd_{t,\xi}:=\partial_t-4\nu\partial^2_\xi-g(t)\partial_\xi[h(\xi)\cdot];
\end{equation}
in particular, we can represent the solution by 
\[
\cv(t,\xi)=\intsol Z(t,\xi;0,\sigma)\cv_0(\sigma)d\sigma,
\]
so that to estimate $\cv(t,0)=\|\cv(t,\cdot)\|_{L^{\infty}}$, we need to get a bound on $\|Z(t,0;0,\cdot)\|_{L^1}$. To do so, we exploit the fact that, for fixed $(t,\xi)$, $Z$ solves the adjoint equation in the $(s,\sigma)$ variables (backwards in time): $-\partial_s Z-4\nu\partial^2_\sigma Z+g(s)h(\sigma)\partial_\sigma Z=0$. Notice that the singular term disappears, and we end up with nothing but a transport-diffusion equation. As a consequence, we obtain (see Theorem \ref{FSest} below)
\begin{equation}\label{c22}
	\|Z(t,0;s,\cdot)\|_{L^1}\leq 1+8e^{\|g\|_{L^1}}\left[\nu^{\frac{\beta-1}{2}}\int_s^t(t-r)^{\frac{\beta-1}{2}}g(r)dr+\nu^{\frac{\beta-2}{2}}\|g\|_{L^1}\int_s^t(t-r)^{\frac{\beta-2}{2}}g(r)dr\right],
\end{equation}
which to our knowledge is a first of a kind bound for the fundamental solution of a parabolic operator that isn't in divergence form in the presence of a \emph{singular} lower order coefficient. We mention in passing that one also has a sharper (yet more cumbersome to deal with) bound, see \eqref{scaledopnorm}, below. Bound \eqref{c22} immediately tells us that we can control $\cv$ in $L_t^1L_x^{\infty}$ in terms of the initial data and $\|g\|_{L^1}$ only, as long as $\beta>0$. Such a strategy is more versatile than dynamically rescaling a stationary modulus of continuity, and is a key ingredient used in proving Theorem \ref{mainthm}. It also allows us to drop the non-decreasing assumption we had to impose on $g$ when using the ansatz, at the expense of requiring $g\in L^1(0,T)$ rather than $L^{1/(1+\beta)}(0,T)$, as well as the introduction of a constant that blows up as $\beta\rightarrow0^+$. 

We remark that bound \eqref{c22} is \emph{not true} if we drop the transport term in \eqref{defopd} and consider instead the operator $\bar{\cd}_{t,\xi}:=\partial_t-4\nu\partial^2_\xi-gh'(\xi)$, as an example provided to us by Elgindi \cite{ElgindiPrivateComm} tells (see Appendix \ref{appce} below). The same example essentially reveals that bound \eqref{c22} is not true even if we flip the sign of the transport term. Indeed, for initial data (and hence solutions) satisfying (S) in Theorem \ref{secondthm}, the sign of the transport term is always good/stabilizing, which is what we would like to quantify. There seems to be an implicit \emph{enhanced regularity} effect coming from the transport term: it seems to ``push'' the flow away from the singularity present in the lower order term at $\xi=0$, which aids diffusion to stabilize the solution, ``enhancing'' the classical parabolic regularity. A \emph{heuristically} similar phenomenon is observed in the 1-dimensional inviscid Burger's equation with added symmetries: it is known that solutions to $\partial_tu-u\partial_x u=0$ do not develop any singularity \emph{forward} (but not backward) in time provided the initial data is non-increasing. The same initial data will lead to a solution that blows up forward (but not backward) in time if we evolve it according to $\partial_t u+u\partial_xu=0$. Our aim now is to better understand this phenomenon (in the context of operator $\cd_{t,\xi}$ \eqref{defopd}), in particular, we shall explain that this ``enhanced regularity'' effect is a property that is embedded into the stochastic flow map associated to the operator $\cd_{t,\xi}$. In other words, it is an intrinsic property of the operator $\cd_{t,\xi}$ that extends to solutions that aren't necessarily symmetric. We start by recalling the classical Feynman-Kac formula: if $\Phi$ is the stochastic flow map associated with the transport-diffusion operator $\partial_t-4\nu\partial^2_\xi-b\partial_\xi$:
\begin{equation}\label{c23}
	\Phi(t,\xi)=\xi-\int_0^tb(s,\Phi(s,\xi))ds+\sqrt{8\nu}W(t),
\end{equation}
where $W$ is standard Brownian motion, and if 
\begin{equation}\label{c24}
	\eta(t,\xi):=\exp\left(\int_0^tc(s,\Phi(s,\xi))ds\right),
\end{equation}
then the solution to $\partial_t\cv-4\nu\partial^2_\xi\cv-b(t,\xi)\partial_\xi\cv-c(t,\xi)\cv=0$ with initial data $\cv_0$ is given by 
\[
\cv(t,\xi)=\mathbb{E}\left[\cv_0(\ca_{t,\xi})\eta(t,\ca_{t,\xi})\right],
\]
where we used the notation $\ca(t,\xi)=\ca_{t,\xi}:=\Phi^{-1}(t,\xi)$ for the inverse flow map, and $\mathbb{E}$ corresponds to averaging over the standard Wiener space (integration over the probability space with respect to the standard Wiener measure). Notice that the effect of the lower order term is encoded in $\eta$, so let us see what happens when $c=\partial_\xi b$. Differentiating \eqref{c23}, we see that $\Psi(t,\xi):=\partial_\xi\Phi(t,\xi)$ solves 
\[
\Psi(t,\xi)=1-\int_0^tc(s,\Phi(s,\xi))\Psi(s,\xi)ds\Rightarrow \Psi(t,\xi)=\partial_\xi\Phi(t,\xi)=\exp\left(-\int_0^tc(s,\Phi(s,\xi))ds\right)=\frac{1}{\eta(t,\xi)},
\]
we therefore have (as $\ca_{t,\xi}$ is the inverse flow map),
\begin{equation}\label{c25}
	\eta(t,\ca_{t,\xi})=\frac{1}{\partial_\xi\Phi(t,\Phi^{-1}(t,\xi))}=\partial_\xi\Phi^{-1}(t,\xi)=\partial_\xi\ca(t,\xi),
\end{equation}
and hence $\cv(t,\xi)=\mathbb{E}[\partial_\xi\ca(t,\xi)\cv_0(\ca_{t,\xi})]$. It was shown in \cite[Proposition 2.1.2]{Gautamthesis}, and also  \cite[Proposition 4.2]{CI2008} that $\ca$ solves 
\begin{equation}\label{c26}
	\partial_t\ca-4\nu\partial^2_\xi\ca-b\partial_\xi\ca+\sqrt{8\nu}\dot{W}\partial_\xi\ca=0,
\end{equation}
with odd initial data $\ca(0,\xi)=\xi$. Taking $b(t,\xi):=g(t)h(\xi)$ makes $\bar{\cb}:=\mathbb{E}[\partial_\xi\ca]$ solve $\partial_t\bar{\cb}-\nu\partial^2_\xi\bar{\cb}-\partial_\xi[b\bar{\cb}]=0$ with initial data $\bar{\cb}(0,\xi)=1$: it satisfies (S) in Theorem \ref{secondthm}. In particular, $\|\bar{\cb}(t,\cdot)\|_{L^{\infty}}$ obeys the same right-hand side bound in \eqref{c22}, allowing us to conclude that solutions to \eqref{eqcvheur} with \emph{arbitrary} initial data (not just symmetric ones) are in $L^1_tL_x^{\infty}$ provided $g\in L^1$, via utilizing the Feynman-Kac representation $|\cv(t,\xi)\|\leq \|\cv_0\|_{L^{\infty}}\|\bar{\cb}(t,\cdot)\|_{L^{\infty}}$. Note carefully that \emph{the average} of $\cb$, $\bar{\cb}$, satisfies (S), $\cb$ itself does not due to the It\^o noise. 

On the other hand, if $c\neq\partial_\xi b$ (for instance, if $b=0$), then we no longer have $\eta(t,\ca_{t,\xi})=\partial_\xi\ca(t,\xi)$, and thus the above analysis breaks down. We do not know how to take advantage of the phenomenon in general. If we flip the sign of the transport term (or remove it all together) while keeping the singular lower order term as is in the operator $\cd_{t,\xi}$, that is, if we consider the operator 
\[
\widetilde{\cd}_{t,\xi}:=\partial_t-4\nu\partial^2_\xi+g(t)h(\xi)\partial_\xi-g(t)h'(\xi),
\]
we lose this enhanced regularity effect. Notice that solutions to $\widetilde{D}_{t,\xi}\cv=0$ still preserve the symmetry (S), in which case the transport term no longer has a good sign (in fact, it has a bad sign and amplifies the instabilities from the singular lower order term). We give a counter example in Appendix \ref{appce} (which is essentially due to Elgindi). 

Finally, we remark that diffusion, of course, has to be there, for if not, then the flow map (and its inverse) becomes odd ($\Phi(t,0)=0=\ca(t,0)$), and so for symmetric initial data, one has 
\[
\|\cv(t,\cdot)\|_{L^{\infty}}=\cv(t,0)=\partial_\xi\ca(t,0)\cv_0(0)=\|\cv_0\|_{L^{\infty}}\exp\left(h'(0)\int_0^tg(s)ds\right),
\]
which is bad news since $h'$ is singular at $\xi=0$. Such analysis allows us to treat the Navier-Stokes equations as ``perturbation'' of drift diffusion, as we now explain.
\subsubsection{The pressure term as a perturbation}\label{heur3}
Recall that in order for the solution to the NSE to obey $\Omega$, the latter needs to satisfy \eqref{c21}. An integration by parts on the right-hand side  yields the condition
\begin{equation}\label{c27}
	\partial_t\Omega(t,\xi)-4\nu\partial^2_\xi\Omega(t,\xi)-g(t)\xi^{\beta}\partial_\xi\Omega(t,\xi)\geq C_d g(t)\int_0^{\xi}\partial_\eta\Omega(t,\eta)\eta^{\beta-1}d\eta.
\end{equation}
Since we have $\partial_\xi\Omega\geq0$, we see that for any $\mu_1\geq1$, if $\Omega$ solves 
\[
\partial_t\Omega(t,\xi)-4\nu\partial^2_\xi\Omega(t,\xi)-\mu_1g(t)\xi^{\beta}\partial_\xi\Omega(t,\xi)=C_{d,\beta} g(t)\int_0^{\xi}\partial_\eta\Omega(t,\eta)\eta^{\beta-1}d\eta,\quad C_{d,\beta}=C_d\beta^{-1},
\]
then clearly $\Omega$ satisfies \eqref{c27}. The idea is to choose $\mu_1$ sufficiently large to make sure that the dynamics are dominated by the transport-diffusion part, allowing us to treat the non-local right-hand side coming from the pressure as a perturbation. Of course, the larger $\mu_1$ is, the worse the bound \eqref{c22} becomes (replace $g$ by $\mu_1 g$ in that bound), and so this approach comes at a price. We remark that the drift-term has to be present to begin with in order to take advantage of it: in particular, the associated non-local equation \emph{does not} have a minimum principle, see Appendix \ref{appce}. The hope is to be able to choose $\mu_1$ independent of any critical or subcritical entities to end up with a meaningful bound on $\cv:=\partial_\xi\Omega$, which happens to solve
\[
	\partial_t\cv(t,\xi)-4\nu\partial^2_\xi\cv(t,\xi)-\mu_1g(t)\partial_\xi[h(\xi)\cv(t,\xi)]=C_dg(t)h'(\xi)\cv(t,\xi).
\]
It is clear that even with the added symmetries (S) from Theorem \ref{secondthm} (which again are still preserved by the evolution owing to classical maximum/minimum principle together with the concavity of $h$ on the positive half-line), the adjoint equation in this case still has a singular lower order term that we do not know how to handle. 

Nonetheless, we may still employ the Feynman-Kac formula: if we set $\mu_2:=\mu_1+C_{d,\beta}$, let $c(t,\xi)=\mu_2g(t)h'(\xi)$, $b:=\mu_1g(t)h(\xi)$, let $\Phi$ be the stochastic diffeomorphism that solves \eqref{c23} with drift $b$, let $\ca_{t,\xi}:=\Phi^{-1}(t,\xi)$ be its inverse, and let $\eta$ be as given in \eqref{c23}. It follows that $\cv(t,\xi)=\mathbb{E}[\eta(t,\ca_{t,\xi})\cv_0(\ca_{t,\xi})]$. If we define $\lambda:=\mu_2/\mu_1$ then we have $c=\lambda\partial_\xi b$, so that analogous to \eqref{c25}, we have $\eta(t,\ca_{t,\xi})=[\partial_\xi\ca(t,\xi)]^{\lambda}$. Thus, $\cv(t,\xi)=\mathbb{E}[(\partial_\xi\ca(t,\xi))^{\lambda}\cv_0(\ca_{t,\xi})]$, and so to bound $\|\cv(t,\cdot)\|_{L^{\infty}}$, we need to control the $\lambda$'th moment of $\cb:=\partial_\xi\ca$, with the latter one solving
\begin{equation}\label{c27n}
	\partial_t\cb-4\nu\partial^2_\xi\cb-\mu_1g(t)\partial_\xi\left[h(\xi)\cb(t,\xi)\right]+\sqrt{8\nu}\dot{W}(t)\partial_\xi\cb(t,\xi)=0,
\end{equation}
as $\ca$ solves \eqref{c26}. Analogous to the observation made in our previous work \cite{Ibdah2021} (and assuming the solution to \eqref{c26} is $C^1$ in time, which is \emph{not} the case as Brownian motion is nowhere differentiable), whenever we have a term of the form $F(t,\nabla\theta)$ (for example a transport term with a spatially independent drift), then such a term makes no contribution when propagating moduli of continuity (see Lemma \ref{localmaxprincp} and how its applied in the proof of Theorem \ref{mainthm} in section \ref{pfmainthm} below). This means that, at least theoretically speaking, one could conjecture the bound (assuming $\nu=1$ for simplicity)
\begin{equation}\label{c28}
	\sup_{\xi\in\rone}\mathbb{E}[\cb(t,\xi)^{\lambda}]\leq \left[1+8\mu_1e^{\|g\|_{L^1}}\left[\int_0^t(t-r)^{\frac{\beta-1}{2}}g(r)dr+\mu_1\|g\|_{L^1}\int_0^t(t-r)^{\frac{\beta-2}{2}}g(r)dr\right]\right]^\lambda.
\end{equation}
Such a bound would (theoretically) follow from the bound we have on the fundamental solution \eqref{c22} (with $g$ being replaced with $\mu_1g$) after propagating moduli of continuity to solutions of the transport-diffusion equation $\partial_t-4\nu\partial^2_\xi-\mu_1g(t)h(\xi)\partial_\xi+\sqrt{8\nu}\dot{W}(t)\partial_\xi$. Since $\mu_1\geq1$ can be chosen arbitrarily large, $\lambda=1+\epsilon$, and so assuming \eqref{c28} is valid, we get $L_t^1L_x^{\infty}$ a-priori bounds on $\cv$ in terms of $\|g\|_{L^{1+\epsilon}}$, for any $\epsilon>0$. Unfortunately, this is much easier said than done: we do not know how to prove \eqref{c28} if $\lambda>1$. The issue is that the coefficient of the It\^o term matches exactly that of diffusion: equation \eqref{c26} is essentially a transport equation with Stratonovich multiplicative noise (no diffusion is seen unless we take expectations). Hence if (for example) we try to approximate Brownian motion with $C^1$ processes (in order to be able to propagate moduli of continuity), the Wong-Zakai principle tells us the approximating sequence solves a transport equation with no diffusion. The other option is to use Duhamel's principle in the stochastic equation for $\cb$ \eqref{c27n},
\[
\cb(t,\xi)=\intsol Z(t,\xi;0,\sigma)d\sigma-\sqrt{8\nu}\int_0^t\intsol Z(t,\xi;s,\sigma)\partial_\sigma\cb(s,\sigma)d\sigma dW_s,
\]
to estimate the $\lambda$'th moment of $\cb $. This would require us to use BDG-type inequalities to handle the It\^o integral, and the problem with those is that they require control of the quadratic variation. Such an approach would also require us to integrate by parts and bound a derivative of the fundamental solution $L^1$ in space, which has a singularity of the form $(t-s)^{-1/2}$. This certainly is not in $L^2$, and hence the BDG approach would fail miserably. It is worth pointing out that this problem has been investigated heavily by Flandoli and many of his collaborators over the years, see for instance \cite{BFGM2019, FGP2010} and the references therein. In all of those results, they were more or less able to control moments of the derivative of the flow map under a critical or sub-critical assumption on the drift term. Here, our drift is at a supercritical level, and it is not clear to us how to bound any moment other than the first one in terms of supercritical entities.

Let us now point out that the conjectured bound \eqref{c28} is in line with the idea of ``treating the pressure term perturbatively'': $\lambda=1+\epsilon$ if $\mu_1$ is of order $\epsilon^{-1}$. Ignoring the factor of $\mu_1$ in bound \eqref{c28}, as $\epsilon$ tends to zero, we converge to \eqref{c22}, the drift-diffusion case with no pressure. Our goal is to treat the pressure perturbatively in the sense that the ``structure'' of drift-diffusion (which we know how to handle) becomes dominant. Although we are unable to prove this when measuring $L^{\infty}$ norms, we shall demonstrate that this idea works if we instead control the $L^1$ norm, which is conserved under evolution of the operator $\cd_{t,\xi}$ \eqref{defopd}. This is nothing but another proof of the classical maximum principle, since $\|\Omega(t,\cdot)\|_{L^{\infty}}\leq \|\cv(t,\cdot)\|_{L^1}$, as $\cv=\partial_\xi\Omega$. Such an $L^1$ bound on $\cv$ translates to an $L^{\infty}$ bound on $u$. Indeed, recall that as long as $\Omega$ solves \eqref{c27}, then the solution to NSE satisfies $|u(t,x)-u(t,y)|\leq \Omega(t,|x-y|)$ for any $t,x,y$. Rather than using $\Omega$ to control $\|\nabla u(t,\cdot)\|_{L^{\infty}}$ in terms of $\partial_\xi\Omega(t,0)$, we may use it to control $\|u(t,\cdot)\|_{L^{\infty}}$ in terms of $\|\Omega(t,\cdot)\|_{L^{\infty}}$, which in turn is controlled by $\|\cv(t,\cdot)\|_{L^1}$. This holds true in both the periodic and whole space setting: in the former we can without any loss in generality assume the solution has zero spatial average, while in the latter case we assume the solution vanishes at infinity. From the representation $\cv(t,\xi)=\mathbb{E}[[\cb(t,\xi)]^\lambda\cv_0(\ca_{t,\xi})]$, with $\cb(t,\xi)=\partial_\xi\ca_{t,\xi}$, an application of Fubini-Tonelli followed by the fact that $\ca_{t,\xi}$ is a diffeomorphism, and another Fubini-Tonelli employment we get 
\[
\intsol|\cv(t,\xi)|d\xi\leq \mathbb{E}\left[\intsol\cb(t,\Phi(t,\xi))^{\lambda-1}|\cv_0(\xi)|d\xi\right]\leq\|\cv_0\|_{L^1}\sup_{\xi\in\rone}\mathbb{E}[\cb^{\lambda-1}(t,\xi)].
\]
Recall that $\lambda=\mu_2/\mu_1=1+C_{d,\beta}/\mu_1$, and that $\mu_1\geq1$ can be chosen arbitrarily larger. Thus, $\lambda-1$ can be made arbitrarily small by choosing $\mu_1$ large enough: in particular, $\lambda-1\in[0,1]$. An application of H\"older's inequality renders $\mathbb{E}[\cb^{\lambda-1}(t,\xi)]\leq [\bar{\cb}(t,\xi)]^{\lambda-1}$, where $\bar{\cb}=\mathbb{E}[\cb]$ (since we are working on a probability space, whose measure is one). Averaging \eqref{c27n} over the probability space (taking expectations) annihilates the white noise, making $\bar{\cb}$ a solution to 
\[
\partial_t\bar{\cb}-4\nu\partial^2_\xi\bar{\cb}-\mu_1g(t)\partial_\xi\left[h(\xi)\bar{\cb}(t,\xi)\right]=0,\quad \bar{\cb}(0,\xi)=1,
\]
that satisfies the symmetric properties (S) from Theorem \ref{secondthm}, whence the bound 
\[
\|\bar{\cb}(t,\cdot)\|_{L^{\infty}}\leq1+8\mu_1e^{\|g\|_{L^1}}\left[\int_0^t\nu^{\frac{\beta-1}{2}}(t-r)^{\frac{\beta-1}{2}}g(r)dr+\nu^{\frac{\beta-2}{2}}\mu_1\|g\|_{L^1}\int_0^t(t-r)^{\frac{\beta-2}{2}}g(r)dr\right]=:G(0,t).
\]
From here, for any $\epsilon\in(0,1)$, if we choose $\mu_1=
C_{d,\beta}/\epsilon\geq C_{d,\beta}\geq1$, we have $\lambda-1=\epsilon$ and $\|\cv(t,\cdot)\|_{L^1}\leq\|\cv_0\|_{L^1}[G(0,t)]^{\epsilon}$. In particular, for any $q\in[1,\infty)$, choosing $\epsilon=1/q$ we get the bound 
\[
\int_0^T\|u(t,\cdot)\|_{L^{\infty}}^qdt\leq \|\cv_0\|_{L^1}^q\int_0^TG(0,t)dt\lesssim M,
\]
where $M$ is a constant depending only on initial data, $q$, and $\|g\|_{L^1}$. Regularity then follows by Ladyzhenskaja-Prodi-Serrin by making sure $q>2$, as demonstrated in Appendix \ref{applps}. This is also inline with the idea of treating the pressure perturbatively. To better demonstrate this, let us for the moment ignore the factor of $1/\epsilon$ that shows up in $G(0,t)$ due to the drift being of order $1/\epsilon$ (of course we \emph{cannot} ignore it and it will show up in the final estimate). As $\epsilon\rightarrow0^+$, we get $\|\cv(t,\cdot)\|_{L^1}\leq \|\cv_0\|_{L^1}$, which certainly is the case for the operator $\cd_{t,\xi}$ \eqref{defopd}. The pressure is treated perturbatively in the sense that the dynamics of classical transport-diffusion dominates, allowing us to extend the supercritical regularity to the incompressible NSE.
\subsection{Preliminaries}\label{secprel}
As indicated earlier, Lemmas \ref{gradbd} and \ref{localmaxprincp} below are ``corollaries'' of Lemmas 3.2 and 3.3 respectively in our previous work \cite{Ibdah2022}, which in turn are based on the arguments in \cite{Kiselev2011, KNS2008, KNV2007}. Theorems \ref{sdiffexis} and \ref{classicalthm} contain classical results which can be found in, for instance, \cite{KunitaNotes1984,Kunitabook1997} and \cite{Friedmanbook1964,IKO1962,LSUbook1968}. We only justify some of the niche claims.
\begin{lem}\label{gradbd}
Suppose a vector field $u\in C^2(\mathbb{R}^d)\cap W^{2,\infty}(\rd)$ has a Lipschitz modulus of continuity $\omega$. It then follows that $\|\nabla u\|_{L^{\infty}}<\omega'(0)$. Here, for a vector-field $u$, the Lipschitz constant is defined as 
\[
\|\nabla u\|_{L^\infty}:=\sup_{x\in \rd}\sup_{\substack {e\in \rd\\ |e|=1}}\left|J_u(x)e\right|,
\]
where $J_u(x)$ is the Jacobian matrix of $u$, the matrix whose row vectors are $\nabla u_j$, evaluated at a point $x\in\rd$.
\end{lem}

\begin{lem}\label{localmaxprincp}
	Suppose $\theta$ is a $C^2(\mathbb{R}^d)$ scalar and has modulus of continuity $\omega$. If $\theta(x_0)-\theta(y_0)=\omega(|x_0-y_0|)$ for some $x_0\neq y_0$ with $x_0-y_0=\xi e_1$, where $e_1$ is the unit vector in the direction of $x_1$ and $\xi>0$, then 
\begin{equation}\label{derv}
\begin{cases}
\partial_1\theta(x_0)=\partial_1\theta(y_0)=\omega'(\xi),\\
\partial_j\theta(x_0)=\partial_j\theta(y_0)=0, \quad j\neq 1,
\end{cases}
\end{equation}
and 
\begin{equation}\label{localdissp}
\Delta\theta(x_0)-\Delta\theta(y_0)\leq 4\omega''(\xi).
\end{equation}
\end{lem}
\begin{defs}\label{defsdiff}
	Let $(\cq,\cf,\cp)$ be a standard Wiener space equipped with the standard Wiener measure. Here, the probability space $\cq$ is the space of $\rone$ valued continuous functions on $[T_0,T_1]$, for some $0\leq T_0<T_1$. Let $W:[T_0,T_1]\times\cq\rightarrow\rone$ be a standard, one-dimensional Wiener process on the Wiener space, and let $\{\cf_t\}$ be the filtration associated to it. A mapping $\Phi:[T_0,T_1]\times\rone\times\cq\rightarrow\rone$ is called a stochastic diffeomorphism if for every fixed $\xi$, $\Phi$ is a continuous $\{\cf_t\}$ adapted process and for every $t\in[T_0,T_1]$, $\Phi(t,\cdot;q):\rone\rightarrow\rone$ is a diffeomorphism (almost surely in $q$).
\end{defs} 
	
\begin{thm}\label{sdiffexis}
	Assume the probabilistic setting in Definition \ref{defsdiff}, let $b:[T_0,T_1]\times\rone\rightarrow\rone$ be a smooth (twice continuously differentiable), deterministic, bounded function (with bounded derivatives), and suppose $\nu>0$ is a given constant. Then there exists a unique stochastic diffeomorphism $\Phi$ that satisfies the SDE
	\begin{equation}\label{defsde}
	\partial_t\Phi=-b(t,\Phi)+\sqrt{8\nu}\dot{W},\quad \Phi(T_0,\xi;q)=\xi,
	\end{equation}
	in the sense that 
	\begin{equation}\label{ingdefsde}
		\Phi(t,\xi;q)=\xi-\int_{T_0}^tb(s,\Phi(s,\xi;q))ds+\sqrt{8\nu}W(t;q)
	\end{equation}
	holds true for every $(t,\xi)\in[T_0,T_1]\times\rone$ and almost surely in $q$. Moreover, for every fixed $t$, and almost surely in $q$, $\Phi$ is twice continuously differentiable in $\xi$ and the inverse flow map, $\ca:=\Phi^{-1}$, satisfies the stochastic PDE
	\begin{equation}\label{eqinvflw}
		\partial_t\ca(t,\xi;q)-4\nu\partial^2_\xi\ca(t,\xi;q)-b(t,\xi)\partial_\xi\ca(t,\xi;q)+\sqrt{8\nu}\dot{W}(t;q)\partial_\xi\ca(t,\xi;q)=0,\quad \ca(0,\xi;q)=\xi,
	\end{equation}
	in the It\^o sense.
\end{thm}
\begin{rmk}\label{impasrmk}
		All of the ``almost sure'' statements above are ``independent'' of the variables $(t,\xi)$: for instance, when we say $\Phi$ satisfies \eqref{ingdefsde} almost surely in $q$, we mean that there exists an event $B\subset \cf$, with $\cp[B]=1$ such that \eqref{ingdefsde} holds true for every $(t,\xi;q)\in[T_0,T_1]\times\rone\times B$. The same applies to all other almost sure statements: one event $B$ works for every $(t,\xi)$.
	\end{rmk}
	\begin{proof}
	The existence, uniqueness, and regularity of the stochastic flow map is classical, see for instance \cite{KunitaNotes1984}. Remark \ref{impasrmk} is explicitly addressed in \cite[Section 7, Chapter 1]{KunitaNotes1984} and then again in the appendix of Chapter 1 by means of a theorem of Kolmogorov. Similar results in a much more general setting (for which ours is a special case) were obtained in \cite[Chapter 3]{Kunitabook1997}. As for the statement regarding the inverse flow map, it was obtained in \cite[Proposition 2.1.2]{Gautamthesis} (see also \cite[Proposition 4.2]{CI2008}).
	\end{proof}
\begin{thm}\label{classicalthm}
	Let $T_0<T_1$, $\nu>0$, and let $b,c,d:[T_0,T_1]\times\rone\rightarrow\rone$, be smooth and bounded. Let $\mathcal{L}_{t,\xi}$ be the parabolic operator $\mathcal{L}_{t,\xi}:=\partial_t-4\nu\partial^2_\xi-b(t,\xi)\partial_\xi-c(t,\xi)$, and let $u_0$ be smooth. Let $u$ solve $\mathcal{L}_{t,\xi}u=d$, with $u(T_0,\xi)=u_0(\xi)$, and let $\mathcal{L}^*_{s,\sigma}:=-\partial_s-4\nu\partial^2_\sigma+\partial_\sigma[b(s,\sigma)\cdot]-c(s,\sigma)$ be the adjoint operator. Then the following statements hold true:
	\begin{enumerate}
	\item $u\geq0$ ($u\leq0$) if $u_0,d\geq0$ ($u_0,d\leq0$).
	\item Let $\Psi$ be the heat kernel (corresponding to $\partial_t-4\nu\partial^2_\xi$). Then there exists a $Q$ such that if \[
	\tilde{Z}(t,\xi;s,\sigma):=\int_s^t\intsol \Psi(t-r,\xi-\mu)Q(r,\mu;s,\sigma)d\mu dr,\quad T_0\leq s<t\leq T_1,\ \xi,\sigma\in\rone,
	\]
	then $Z(t,\xi;s,\sigma):=\Psi(t-s,\xi-\sigma)+\tilde{Z}(t,\xi;s,\sigma)$ solves $\mathcal{L}_{t,\xi}Z=0$ for every fixed $(s,\sigma)\in[T_0,t)\times\rone$, while $\mathcal{L}^*_{s,\sigma}Z=0$ for every fixed $(t,\xi)\in(s,T_1]\times\rone$. Moreover, $Z\geq0$, there exists a constant $c_B$ depending on $\nu$, $b$, and $c$ such that $|Q(r,\mu;s,\sigma)|\leq c_B(r-s)^{-1}\exp(-c_B|\mu-\sigma|^2(r-s)^{-1})$, and for any given smooth $u_0$, if we define
	\[
	u(t,\xi):=\intsol Z(t,\xi;T_0,\sigma)u_0(\sigma)+\int_{T_0}^t\intsol Z(t,\xi;s,\sigma)d(s,\sigma)d\sigma ds,
	\]
	then $u$ solves $\mathcal{L}_{t,\xi}u=d$, with initial data $u(T_0,\xi)=u_0(\xi)$. 
	\item Suppose that $b(t,\cdot)$ is odd about $\xi=0$ and $c(t,\cdot)$ is even about $\xi=0$. If $u_0$ and $d(t,\cdot)$ are both even (odd), it follows that $u(t,\cdot)$ is also even (odd). In the case where $u_0$, $d(t,\cdot)$, and $c(t,\cdot)$ are all even, if we further assume that $u_0$ and $d$ both satisfy (S) from Theorem \ref{secondthm}, while $\partial_\xi c(t,\cdot)\leq 0$ on $(0,\infty)$, then $u(t,\cdot)$ also satisfies (S) for every $t\in[0,T]$.
	\item Let $\Phi$ be the stochastic flow map alluded to in Theorem \ref{sdiffexis}, let $\ca:=\Phi^{-1}$ be its inverse, and let us denote $\ca_{t,\xi}=\ca(t,\xi;q)$. Set 
	\[
	\eta(t,\xi):=\exp\left(\int_{T_0}^tc(r,\Phi(r,\xi))dr\right),
	\]
	and let $u$ solve $\mathcal{L}_{t,\xi}u=d$ with $u(T_0,\xi)=u_0(\xi)$. It follow that we have the representation 
	\[
	u(t,\xi)=\mathbb{E}\left[\eta(t,\ca_{t,\xi})u_0(\ca_{t,\xi})\right]+\mathbb{E}\left[\eta(t,\ca_{t,\xi})\int_{T_0}^t\frac{d(s,\Phi(s,\ca_{t,\xi}))}{\eta(s,\ca_{t,\xi})}ds\right],
	\]
	where $\mathbb{E}$ denotes averaging over the probability space (taking expectations).
	\end{enumerate}
\end{thm}
\begin{proof}
Item (1) follows from the standard minimum/maximum principle (since the lower order term $c$ is assumed to be bounded), i.e., apply the minimum/maximum principle to $\tilde{u}(t,\xi):=u(t,\xi)\exp(-\int_{T_0}^t\|c(s,\cdot)\|_{L^{\infty}}ds)$. Item (2) is a classical result, see for instance \cite{Friedmanbook1964, IKO1962, LSUbook1968}. As for item (3), setting 
\begin{align*}
&\varphi(t,\xi):=\intsol \Psi(t-T_0,\xi-\sigma)u_0(\sigma)d\sigma+\int_{T_0}^t\intsol\Psi(t-s,\xi-\sigma)d(s,\sigma)d\sigma ds,\\
&K(t,\xi;s,\sigma):=\partial_\sigma[b(s,\sigma)\Psi(t-s,\xi-\sigma)]+c(s,\sigma)\Psi(t-s,\xi-\sigma),
\end{align*}
we see that by Duhamel's principle and an integration by parts, $u$ solves the Volterra integral equation
\[
u(t,\xi)=\varphi(t,\xi)+\int_{T_0}^t\intsol K(t,\xi;s,\sigma)u(s,\sigma)d\sigma ds
\]
if and only if $u$ solves $\mathcal{L}_{t,\xi}u=d$ and $u(0,\xi)=u_0(\xi)$. As the kernel has a weak (integrable) singularity, this always has a unique solution \cite{Pogorzelskibook1966}, and hence must be the solution to the initial-value problem $\mathcal{L}_{t,\xi}u=d$, $u(0,\xi)=u_0(\xi)$. It is clear that $\varphi(t,\cdot)$ is even (odd) if $u_0$, $d$ are both even (odd), and so the fact that the solution inherits the symmetries of the forcing term and the initial data follows from the symmetry of the heat kernel and the coefficients $b$ and $c$ (symmetry of the kernel $K$, and by extension, symmetry of the resolvent). To prove the second statement of item (3), first notice that since $u_0$ and $d$ are both non-negative, so is $u$. Let us now set $v:=\partial_\xi u$ and observe that since $u(t,\cdot)$ is even and smooth, $v(t,\cdot)$ is odd and smooth, i.e., $v(t,0)=0$. It follows that $v$ solves
\begin{align*}
	&\partial_t v-4\nu\partial^2_\xi v-\partial_\xi[bv]-cv=u\partial_\xi c+\partial_\xi d,\quad \forall (t,\xi)\in(0,T]\times[0,\infty),\\
	&v(t,0)=0,\quad \forall t\in[0,T],\\
	&v(0,\xi)=u_0'(\xi),\quad \forall \xi\in[0,\infty).
\end{align*}
Since $u\geq0$, $\partial_\xi c\leq0$, $\partial_\xi d\leq0$, and $u_0'\leq0$ on $[0,T]\times[0,\infty)$, the minimum principle tells us that $v\leq0$ on $[0,T]\times[0,\infty)$. 

Item (4) is nothing but the classical Feynman-Kac formula, see for instance \cite{Kunitabook1997}. The reader can also readily verify the representation by (a very careful) direct differentiation, keeping in mind It\^o's Lemma and it's product rule: at the end of the day, loosely speaking this is nothing but the method of characteristics adapted to diffusion.
\end{proof}
\section{Constructing the modulus of continuity}\label{mainsec}
The main result in this section is the following.
\begin{thm}\label{thmbuildmod}
Let $T>0$, $\beta\in(0,1)$, $\mu_1>0$, $\nu>0$, and $C\geq0$ be arbitrary given constants, and let $g:[0,T]\rightarrow[0,\infty)$ be a given continuous function. For $0\leq s<t\leq T$, define
\begin{align}
&\Gamma(s,t):=\int_s^tg(r)dr,\label{defgamma2}\\
&G(s,t):=1+8\mu_1\nu^{\frac{\beta-2}{2}}e^{2\mu_1\Gamma(s,t)}\left[\nu^{1/2}\int_s^t(t-r)^{\frac{\beta-1}{2}}g(r)dr+\mu_1\Gamma(s,t)\int_s^t(t-r)^{\frac{\beta-2}{2}}g(r)dr\right].\label{defG2}
\end{align}
Suppose that $a:[0,T]\times[0,\infty)\rightarrow\rone$ is smooth, and $a(t,0)=0$ for every $t\in[0,T]$. Let $\Omega_0:[0,\infty)\rightarrow\mathbb{R}$ be a given smooth function that satisfies $\Omega_0(0)=0$. Then there exists an $\Omega\in C_t^1C_{\xi}^2((0,T]\times[0,\infty))$ with $\partial_\xi\Omega\in C([0,T]\times[0,\infty))$ such that $\Omega(t,\cdot)$ solves
\begin{equation}\label{c31}
	\partial_t\Omega(t,\xi)-4\nu\partial^2_\xi\Omega(t,\xi)-\mu_1g(t)\xi^{\beta}\partial_\xi\Omega(t,\xi)=Cg(t)\int_0^\xi\eta^{\beta-1}\partial_\eta\Omega(t,\eta)d\eta+a(t,\xi)
\end{equation}
on $(t,\xi)\in(0,T]\times(0,\infty)$, and satisfies the initial-boundary conditions
\begin{align}
	&\Omega(t,0)=0,\quad \forall t\in[0,T],\label{BCcondomega}\\
	&\Omega(0,\xi)=\Omega_0(\xi),\quad \forall \xi\in[0,\infty)\label{ICcondomega}.
\end{align}
Moreover if $\mu_1>C\beta^{-1}$, then we have the a-priori bound
\begin{equation}\label{maindbomega}
	\|\Omega(t,\cdot)\|_{L^{\infty}}\leq\|\partial_\xi\Omega(t,\cdot)\|_{L^{1}}\leq\left[\|\Omega_0'\|_{L^1}+\int_0^t\|\partial_\xi a(s,\cdot)\|_{L_1}ds\right]G^{\lambda-1}(0,t),\quad \lambda:=1+\frac{C}{\beta\mu_1}.
\end{equation}
If $C=0$, one has the stronger control
\begin{equation}\label{bddd}
	\|\partial_\xi\Omega(t,\cdot)\|_{L^{\infty}}\leq \left[\|\Omega_0'\|_{L^{\infty}}+\int_0^t\|\partial_\xi a(s,\cdot)\|_{L^{\infty}}ds\right]G(0,t).
\end{equation}
If we assume that $a(t,\cdot)$ is non-decreasing and concave (for every $t\in[0,T]$), and if $\|a(t,\cdot)\|_{L^{\infty}}\leq f(t)$, then bound \eqref{bddd} (assuming $C=0$) can be sharpened to read
\begin{equation}\label{bddd2}
	\|\partial_\xi\Omega(t,\cdot)\|_{L^{\infty}}\leq \|\Omega_0'\|_{L^{\infty}}G(0,t)+\int_0^t\left[\nu^{-\frac{1}{2}}(t-s)^{-\frac{1}{2}}+\mu_1\nu^{\frac{\beta-2}{2}}\int_s^t(t-r)^{\frac{\beta-2}{2}}g(r)dr\right]f(s)ds.
\end{equation}
Finally, if $\Omega_0$ and $a(t,\cdot)$ are both non-decreasing and concave (for every fixed $t$), then so is the solution to \eqref{c31} for any $C\geq0$.
\end{thm}
\begin{rmk}\label{rmkforce}
The forcing term we have in mind (for our applications) is one of the type $a(t,\xi)=f(t)\chi(\xi)$ where given $\alpha\in(0,1)$, $\chi(\xi)=\xi^{\alpha}$, for $\xi\in[0,1]$, and constant for $\xi\geq1$, extended oddly about zero. If necessary, one could smoothen out the singularity in its derivative by standard mollification and take advantage of the extra symmetries involved as described in the compactness argument (see \S\ref{pfthmbuildmod}, below). The point is that bounds \eqref{maindbomega} and \eqref{bddd2} are independent of any regularity assumptions on the derivative of $a$: they only depend on $\|a\|_{L_t^1L_x^{\infty}}$ when considering such forcing terms, since $\partial_\xi a(t,\cdot)$ satisfies (S) in Theorem \ref{secondthm}: in particular, $\partial_\xi a(t,\cdot)\geq0$ and hence $\|\partial_\xi a(t,\cdot)\|_{L^1}\leq \|a(t,\cdot)\|_{L^\infty}$. In fact, they are even independent of the exponent $\alpha\in(0,1)$. Estimate \eqref{bddd} on the other hand requires $\alpha=1$, but for our applications, we are only interested in symmetric forces, so we will always be using \eqref{bddd2} when needed.
\end{rmk}

\begin{rmk}
Theorem \ref{thmbuildmod} does not tell the full story: one can obtain bounds on $\|\partial_\xi\Omega(t,\cdot)\|_{L^{p}}$ in terms of powers of $G$ even when $C>0$ for certain values of $p$. See Appendix \ref{applpbds} for more details.
\end{rmk}
\begin{rmk}\label{symmrmk}
As explained below, we construct a solution to \eqref{c31} as an antiderivative of some $\cv$ that solves the differentiated equation. Thus, the fact that $\Omega(t,\cdot)$ is non-decreasing and concave follows from assuming $\Omega_0'$ and $\partial_\xi a(t,\cdot)$ satisfy (S) from Theorem \ref{secondthm}. It is also worth noting that even if $\Omega_0$ is \emph{convex} (rather than concave), then the solution \emph{will not preserve this property} (not unless we assume $\Omega_0$ is decreasing and hence negative), due to the concavity of $\xi^{\beta}$. We have no interest in negative solutions in this work, as $\Omega$ will be a modulus of continuity.
\end{rmk}
\begin{rmk}
	We also have uniform bounds on the first and second derivatives of $\Omega$ (up to the boundary). Those are in terms of sub-critical quantities (as $\|g\|_{L^p}$ with $p>2/(\beta+1)$). Those bounds are not used to control any norm of the solution to \eqref{NSE}, but are required to rigorously derive the a-priori bounds in Theorem \ref{mainthm}.
\end{rmk}

The first step is to get rid of the non-local term appearing in \eqref{c31}: let $h_0:\rone\rightarrow\rone$ ($a_0:[0,T]\times\rone\rightarrow\rone$) be the odd extension of $\xi^\beta$ ($a(t,\cdot)$) about $\xi=0$, and consider
\[
\Omega(t,\xi):=\int_0^{\xi}\cv(t,\eta)d\eta,\quad (t,\xi)\in[0,T]\times[0,\infty),
\]
where $\cv$ is an \emph{even} solution to
\begin{align}
&\partial_t\cv(t,\xi)-4\nu\partial^2_\xi\cv(t,\xi)-\mu_1g(t)\partial_\xi[h_0(\xi)\cv(t,\xi)]-\frac{C}{\beta}g(t)h_0'(\xi)\cv(t,\xi)=\partial_\xi a_0(t,\xi),\quad (t,\xi)\in(0,T]\times\rone,\label{c32}\\
&\cv(0,\xi)=\cv_0(\xi)=\Omega_0'(\xi)\nonumber.
\end{align}
To construct such a $\cv$, we approximate $h_0$ by a sequence of smooth, bounded, non-decreasing functions indexed by $\epsilon\in(0,1/80)$, $h_\epsilon$, and we obtain uniform in epsilon bounds on the sequence of approximating solutions $\cv_\epsilon$. It is also important for each $h_\epsilon$ to preserve the concavity of $h_0$ on $[0,\infty)$. To do so, for any given $\epsilon<1/80$, we define a new function
\[
h_{\epsilon}(\xi):=
\begin{cases}
	h_0(\xi),&|\xi|\leq\epsilon^{-1},\\
	s_0(\xi), &\epsilon^{-1}\leq|\xi|\leq 2\epsilon^{-1},\\
	2^\beta\epsilon^{-\beta}, &|\xi|\geq 2\epsilon^{-1}.\\
\end{cases}
\]
where $s_0$ is a smooth, bounded, non-decreasing function that preserves H\"older continuity of $h_0(\xi)$ and is concave (convex) when $\xi\geq0$ ($\xi<0$). Next, we let $\chi_{\epsilon}$ be a standard even mollifier at the level $\epsilon$, and we abuse notation by denoting the mollification of $h_{\epsilon}$ by $h_{\epsilon}$ as well. We seek to construct a sequence of solutions $\cv_{\epsilon}$ to \eqref{c32} (with $h_0$ being replaced by $h_\epsilon$), that satisfies \eqref{maindbomega} (and \eqref{bddd}-\eqref{bddd2} when $C=0$) uniformly in $\epsilon$, before passing to the zero $\epsilon$ limit in \S\ref{pfthmbuildmod}. Of course, the presence of a singularity in $h_0'$ at $\xi=0$ presents a major challenge in obtaining estimates, which really is the heart of the matter. In particular, the only uniform in $\epsilon$ bounds we are allowed to impose on $h_\epsilon$ (and $h_\epsilon'$) are: 
\begin{align}
&|h_\epsilon(\xi)|\leq |\xi|^{\beta},\quad \forall \xi\in\rone,\label{ptwsbdh}\\
&|h_\epsilon'(\xi)|=h_\epsilon'(\xi)\leq 2\beta|\xi|^{\beta-1},\quad \forall |\xi|\geq2\epsilon>0,\label{ptwsbdhprime}\\
&\int_{-\xi}^{\xi}|h_\epsilon'(\sigma)|d\sigma=\int_{-\xi}^{\xi}h_\epsilon'(\sigma)d\sigma\leq 2|\xi|^\beta,\quad \forall \xi\in\rone\label{ptwsbdhprime2}
\end{align}
Furthermore, we point out that $h_\epsilon$ inherits the concavity of $\xi^{\beta}$ on $[0,\infty)$, that is 
\begin{equation}\label{concavheps}
h_\epsilon''(\xi)\leq 0,\quad \forall \xi\geq0.	
\end{equation}
This can be proven by directly differentiating the mollified function and using the symmetric properties of the mollifier and $h_0$: the singularity in $h_0'$ at $\xi=0$ is integrable. That being said, the following is the main Proposition used in proving Theorem \ref{thmbuildmod}.
\begin{prop}\label{keyprop}
Let $\cv_0:\rone\rightarrow\rone$ be a given smooth function. Let $T>0$, $\mu_1\in(0,\infty)$, $\mu_2\in[0,\infty)$ be given constants, let $g:[0,T]\rightarrow[0,\infty)$ be continuous, and let $d:[0,T]\times\rone\rightarrow\rone$ be smooth. Then there exists a sequence of smooth functions indexed by $\epsilon>0$, $\{\cv_{\epsilon}\}_{\epsilon>0}$, such that $\cv_{\epsilon}(0,\xi)=\cv_0(\xi)$ and for which
\begin{equation}\label{c33}
	\partial_t\cv_\epsilon(t,\xi)-4\nu\partial^2_\xi\cv_{\epsilon}(t,\xi)-\mu_1g(t)h_{\epsilon}(\xi)\partial_\xi\cv_{\epsilon}(t,\xi)=\mu_2g(t)h_{\epsilon}'(\xi)\cv_{\epsilon}(t,\xi)+d(t,\xi),
\end{equation}
holds true for every $(t,\xi)\in(0,T]\times\rone$. Let us set $\lambda:=\mu_2/\mu_1$ and suppose that $\lambda\in[0,2]$. Then the following a-priori bounds hold true for every $t\in[0,T]$:
\begin{align}
&\|\cv_\epsilon(t,\cdot)\|_{L^{\infty}}\leq \left[\|\cv_0\|_{L^{\infty}}+\int_0^t\|d(s,\cdot)\|_{L^\infty}ds\right]G^{\lambda}(0,t),\quad \lambda\in[0,1],\label{linftybd}\\
&\|\cv_\epsilon(t,\cdot)\|_{L^1}\leq \left[\|\cv_0\|_{L^1}+\int_0^t\|d(s,\cdot)\|_{L^1}ds\right]G^{\lambda-1}(0,t),\quad \lambda\in[1,2],\label{l1bd}
\end{align}
where $G$ is as defined in \eqref{defG2}. If for every $t\in[0,T]$, $d(t,\cdot)$ satisfies (S) from Theorem \ref{secondthm}, then one can improve the bound \eqref{linftybd} when $\lambda=1$ to 
\begin{equation}\label{implinftybd}
	\|\cv_\epsilon(t,\cdot)\|_{L^{\infty}}\leq \|\cv_0\|_{L^{\infty}}G(0,t)+\int_0^t\nu^{-\frac{1}{2}}\left[(t-s)^{-\frac{1}{2}}+\mu_1\nu^{\frac{\beta-2}{2}}\int_s^t(t-r)^{\frac{\beta-2}{2}}g(r)dr\right]\|d(s,\cdot)\|_{L^1}ds.
\end{equation}
If both of $\cv_0$ and $d(t,\cdot)$ (for every $t\in[0,T]$) satisfy (S) from Theorem \ref{secondthm}, then so does $\cv_{\epsilon}(t,\cdot)$ (regardless of the value of $\lambda$). In particular, $\cv_{\epsilon}(t,\cdot)$ is even, non-negative, and non-increasing on $(0,\infty)$ (so that by symmetry, it is maximized at $\xi=0$ for any $t\in[0,T]$).
\end{prop}
\begin{rmk}
Again, bounds \eqref{linftybd}-\eqref{implinftybd} aren't the only ones available. See Appendix \ref{applpbds} for more comprehensive estimates.	
\end{rmk}
\begin{rmk}
In line with Remark	\ref{symmrmk} (and assuming $d=0$ for simplicity), $\cv_\epsilon$ will satisfy (S) from Theorem $\ref{secondthm}$ provided initially it does. A typical example we have in mind is $\cv_0=e^{-\xi^2}$. On the other hand, if $\cv_0$ is increasing on $(0,\infty)$ and hence decreasing on $(-\infty,0)$, then this property cannot be preserved unless $\cv_0\leq0$ (so that the solution is as well). This is due to the (strict) concavity of $h_\epsilon$ on $(0,\infty)$.
\end{rmk}

\subsection{Enhanced Regularity: the special case $\lambda=1$ and symmetric, non-negative solutions}\label{secspcase}
In this section, we prove the following special case of Proposition \ref{keyprop}: the special case when $\lambda=1$ ($\mu_2=\mu_1$), $d=0$, and when the initial data satisfies (S) from Theorem \ref{secondthm}: for example, constant initial data. Moving on, we drop the subscript $\epsilon$ to avoid cumbersome notation, in which case \eqref{c33} (with $\mu_2=\mu_1$ and $d=0$) reduces to 
\begin{equation}\label{eqvspcase}
	\partial_t\cv(t,\xi)-4\nu\partial^2_\xi\cv(t,\xi)-\mu_1g(t)\partial_\xi[h(\xi)\cv(t,\xi)]=0,\quad \forall(t,\xi)\in(0,T]\times\rone.
\end{equation}

\begin{lem}\label{lemspcase}
	Let $\cv_0:\rone\rightarrow[0,\infty)$ and $T>0$ be given. Suppose that $\cv_0$ satisfies condition (S) in Theorem \ref{secondthm}. Let $g:[0,T]\rightarrow[0,\infty)$ be continuous, let $\mu_1\in[0,\infty)$ be a given constant, and let $G$ be as defined in \eqref{defG2}. Then there exists a unique solution to \eqref{eqvspcase} satisfying (S) from Theorem \ref{secondthm} such that $\cv(0,\xi)=\cv_0(\xi)$ and for which the following a-priori bound holds true for every $t\in[0,T]$:
\begin{equation}\label{splinftybd}
	\|\cv(t,\cdot)\|_{L^{\infty}}\leq \|\cv_0\|_{L^{\infty}}G(0,t).
\end{equation}
\end{lem}
Before proving Lemma \ref{lemspcase}, let us start by explaining the strategy, which revolves around bounding the fundamental solution to the parabolic operator
\begin{equation}\label{parabop}
	\cd_{t,\xi}:=\partial_t-4\nu\partial^2_\xi-\mu_1g(t)\partial_\xi[h(\xi)\cdot], \quad (t,\xi)\in[0,T]\times\rone.
\end{equation}
The key idea used in obtaining bound \eqref{splinftybd} is exploiting the fact that the adjoint operator $\cd_{s,\sigma}^*$, given by 
\[
\cd_{s,\sigma}^*:=-\partial_s-4\nu\partial^2_\sigma+\mu_2g(s)h(\sigma)\partial_\sigma,\quad (s,\sigma)\in(0,t)\times\rone,
\]
has no singular lower order terms. This observation can be taken advantage of as follows: item 2 from Theorem \ref{classicalthm} tells us that the operator $\cd_{t,\xi}$ has a fundamental solution $Z(t,\xi;s,\sigma)$, where $0\leq s<t\leq T$ and $(\xi,\sigma)\in\rone\times\rone$, such that for any smooth enough $\cv_0$, if we define 
\[
\cv(t,\xi):=\intsol Z(t,\xi;s,\sigma)\cv_0(\sigma)d\sigma,\quad (t,\xi)\in(s,T]\times\rone,
\]
then $\cv$ solves \eqref{eqvspcase} with initial data $\cv(s,\xi)=\cv_0$. Since we assumed that the initial data satisfies (S) from Theorem \ref{secondthm}, the classical parabolic maximum/minimum principle (item 3 from Theorem \ref{classicalthm}) tells us that the solution, $\cv(t,\cdot)$, does as well for every $t\in[0,T]$. Indeed, $h$ is a non-decreasing, odd (about $\xi=0$) function that is concave on $[0,\infty)$. In particular, we have 
\[
\|\cv(t,\cdot)\|_{L^{\infty}}=\cv(t,0)\leq \|\cv_0\|_{L^{\infty}}\intsol Z(t,0;0,\sigma)d\sigma,
\]
and so we are interested in obtaining bounds on $\|Z(t,0;s,\cdot)\|_{L^1}$. This is where the adjoint operator comes into play: for every fixed $(t,\xi)$, the fundamental solution satisfies the adjoint equation (backwards in time), i.e.
\begin{equation*}
	-\partial_rZ(t,\xi;r,\sigma)-4\nu\partial^2_\sigma Z(t,\xi;r,\sigma)+\mu_1g(r)h(\sigma)\partial_\sigma Z(t,\xi;r,\sigma)=0,\quad (r,\sigma)\in(s,t)\times(-\infty,\infty).
\end{equation*}
Notice (again) that the singular term $h'$ disappears and we have nothing but a transport-diffusion equation: theoretically, it should have a maximum principle. However, at the initial (backwards) time $r=t$, the fundamental solution is a delta function (a distribution), so it does not make sense to talk about maximum principles. Nevertheless, this singularity comes from the heat kernel:
\begin{equation}\label{hk}
	\Psi(r,y):=\frac{1}{\sqrt{16\pi\nu r}}\exp\left(\frac{-y^2}{16\nu r}\right),\quad (r,y)\in(0,\infty)\times\rone,
\end{equation}
which we can bound in any norm independent of the drift. As the fundamental solution takes the form $Z(t,\xi;s,\sigma)=\Psi(t-s,\xi-\sigma)+\tilde{Z}(t,\xi;s,\sigma)$, we only need to control $\tilde{Z}$, which is a volume potential with density $Q$: 
\begin{equation}\label{FSdensity}
	\tilde{Z}(t,\xi;s,\sigma)=\int_s^t\intsol \Psi(t-r,\xi-\mu)Q(r,\mu;s,\sigma)d\mu dr,\quad |Q(r,\mu;s,\sigma)|\leq \frac{c_B}{r-s} \exp\left(-\frac{c_B|\mu-\sigma|^2}{r-s}\right).
\end{equation}
The constant $c_B$ in \eqref{FSdensity} depends on quantities like $\|g\|_{L^{\infty}}$ and $\|h'\|_{L^{\infty}}$ meaning we cannot use it to get useful bounds on $\tilde{Z}$. However, one can use \eqref{FSdensity} to show that $\|\tilde{Z}(t,\xi;t^-\cdot)\|_{L^p}=0$. From $\cd_{s,\sigma}^*Z=0$, we see that $\tilde{Z}$ solves a forced transport-diffusion equation with homogenous data and non-homogenous forcing:
\begin{align}
	&-\partial_r\tilde{Z}(t,\xi;r,\sigma)-4\nu\partial^2_\sigma \tilde{Z
}(t,\xi;r,\sigma)+\mu_1g(r)h(\sigma)\partial_\sigma \tilde{Z}(t,\xi;r,\sigma)=-\mu_1g(r)h(\sigma)\partial_\sigma\Psi(t-r,\xi-\sigma),\nonumber\\
	&\lim_{r\rightarrow t^-}\intsol |\tilde{Z}(t,\xi;r,\sigma)|^pd\sigma=0,\quad \forall p\in[1,\infty),\ \xi\in\rone.\label{ncpf3.2}
\end{align}
Evaluating at $\xi=0$, we observe that H\"older continuity of $h$ \eqref{ptwsbdh} allows us to reduce the intensity of the singularity in $\|\partial_\sigma\Psi(t-r,\cdot)\|_{L^{\infty}}$ (at $t=r$), leading to a meaningful $L^{\infty}$ bound:
\[
\|\tilde{Z}(t,0;s,\cdot)\|_{L^\infty}\leq \mu_1\nu^{\frac{\beta-2}{2}}\int_s^t(t-r)^{\frac{\beta-2}{2}}g(r)dr,
\]
which can then be bootstrapped to control $\|\tilde{Z}(t,0;s,\cdot)\|_{L^1}$ courtesy of the fact that $h'$ is bounded uniformly away from zero \eqref{ptwsbdhprime}, while the singularity in $h'$ at $\xi=0$ is integrable, \eqref{ptwsbdhprime2}. All this is made rigorous in the following Theorem.
\begin{thm}\label{FSest}
Let $\Gamma$, $G$ be as in \eqref{defgamma2}-\eqref{defG2}, and let $Z$ be the fundamental solution to the operator $\cd_{t,\xi}$ \eqref{parabop}. Then the following bounds hold true for every $0\leq s<t\leq T$:
\begin{align}
&\|Z(t,0;s,\cdot)\|_{L^{\infty}}\leq \nu^{-\frac{1}{2}}(t-s)^{-\frac{1}{2}}+\mu_1\nu^{\frac{\beta-2}{2}}\int_s^t(t-r)^{\frac{\beta-2}{2}}g(r)dr,\label{supopbd}\\
&\intsol Z(t,0;s,\sigma)d\sigma\leq G(s,t)\label{goodopnorm},\\
&\intsol Z(t,0;s,\sigma)d\sigma\leq\left[1+2\mu_1(1-\beta)^{\beta}\int_s^tg(r)\|Z(t,0;r,\cdot)\|_{L^{\infty}}^{1-\beta}dr\right]^{\frac{1}{1-\beta}}\label{scaledopnorm}.
\end{align}
\end{thm}
\begin{rmk}
One can of course plug in \eqref{supopbd} into \eqref{scaledopnorm} to end up with an estimate depending only on $g$, $\beta$, $\nu$, and $\mu_1$. Bound \eqref{scaledopnorm} is sharper than \eqref{goodopnorm}: for instance, it does not grow exponentially in $\|g\|_{L^1}$. However, bound \eqref{goodopnorm} is much easier to manipulate. 
\end{rmk}

\begin{proof}
Recall that the fundamental solution takes the form $Z=\Psi+\tilde{Z}$, and so  for any fixed $(t,\xi)$, $\tilde{Z}$ satisfies \eqref{ncpf3.2} backwards in time. It follows that if $\cw(t;r,\sigma):=\tilde{Z}(t,0;r,\sigma)$, then
\[
-\partial_{r}\cw(t;r,\sigma)-4\nu\partial^2_\sigma\cw(t;r,\sigma)+\mu_1g(r)h(\sigma)\partial_\sigma \cw(t;r,\sigma)=-\mu_1g(r)h(\sigma)\partial_\sigma\Psi(t-r,\sigma),\quad \forall (r,\sigma)\in[s,t)\times\rone,
\]
and using \eqref{FSdensity}, it is straightforward to verify
\begin{equation}\label{ICFS}
	\lim_{r\rightarrow t^{-}}\intsol|\cw(t;r,\sigma)|^pd\sigma=0,\quad \forall p\in[1,\infty).
\end{equation}
To get \eqref{supopbd}, we get $L^p$ bounds on $\cw$ then send $p$ to infinity. For $p\geq2$, set $f_p(r):=\|\cw(t;r,\cdot)\|^p_{L^p}$, and notice that 
\begin{align*}
f_p'(r)=&-4p\nu\intsol|\cw(t;r,\sigma)|^{p-2}\cw(t;r,\sigma)\partial^2_\sigma\cw(t;r,\sigma)d\sigma \nonumber\\
&+\mu_1pg(r)\intsol h(\sigma)\partial_\sigma\cw(t;r,\sigma) |\cw(t;r,\sigma)|^{p-2}\cw(t;r,\sigma)d\sigma\nonumber\\
&+\mu_1pg(r)\intsol|\cw(t;r,\sigma)|^{p-2}\cw(t;r,\sigma)h(\sigma)\partial_\sigma\Psi(t-r,\sigma)d\sigma.\label{c35}
\end{align*}
Integration by parts yields:
\begin{equation*}\label{c36}
-\intsol|\cw(t;r,\sigma)|^{p-2}\cw(t;r,\sigma)\partial^2_\sigma\cw(t;r,\sigma)d\sigma=(p-1)\intsol|\cw(t;r,\sigma)|^{p-2}(\partial_\sigma\cw(t;r,\sigma))^2d\sigma\geq0,
\end{equation*}
and since $\partial_\sigma\cw(t;r,\sigma) |\cw(t;r,\sigma)|^{p-2}\cw(t;r,\sigma)=p^{-1}\partial_{\sigma}\left[|\cw(t;r,\sigma)|^{p}\right]$, we get
\begin{equation*}\label{c37}
	p\intsol h(\sigma)\partial_\sigma\cw(t;r,\sigma) |\cw(t;r,\sigma)|^{p-2}\cw(t;r,\sigma)d\sigma=-\intsol h'(\sigma)|\cw(t;r,\sigma)|^{p}d\sigma\geq -\|h'\|_{L^{\infty}}f_p(r),
\end{equation*}
whence the lower bound 
\begin{equation}\label{c35b}
	f_p'(r)\geq -\mu_1g(r)\|h'\|_{L^{\infty}}f_p(r)+\mu_1pg(r)\intsol|\cw(t;r,\sigma)|^{p-2}\cw(t;r,\sigma)h(\sigma)\partial_\sigma\Psi(t-r,\sigma)d\sigma,
\end{equation}
since $\mu_1$ and $g$ are both non-negative. Recall that $h(\sigma)=h_{\epsilon}(\sigma)$ is a mollification of $\sigma^{\beta}$, so that $\|h_\epsilon'\|_{L^{\infty}}$ is a bad constant. As will be shown below, upon letting $p\rightarrow\infty$ the bad constant $\|h'\|_{L^{\infty}}$ will dissapear (as it should, since we are obtaining an $L^\infty$ bound for a solution to a transport-diffusion equation). To control the remaining integral, we first employ H\"older's (integral) inequality together with $|h(\sigma)|\leq|\sigma|^{\beta}$ from \eqref{ptwsbdh}:
\begin{equation}\label{c38}
	p\intsol|\cw(t;r,\sigma)|^{p-2}\cw(t;r,\sigma)h(\sigma)\partial_\sigma\Psi(t-r,\sigma)d\sigma\geq-p f_p^{\frac{p-1}{p}}(r)\left[\intsol |\sigma|^{p\beta}|\partial_\sigma\Psi(t-r,\sigma)|^pd\sigma\right]^{1/p}.
\end{equation}
The bound
\[
|\sigma|^{\beta}|\partial_\sigma\Psi(t-r,\sigma)|=\frac{|\sigma|^{\beta+1}}{8\nu (t-r)}\Psi(t-r,\sigma)\leq \sqrt{2^{5\beta-1}\nu^{\beta-1}}(t-r)^{\frac{\beta-1}{2}}\Psi(t-r,\sigma/\sqrt{2}),
\]
where we used the inequality $|\sigma|^{\gamma}\leq a^{\gamma/2}\exp(\sigma^2/a)$ valid for $\gamma,a>0$, allows us to deduce 
\begin{equation}\label{c39}
	\left[\intsol |\sigma|^{p\beta}|\partial_\sigma\Psi(t-r,\sigma)|^pd\sigma\right]^{1/p}\leq \sqrt{2^{5\beta-1}\nu^{\beta-1}}(t-r)^{\frac{\beta-1}{2}}2^{\frac{1}{2p}}\|\Psi(t-r,\cdot)\|_{L^p}.
\end{equation}
Plugging \eqref{c39} into \eqref{c38} yields
\begin{equation}\label{c310}
	p\intsol|\cw(t;r,\sigma)|^{p-2}\cw(t;r,\sigma)h(\sigma)\partial_\sigma\Psi(t-r,\sigma)d\sigma\geq -p2^{\frac{1}{2p}}A_{\beta,\nu}(t-r)^{\frac{\beta-1}{2}}f_p^{\frac{p-1}{p}}(r)\|\Psi(t-r,\cdot)\|_{L^p},
\end{equation}
where $A_{\beta,\nu}:=\sqrt{2^{5\beta-1}\nu^{\beta-1}}$. If we define $F_p(r):=\exp\left(-\mu_1\|h'\|_{L^{\infty}}\Gamma(r,t)\right)f_p(r)$ (where $\Gamma$ is as defined in \eqref{defgamma2}) and recall once again that $\mu_1$ and $g$ are both non-negative, we may multiply \eqref{c310} by $\mu_1 g(r)$ and use the resultant inequality to bound the integral on the right-hand side of \eqref{c35b} from below, yielding the inequality
\[
F_p'(r)\geq-\mu_1p2^{\frac{1}{2p}}A_{\beta,\nu}(t-r)^{\frac{\beta-1}{2}}g(r)F_p^{\frac{p-1}{p}}(r)\|\Psi(t-r,\cdot)\|_{L^p}.
\]
Integrating the above inequality from $r=s$ to $r=t$, and using \eqref{ICFS} we get 
\[
f^{1/p}_p(s)\leq \mu_12^{\frac{1}{2p}}A_{\beta,\nu}\int_s^t(t-r)^{\frac{\beta-1}{2}}g(r)\|\Psi(t-r,\cdot)\|_{L^p}dr\exp\left[\frac{\|h'\|_{L^{\infty}}}{p}\Gamma(s,t)\right].
\]
Sending $p\rightarrow\infty$ and noting that $\|\Psi(t-r,\cdot)\|_{L^{\infty}}\leq2^{-2}\nu^{-1/2}(t-r)^{-1/2}$ we get
\begin{equation}\label{c311}
	\|\cw(t;s,\cdot)\|_{L^{\infty}}=\|\tilde{Z}(t,0;s,\cdot)\|_{L^{\infty}}\leq \sqrt{2^{5(\beta-1)}}\mu_1\nu^{\frac{\beta-2}{2}}\int_s^t(t-r)^{\frac{\beta-2}{2}}g(r)dr,
\end{equation}
from which \eqref{supopbd} follows. Next, for $\epsilon>0$, let $\chi_{\epsilon}$ be a smooth, convex function such that $\chi_{\epsilon}(u)\rightarrow|u|$ uniformly on compact sets, and set 
\[
f_\epsilon(r):=\intsol \chi_{\epsilon}(\cw(t;r,\sigma))d\sigma,
\]
from which we obtain
\begin{align}
f_\epsilon'(r)=&-4\nu\intsol\chi_{\epsilon}'(\cw(t;r,\sigma))\partial^2_\sigma \cw(t;r,\sigma)d\sigma+\mu_1g(r)\intsol h(\sigma)\partial_\sigma\cw(t;r,\sigma) \chi_{\epsilon}'(\cw(t;r,\sigma))d\sigma\nonumber\\
&+\mu_1g(r)\intsol h(\sigma)\chi_{\epsilon}'(\cw(t;r,\sigma))\partial_\sigma\Psi(t-r,\sigma)d\sigma.\label{c312}
\end{align}
It is clear that $\chi_\epsilon''\geq0$, $0\leq\chi_{\epsilon}(u)\leq|u|$, and $\partial_\sigma[\chi_\epsilon(\cw)]=\chi_\epsilon'(\cw)\partial_\sigma\cw$, so that upon integration by parts and using $\|h'\|_{L^1(-1,1)}\leq 2$ together with $h'(\sigma)\leq 2\beta$ whenever $|\sigma|\geq 1$ (bounds \eqref{ptwsbdhprime} and \eqref{ptwsbdhprime2}) we arrive at
\begin{align}
&-4\nu\intsol\chi_{\epsilon}'(\cw(t;r,\sigma))\partial^2_\sigma \cw(t;r,\sigma)d\sigma+\mu_1g(r)\intsol h(\sigma)\partial_\sigma\cw(t;r,\sigma) \chi_{\epsilon}'(\cw(t;r,\sigma))d\sigma\nonumber\\
&\geq-\mu_1g(r)\intsol h'(\sigma)\chi_{\epsilon}(\cw(t;r,\sigma))d\sigma\geq -2\mu_1g(r)\|\cw(t;r,\cdot)\|_{L^{\infty}}-2\mu_1g(r)f_{\epsilon}(r).\label{c313}
\end{align}
On the other hand, using $|\chi_\epsilon'|\leq 1$ and \eqref{c39} with $p=1$ we get 
\begin{equation}\label{c314}
\mu_1g(r)\intsol h(\sigma)\chi_{\epsilon}'(\cw(t;r,\sigma))\partial_\sigma\Psi(t-r,\sigma)d\sigma\geq -\sqrt{2^{5\beta}}\mu_1\nu^{\frac{\beta-1}{2}}g(r)(t-r)^{\frac{\beta-1}{2}}.
\end{equation}
Bounding \eqref{c312} from below by \eqref{c313} and \eqref{c314} renders:
\[
f_{\epsilon}'(r)+2\mu_1g(r)f_{\epsilon}(r)\geq-8\mu_1g(r)\left[\|\cw(t;r,\cdot)\|_{L^{\infty}}+\nu^{\frac{\beta-1}{2}}(t-r)^{\frac{\beta-1}{2}}\right]
\]
Integrating this from $r=s$ to $r=t$ and using \eqref{ICFS} once again we get
\begin{equation}\label{c315}
f_\epsilon(s)\leq 8\mu_1e^{2\mu_1\Gamma(s,t)}\int_s^tg(r)\left[\|\cw(t;r,\cdot)\|_{L^{\infty}}+\nu^{\frac{\beta-1}{2}}(t-r)^{\frac{\beta-1}{2}}\right]dr. 
\end{equation}
Since $s\leq r<t$, \eqref{c311} gives us
\[
\|\cw(t;r,\cdot)\|_{L^{\infty}}\leq \mu_1\nu^{\frac{\beta-2}{2}}\int_r^t(t-r')^{\frac{\beta-2}{2}}g(r')dr'\leq \mu_1\nu^{\frac{\beta-2}{2}}\int_s^t(t-r')^{\frac{\beta-2}{2}}g(r')dr'.
\]
Plugging this into \eqref{c315}, sending $\epsilon\rightarrow0^+$ and noting that $\|\Psi(t-s,\cdot)\|_{L^1}=1$ we conclude with
\[
\intsol Z(t,0;s,\sigma)\leq 1+8\mu_1\nu^{\frac{\beta-2}{2}}\left[\mu_1\Gamma(s,t)\int_s^t(t-r)^{\frac{\beta-2}{2}}g(r)dr+\nu^{1/2}\int_s^t(t-r)^{\frac{\beta-1}{2}}g(r)dr\right]e^{2\mu_1\Gamma(s,t)},
\]
for every $0\leq s<t\leq T$, which is precisely \eqref{goodopnorm}.

To get \eqref{scaledopnorm}, we let $f$ be the $L^1$ norm of the full fundamental solution: $f(r):=\|Z(t,0;r,\cdot)\|_{L^1}$, $s\leq r< t$, and note that since $Z\geq0$, we may integrate \eqref{ncpf3.2} in $\sigma$ and evaluate at $\xi=0$ to get that $f'$ satisfies
\[
f'(r)=\mu_1g(r)\intsol h(\sigma)\partial_\sigma Z(t,0;r,\sigma)d\sigma=-\mu_1g(r)\intsol h'(\sigma)Z(t,0;r,\sigma)d\sigma.
\]
For $\rho>0$, we split the integral
\[
\intsol h'(\sigma)Z(t,0;r,\sigma)d\sigma=\int_{|\sigma|\leq\rho}h'(\sigma)Z(t,0;r,\sigma)d\sigma+\int_{|\sigma|\geq\rho}h'(\sigma)Z(t,0;r,\sigma)d\sigma.
\]
For the first integral, we bound it from above by utilizing \eqref{ptwsbdhprime2}
\[
\int_{|\sigma|\leq\rho}h'(\sigma)Z(t,0;r,\sigma)d\sigma\leq2 \|Z(t,0;r,\cdot)\|_{L^{\infty}}|h(\rho)|\leq 2\|Z(t,0;r,\cdot)\|_{L^{\infty}}\rho^{\beta},
\]
while the second one is controlled by employing \eqref{ptwsbdhprime}
\[
\int_{|\sigma|\leq\rho}h'(\sigma)Z(t,0;r,\sigma)d\sigma\leq2\beta f(r)\rho^{\beta-1},
\]
whence
\[
0\leq\intsol h'(\sigma)Z(t,0;r,\sigma)d\sigma\leq2\|Z(t,0;r,\cdot)\|_{L^{\infty}}\rho^{\beta}+2\beta f(r)\rho^{-(1-\beta)}.
\]
The function $R(\rho):=a\rho^{r_0}+b\rho^{-r_1}$, where $a,b, r_0, r_1>0$ is optimized when $\rho=\rho_*$, with
\[
\rho_*=\left(\frac{br_1}{ar_0}\right)^{\frac{1}{r_0+r_1}},\quad R(\rho_*)=\left(\frac{r_0+r_1}{r_1}\right)\left(\frac{r_1}{r_0}\right)^{\frac{r_0}{r_0+r_1}}\left(a^{r_1}b^{r_0}\right)^{\frac{1}{r_0+r_1}},
\]
leading to
\[
0\leq\intsol h'(\sigma)Z(t,0;r,\sigma)d\sigma\leq C_{\beta}\|Z(t,0;r,\cdot)\|_{L^{\infty}}^{1-\beta}f^{\beta}(r),\quad C_{\beta}=2(1-\beta)^{\beta-1}.
\]
As $\mu_1$ and $g$ are both non-negative, \eqref{scaledopnorm} follows immediately from integrating
\[
f'(r)\geq-C_{\beta}\mu_1g(r)\|Z(t,0;r,\cdot)\|_{L^{\infty}}^{1-\beta}f^{\beta}(r)
\]
from $r=s$ to $r=t$ while noting that $f(t)=1$.
\end{proof}
Having proven Theorem \ref{FSest}, Lemma \ref{lemspcase} is an immediate Corollary: the solution to \eqref{eqvspcase} is represented by 
\[
\cv(t,\xi)=\intsol Z(t,\xi;0,\sigma)\cv_0(\sigma)d\sigma.
\]
The symmetric assumptions imposed on the initial data means that the solution will inherit them as well, according to item (3) in Theorem \ref{classicalthm} (the function $h_\epsilon$ is odd and $h'_\epsilon$ satisfies (S) from Theorem \ref{secondthm}). Thus, $\cv(t,\cdot)$ is maximized at $\xi=0$ for every $t$, from which \eqref{splinftybd} follows immediately from \eqref{supopbd}-\eqref{goodopnorm}.
\subsection{Proof of Proposition \ref{keyprop}}\label{pfkeyprop}
We now allow $\lambda\in[0,2]$, drop the extra assumptions (S) made on the initial data, and add a forcing term. As explained in \S\ref{heur2}, the symmetric properties (S) are intrinsic to the associated stochastic flow map. To exploit the enhanced regularity coming from the transport term, we employ the Feynman-Kac formula, item 4 in Theorem \ref{classicalthm}: let $\cv$ solve \eqref{c33} (dropping the subscript $\epsilon$ for convenience). It follows that if $\Phi$ is the stochastic diffeomorphism associated to the transport diffusion operator $\partial_t-4\nu\partial^2_\xi-\mu_1g(t)h(\xi)\partial_\xi$ given to us by Theorem \ref{sdiffexis},
\[
\partial_t\Phi(t,\xi;q)=-\mu_1g(t)h(\Phi(t,\xi;q))+\sqrt{8\nu}\dot{W}(t;q),\quad \Phi(0,\xi;q)=\xi, 
\]
and if we define $\ca:=\Phi^{-1}$ to be its inverse, set 
\[
\eta(t,\xi):=\exp\left(\mu_2\int_0^tg(r)h'\left(\Phi(r,\xi\right))dr\right),
\]
and adopt the notation  $\ca_{t,\xi}=\ca(t,\xi;q)$ and $\Phi_{t,\xi}=\Phi(t,\xi;q)$ (where $q$ is the probability variable), then solution to \eqref{c33} can be represented by 
\begin{equation}\label{c316}
	\cv(t,\xi)=\mathbb{E}\left[\cv_0(\ca_{t,\xi})\eta(t,\ca_{t,\xi})\right]+\mathbb{E}\left[\eta(t,\ca_{t,\xi})\int_0^t\frac{d(s,\Phi(s,\ca_{t,\xi}))}{\eta(s,\ca_{t,\xi})}ds\right].
\end{equation}
Since the lower order term in \eqref{c33}, the term $\mu_2 g(t)h'(\xi)$, is a constant multiple of the derivative of the transport term, the term $\mu_1 g(t)h(\xi)$, we can represent $\eta$ in terms of the derivative of the inverse flow map, $\cb:=\partial_\xi\ca$. Indeed, let us observe that if $\psi:=\partial_\xi\Phi$, then (dropping the probability variable $q$ for convenience)
\[
\partial_t\psi(t,\xi)=-\mu_1g(t)h'(\Phi(t,\xi))\psi(t,\xi),\quad \psi(0,\xi)=1,
\]
so that 
\[
\psi(t,\xi)=\partial_\xi\Phi(t,\xi)=\exp\left(-\mu_1\int_0^tg(r)h'(\Phi(r,\xi))dr\right).
\]
Since $\ca_{t,\xi}=\Phi^{-1}(t,\xi)$ is the inverse flow map, if we set $\cb(t,\xi):=\partial_\xi\ca(t,\xi)$ we get
\[
\cb(t,\xi)=\frac{1}{\partial_\xi\Phi(t,\ca_{t,\xi})}=\exp\left(\mu_1\int_0^tg(r)h'(\Phi(r,\ca_{t,\xi}))dr\right)=\left[\exp\left(\int_0^tg(r)h'(\Phi(r,\ca_{t,\xi}))dr\right)\right]^{\mu_1}\geq 1,
\]
with the last inequality holding true almost surely as $gh'\geq0$. From here, it is clear that if $\lambda:=\mu_2/\mu_1$, then 
\begin{equation}\label{cc315}
	\cb^{\lambda}(t,\xi)=[\partial_\xi\ca(t,\xi)]^{\lambda}=\exp\left(\mu_2\int_0^tg(r)h'\left(\Phi(r,\ca_{t,\xi}\right))dr\right)=\eta(t,\ca_{t,\xi}),
\end{equation}
or equivalently, $\eta(t,\xi)=\cb^{\lambda}(t,\Phi_{t,\xi})$. Plugging this into \eqref{c316}:
\begin{equation}\label{c317}
\cv(t,\xi)=\underbrace{\mathbb{E}\left[\cv_0(\ca_{t,\xi})\cb^{\lambda}(t,\xi)\right]}_{=:\cv_1(t,\xi)}+\underbrace{\mathbb{E}\left[\cb^{\lambda}(t,\xi)\int_0^td(s,\Phi(s,\ca_{t,\xi}))\cb^{-\lambda}(s,\Phi(s,\ca_{t,\xi}))ds\right]}_{=:\cv_2(t,\xi)}.
\end{equation}
Such a representation makes it clear that in order to get $L^p$ bounds on the solution, one has to get estimates on various moments on $\cb$, which happens to solve 
\begin{equation}\label{c318}
	\partial_t\cb(t,\xi;q)-4\nu\partial^2_\xi\cb(t,\xi;q)-\mu_1g(t)\partial_\xi[h(\xi)\cb(t,\xi;q)]+\sqrt{8\nu}\dot{W}(t;q)=0,\quad \cb(0,\xi;q)=1,
\end{equation}
since $\ca$ solves
\[
\partial_t\ca(t,\xi;q)-4\nu\partial^2_\xi\ca(t,\xi;q)-\mu_1g(t)h(\xi)\partial_\xi\ca(t,\xi;q)+\sqrt{8\nu}\dot{W}(t;q)=0,\quad \ca(0,\xi;q)=\xi,
\]
according to Theorem \ref{sdiffexis}. For $\lambda\in[0,1]$, we may apply H\"older's inequality together with the fact that $\cb\geq1$ almost surely, to get
\begin{equation}\label{c319}
	|\cv(t,\xi)|\leq \left[\|\cv_0\|_{L^{\infty}}+\int_0^t\|d(s,\cdot)\|_{L^{\infty}}ds\right]\mathbb{E}[\cb^{\lambda}(t,\xi)]\leq \left[\|\cv_0\|_{L^{\infty}}+\int_0^t\|d(s,\cdot)\|_{L^{\infty}}ds\right]\left(\mathbb{E}[\cb(t,\xi)]\right)^{\lambda}.
\end{equation}
Setting $\bar{\cb}(t,\xi):=\mathbb{E}[\cb (t,\xi)]$ we see from \eqref{c318} that 
\[
\partial_t\bar{\cb}(t,\xi)-4\nu\partial^2_\xi\bar{\cb}(t,\xi)-\mu_1g(t)\partial_\xi [h(\xi)\bar{\cb}(t,\xi)]=0,\quad \bar{\cb}(0,\xi)=1.
\]
In particular, $\bar{\cb}$ satisfies the hypotheses required in Lemma \ref{lemspcase} (i.e., condition (S) from Theorem \ref{secondthm}), so that 
\begin{equation}\label{cc319}
	\sup_{\xi\in\rone}\mathbb{E}[\cb(t,\xi)]=\|\bar{\cb}(t,\cdot)\|_{L^{\infty}}=\bar{\cb}(t,0)\leq G(0,t).
\end{equation}
Plugging this into \eqref{c319} and taking the supremum over $\xi$ on the left-hand side immediately gives us \eqref{linftybd}. When $\lambda=1$ and if $d(t,\cdot)$ satisfies (S) from Theorem \ref{secondthm}, then the volume potential $\cv_2$ can be represented by 
\[
\cv_2(t,\xi)=\int_0^t\intsol Z(t,\xi;s,\sigma)d(s,\sigma)d\sigma ds,
\]
where $Z$ is the fundamental solution to the operator $\cd_{t,\xi}$ \eqref{parabop}. According to item (3) in Theorem \ref{classicalthm}, the solution will be maximized at $\xi=0$, i.e., $\|\cv_2(t,\cdot)\|_{L^{}\infty}=\cv_2(t,0)$, from which \eqref{implinftybd} follows from evaluating the above at $\xi=0$ and \eqref{supopbd}. 

To get $L^{\infty}$ bounds when $\lambda\geq1$ requires us to estimate the $\lambda$'th moment of $\cb$ in $L^{\infty}$, and we were unable to obtain such estimates; at least not in terms of supercritical entities, as explained in section \S\ref{heur3}. On the other hand, if we try to bound the solution $\cv$ in $L^1$ via utilizing the representation \eqref{c317}, we see that we need to bound the $\lambda-1$ moment of $\cb$, which is certainly less than one if we assume $\lambda\in[1,2]$. Indeed, we first employ the linearity property of taking expectations (Fubini-Tonelli), followed by the change in variables $\sigma:=\ca_{t,\xi}$ (or equivalently $\xi=\Phi_{t,\sigma}$) together with the fact that $\cb=\partial_\xi\ca\geq1$ almost surely, and H\"older's inequality to get
\begin{align*}
\intsol|\cv_1(t,\xi)|d\xi&\leq \mathbb{E}\left[\intsol |\cv_0(t,\ca_{t,\xi})|\cb^{\lambda}(t,\xi)d\xi\right]=\intsol |\cv_0(t,\sigma)|\mathbb{E}[\cb^{\lambda-1}(t,\Phi_{t,\sigma})]d\sigma\\
&\leq\intsol |\cv_0(t,\sigma)|\left(\mathbb{E}[\cb(t,\Phi_{t,\sigma})]\right)^{{\lambda-1}}d\sigma\leq \|\cv_0\|_{L^1}\|\bar{\cb}(t,\cdot)\|_{L^{\infty}}^{\lambda-1}.
\end{align*}
The volume potential $\cv_2$ is handled similarly: invoking a change in variables $\sigma:=\ca_{t,\xi}$
\[
\|\cv_2(t,\cdot)\|_{L^1}\leq\int_0^t\mathbb{E}\left[\intsol \cb^{\lambda}(t,\xi)\frac{|d(s,\Phi(s,\ca_{t,\xi}))|}{\cb^{\lambda}(s,\Phi(s,\ca_{t,\xi}))}d\xi\right]ds=\int_0^t\mathbb{E}\left[\intsol \cb^{\lambda-1}(t,\Phi_{t,\sigma})\frac{|d(s,\Phi_{s,\sigma})|}{\cb^{\lambda}(s,\Phi_{s,\sigma})}d\sigma\right]ds.
\]
Employing one more change in variable $y:=\Phi_{s,\sigma}$ yields:
\begin{align*}
\|\cv_2(t,\cdot)\|_{L^1}\leq&\int_0^t\mathbb{E}\left[\intsol \cb^{\lambda-1}(t,\Phi(t,\ca_{s,y}))|d(s,y)|\cb^{1-\lambda}(s,y)dy\right]ds\\
&\leq \int_0^t\intsol \mathbb{E}\left[\cb^{\lambda-1}(t,\Phi(t,\ca_{s,y}))\right]|d(s,y)|dy ds,
\end{align*}
where in the last inequality we used the fact that $1-\lambda\leq0$ (since $\lambda\in[1,2]$) and $\cb\geq1$ almost surely. As in the case of $\cv_1$, we may exploit the fact that $\lambda\in[1,2]$ and employ H\"older's inequality to obtain
\[
\|\cv_2(t,\cdot)\|_{L^1}\leq \|\bar{\cb}(t,\cdot)\|_{L^{\infty}}^{\lambda-1}\int_0^t\intsol |d(s,y)|dy ds,
\]
whence
\[
\|\cv(t,\cdot)\|_{L^1}\leq\left[\|\cv_0\|_{L^1}+\int_0^t\|d(s,\cdot)\|_{L^1} ds\right]\|\bar{\cb}(t,\cdot)\|_{L^{\infty}}^{\lambda-1},
\]
and so the claimed bound \eqref{l1bd} follows from the above inequality together with \eqref{cc319}. Finally, it is straightforward to check that if the initial data and forcing satisfy (S) from Theorem \ref{secondthm}, then so will the solution (according to item (3) in Theorem \ref{classicalthm} since $h$ is odd and $h'$ does satisfy (S) as well).
\subsection{Proof of Theorem \ref{thmbuildmod}}\label{pfthmbuildmod}
Assume the hypotheses of Theorem \ref{thmbuildmod}, extend $\Omega_0$ and the forcing term $a$ in an odd fashion about $\xi=0$, let $\cv_0=\Omega_0'$, $d:=\partial_\xi a$, and let $\{\cv_\epsilon\}$ be the sequence of solutions to \eqref{c33} given to us by Proposition \ref{keyprop} with any $\mu_1\geq C\beta^{-1}$ and $\mu_2=\mu_1+C\beta^{-1}$ (so that $\lambda:=\mu_2/\mu_1\in[1,2]$). Notice that $\cv_\epsilon$ is smooth and even, so that if we define 
\begin{equation}\label{defoeps}
	\Omega_\epsilon(t,\xi):=\int_0^\xi\cv_\epsilon(t,\eta)d\eta,
\end{equation}
then this is a smooth odd function that solves \eqref{c31} ($\cv_\epsilon(t,\cdot)$ is even, forcing $\partial_\xi\cv_{\epsilon}(t,0)=0$), with $h_\epsilon$ as drift and $h'_\epsilon$ inside the integral on the right-hand side ($h_\epsilon'$ converges to $\beta\xi^{\beta-1}$, hence the factor of $\beta^{-1}$ in $\mu_2$). We wish to obtain uniform in $\epsilon$ bounds on $\Omega_\epsilon$ in order to construct a solution, utilizing a classical Arzel\`a-Ascoli Theorem. To be able to pass to the zero $\epsilon$ limit (along a subsequence if necessary), we need to control some H\"older norm of $\cv_\epsilon(t,\cdot)$. This norm is allowed to depend on sub-critical quantities, say $\|g\|_{L^{p}}$ some $p>2/(1+\beta)$, since it does not show up in the final estimate. It is only needed to make sure $\Omega$ is twice continuously differentiable in space (and once in time). Even though we are allowed to make sub-critical assumptions when running the compactness argument, the fact that the drift term is unbounded presents some technical difficulties that need to be taken care of delicately. This is why in Remark \ref{rmkforce} we say that if the forcing term $a(t,\cdot)$ is non-decreasing and concave (which makes $d(t,\cdot):=\partial_\xi a(t,\cdot)$ satisfies the symmetries (S) from Theorem \ref{secondthm}) we can guarantee the existence of a classical solution to \eqref{c31} under the assumption that $a$ is just continuous and bounded. Without such symmetries, our approach requires higher-regularity from the forcing term $a$: it needs to be Lipschitz in the spatial variable.

Let us start by obtaining $L^{\infty}$ bounds that are uniform in $\epsilon$. We emphasize again that such a bound is only required to construct a $C_t^1C_x^2$ solution to \eqref{c31}. The key estimates that we care about at the end of the day are \eqref{linftybd} and \eqref{l1bd}. First, recall the representation \eqref{c317}
\begin{equation}\label{c319n}
\cv_\epsilon(t,\xi)=\underbrace{\mathbb{E}\left[\cv_0(\ca_{t,\xi})\cb^{\lambda}(t,\xi)\right]}_{=:\cv_{\epsilon,1}(t,\xi)}+\underbrace{\mathbb{E}\left[\cb^{\lambda}(t,\xi)\int_0^td(s,\Phi(s,\ca_{t,\xi}))\cb^{-\lambda}(s,\Phi(s,\ca_{t,\xi}))ds\right]}_{=:\cv_{\epsilon,2}(t,\xi)},
\end{equation}
where $\cb$ solves 
\[
\partial_t\cb(t,\xi;q)-4\nu\partial^2_\xi\cb(t,\xi;q)-\mu_1g(t)\partial_\xi[h_\epsilon(\xi)\cb(t,\xi;q)]+\sqrt{8\nu}\dot{W}(t;q)=0,\quad \cb(0,\xi;q)=1.
\]
It follows that (utilizing the almost sure lower bound $\cb\geq1$)
\begin{equation}\label{c319n2}
	\|\cv_\epsilon(t,\cdot)\|_{L^{\infty}}\leq \left[\|\cv_0\|_{L^{\infty}}+\int_0^t\|d(s,\cdot)\|_{L^{\infty}}ds\right]\sup_{\xi\in\rone}\mathbb{E}[\cb^{\lambda}(t,\xi)].
\end{equation}
Keeping in mind that $\lambda=\mu_2/\mu_1$, an application of It\^o's Lemma followed by taking expectations tells us that $\cb_{\lambda}(t,\xi):=\mathbb{E}[\cb^{\lambda}(t,\xi)]$ solves 
\[
\partial_t\cb_{\lambda}(t,\xi)-4\nu\partial^2_\xi\cb_{\lambda}(t,\xi)-\mu_1g(t)h_\epsilon(\xi)\partial_\xi\cb_{\lambda}(t,\xi)-\mu_2g(t)h_\epsilon'(\xi)\cb_{\lambda}(t,\xi)=0,\quad \cb_{\lambda}(0,\xi)=1.
\]
Further for every fixed $t$, $\cb_\lambda(t,\cdot)$ satisfies (S) from Theorem \ref{secondthm}, so that $\cb_{\lambda}(t,0)=\|\cb_{\lambda}(t,\cdot)\|_{L^{\infty}}$, from which an application of Duhamel's principle followed by an integration by parts tell us that 
\begin{align*}
\|\cb_{\lambda}(t,\cdot)\|_{L^{\infty}}=\cb_{\lambda}(t,0)=&1+\mu_2\int_0^tg(s)\intsol\Psi(t-s,\sigma)h_\epsilon'(\sigma)\cb_{\lambda}(s,\sigma)d\sigma ds\\
&-\mu_1\int_0^tg(s)\intsol\partial_\sigma\left[\Psi(t-s,\sigma)h_\epsilon(\sigma)\right]\cb_{\lambda}(s,\sigma)d\sigma ds,
\end{align*}
where $\Psi$ is the heat kernel. A simple change in variable reveals that for any $r\in\rone$, we have 
\begin{equation}\label{c319n3}
	\intsol\Psi(t,\sigma)|\sigma|^rd\sigma=C_r(\nu t)^{r/2},\quad C_r=\frac{4^r}{\sqrt{\pi}}\intsol|\sigma|^re^{-|\sigma|^2}d\sigma.
\end{equation}
From here, a straightforward application of the singular version of Gronwall's inequality (see Lemma \ref{singgron} in Appendix \ref{applps} below) together with \eqref{c319n3} yields a constant $K_0>0$ that is uniform in $\epsilon$ such that:
\begin{equation}\label{uniformepsbd}
	\sup_{t\in[0,T]}\|\cv_\epsilon(t,\cdot)\|_{L^{\infty}}\leq K_0.
\end{equation} 
The constant $K_0$ will depend on, for instance, $\|g\|_{L^{p}}$, $p>2/(1+\beta)$. Again, even though this is a subcritical assumption, this quantity does not show up in the final bounds \eqref{linftybd} and \eqref{l1bd}. It is only required to construct a smooth solution. Moreover, the constant $K_0$ depends on the $L_t^{1}L_x^{\infty}$ norm of $d$, i.e., $a(t,\cdot)$ has to be Lipschitz. If $d(t,\cdot)$ satisfies (S) from Theorem \ref{secondthm} (for which the forcing term $a(t,\cdot)$ described in Remark \ref{rmkforce} is a special case), then we show that $K_0$ can be chosen to depend on the weaker $L_t^{q}L_x^1$ norm of $d$ (some $q\geq2$) instead. Indeed, $\cv_{\epsilon,1}$ from \eqref{c319n} is bounded independent of $d$, while under such symmetry assumptions, the volume potential $\cv_{\epsilon,2}$ is non-negative and is maximized at $\xi=0$. In particular, an application of Duhamel's principle yields:
\begin{align*}
\|\cv_{\epsilon,2}(t,\cdot)\|_{L^{\infty}}=\cv_{\epsilon,2}(t,0)=&\mu_2\int_0^tg(s)\intsol\Psi(t-s,\sigma)h_\epsilon'(\sigma)\cv_{2}(s,\sigma)d\sigma ds\\
&-\mu_1\int_0^tg(s)\intsol\partial_\sigma\left[\Psi(t-s,\sigma)h_\epsilon(\sigma)\right]\cv_{2}(s,\sigma)d\sigma ds\\
&+\int_0^t\intsol\Psi(t-s,\sigma)d(s,\sigma)d\sigma ds.
\end{align*}
The last integral can be controlled by 
\[
\int_0^t\intsol\Psi(t-s,\sigma)d(s,\sigma)d\sigma ds\leq \nu^{-1/2}\int_0^t(t-s)^{-1/2}\|d(s,\cdot)\|_{L^{1}}ds,
\]
while the first two integrals can be handled as in the case of $\cb_\lambda$ above, before applying the singular Gronwall inequality \ref{singgron} below. At the end of it all, we see that $K_0$ in this case depends on subcritical assumptions on $g$ and the $L_t^qL_x^1$ norm of $d$, for some $q>2$. This is the only difference between symmetric and non-symmetric forces. The rest of the compactness argument carries over regardless of symmetry assumptions; in particular we will see that the H\"older semi-norm of $\cv_{\epsilon,2}$ will always be bounded in terms of the $L_t^qL_x^1$ norm of $d$, regardless of symmetries.

Speaking of which, and as mentioned earlier, the fact that we are dealing with an unbounded drift velocity (in the limit) presents some technical difficulties that restrict us from obtaining global H\"older estimates. This is not an issue, since Arzel\`a-Ascoli, see for instance \cite[Theorem 4.44]{Folland1999book}, does not require global continuity estimates; local ones work just fine. To do so, we need to prove the following ``weighted'' H\"older estimates for the heat kernel \eqref{hk}: for any $\xi\in\rone$, any $t>0$, any $\delta\in(0,1)$, and any $\gamma\in(0,1)$ we have
\begin{align}
	&\intsol\left|[\partial_\sigma \Psi(t,\xi-\sigma+\delta)-\partial_\sigma\Psi(t,\xi-\sigma)]h_\epsilon(\sigma)\right|d\sigma\leq \frac{8\delta^{\gamma}}{(\nu t)^{(\gamma+1)/2}}(1+|\xi|^{\beta}+(\nu t)^{\beta/2})^{\gamma}((\nu t)^{\beta}+|\xi|^{\beta})^{1-\gamma},\label{wgthdholdgrad}\\
	&\intsol\left|[\Psi(t,\xi-\sigma+\delta)-\Psi(t,\xi-\sigma)]h_\epsilon'(\sigma)\right|d\sigma\leq \frac{\delta^{\gamma}}{(\nu t)^{(\gamma+1)/2}}[1+(\nu t)^{\gamma/2}].\label{wgthdhold}
\end{align}
We also need to justify using the following standard H\"older estimate (valid for any $\xi,\sigma,\delta,\gamma$):
\begin{equation}\label{c335n}
	|\Psi(t,\xi-\sigma+\delta)-\Psi(t,\xi-\sigma)|\leq (\nu t)^{-(\gamma+1)/2}\delta^{\gamma}.
\end{equation}
We start by proving \eqref{wgthdholdgrad}: note that for any $\mu\in\rone$,
\begin{align}
\intsol\left|\partial_\sigma \Psi(t,\mu-\sigma)h_\epsilon(\sigma)\right|d\sigma&\leq |h_\epsilon(\mu)|\intsol\left|\partial_\sigma \Psi(t,\mu-\sigma)\right|d\sigma+\intsol|h_\epsilon(\sigma)-h_\epsilon(\mu)||\partial_\sigma\Psi(t,\mu-\sigma)|d\sigma\nonumber\\
&\leq 8((\nu t)^{-1/2}|\mu|^{\beta}+(\nu t)^{\frac{\beta-1}{2}}).\label{c331}
\end{align}
Next, we observe that from $|\partial^2_\mu\Psi(t,\mu)|\leq (2\nu t)^{-1}\Psi(t,\mu/2)$ and $|\mu|^\beta\Psi(t,\mu/\sqrt{2})\leq2^{5\beta/2} (\nu t)^{\beta/2}\Psi(t,\mu/2)$ we have
\begin{align*}
&\left|[\partial_\sigma \Psi(t,\xi-\sigma+\delta)-\partial_\sigma\Psi(t,\xi-\sigma)]h_\epsilon(\sigma)\right|=\left|\int_{\xi-\sigma}^{\xi-\sigma+\delta}\partial^2_\mu\Psi(t,\mu)h_\epsilon(\sigma)d\mu\right|\nonumber\\
&\leq \int_{\xi-\sigma}^{\xi-\sigma+\delta}|\partial^2_\mu\Psi(t,\mu)(h_\epsilon(\sigma)-h_\epsilon(\xi-\mu))|d\mu+\int_{\xi-\sigma}^{\xi-\sigma+\delta}|\partial^2_\mu\Psi(t,\mu)h_\epsilon(\xi-\mu)|d\mu\nonumber\\
&\leq \frac{4}{\nu t}\int_{\xi-\sigma}^{\xi-\sigma+\delta}\Psi\left(t,\frac{\mu}{2}\right)\left[|\sigma-\xi+\mu|^{\beta}+|\xi|^{\beta}+(\nu t)^{\beta/2}\right]d\mu\leq \frac{4(1+|\xi|^{\beta}+(\nu t)^{\beta/2})}{\nu t}\int_{\xi-\sigma}^{\xi-\sigma+\delta}\Psi\left(t,\frac{\mu}{2}\right)d\mu,
\end{align*}
where in the last inequality we used $|\sigma-\xi+\mu|\leq\delta\leq1$. Setting $A:=4(1+|\xi|^{\beta}+t^{\beta/2})$ and integrating the above over $\sigma$ we get
\begin{equation}\label{c332}
	\intsol\left|[\partial_\sigma \Psi(t,\xi-\sigma+\delta)-\partial_\sigma\Psi(t,\xi-\sigma)]h_\epsilon(\sigma)\right|d\sigma\leq \frac{A}{\nu t}\intsol\int_{\xi-\sigma}^{\xi-\sigma+\delta}\Psi\left(t,\frac{\mu}{2}\right)d\mu d\sigma=\frac{2A\delta}{\nu t}.
\end{equation}
The claimed bound \eqref{wgthdholdgrad} follows from \eqref{c331} and \eqref{c332}, together with the simple identity $a=a^{\gamma}a^{1-\gamma}$. As in for \eqref{wgthdhold}, we first note that for any $\sigma\in\rone$,
\begin{equation}\label{c333}
|\Psi(t,\xi-\sigma+\delta)-\Psi(t,\xi-\sigma)|\leq \frac{1}{\sqrt{2\nu t}}\int_{\xi-\sigma}^{\xi-\sigma+\delta}\Psi\left(t,\frac{\mu}{\sqrt{2}}\right)d\mu,
\end{equation}
so that from bound \eqref{ptwsbdhprime}, $h'(\sigma)\leq 1$ for any $|\sigma|\geq1$, we get 
\begin{equation}\label{c334}
	\int_{|\sigma|\geq 1}\left|[\Psi(t,\xi-\sigma+\delta)-\Psi(t,\xi-\sigma)]h_\epsilon'(\sigma)\right|d\sigma\leq \frac{2\delta}{\sqrt{\nu t}}.
\end{equation}
On the other hand, we have 
\[
\int_{\xi-\sigma}^{\xi-\sigma+\delta}\Psi\left(t,\frac{\mu}{\sqrt{2}}\right)d\mu=\left[\int_{\xi-\sigma}^{\xi-\sigma+\delta}\Psi\left(t,\frac{\mu}{\sqrt{2}}\right)d\mu\right]^{\gamma}\left[\int_{\xi-\sigma}^{\xi-\sigma+\delta}\Psi\left(t,\frac{\mu}{\sqrt{2}}\right)d\mu\right]^{1-\gamma}\leq\frac{2^{-\frac{\gamma+3}{2}}\delta^{\gamma}}{(\nu t)^{\gamma/2}}.
\]
Bounding the right-hand side of \eqref{c333} with the above, we get 
\begin{equation}\label{c335}
	\int_{|\sigma|\leq 1}\left|[\Psi(t,\xi-\sigma+\delta)-\Psi(t,\xi-\sigma)]h_\epsilon'(\sigma)\right|d\sigma\leq \frac{2^{1-\gamma}\delta^{\gamma}}{(\nu t)^{(1+\gamma)/2}}\int_{|\sigma|\leq 1}h_\epsilon'(\sigma)d\sigma\leq \frac{2^{-\frac{\gamma+2}{2}}\delta^{\gamma}}{(\nu t)^{(1+\gamma)/2}},
\end{equation}
from which \eqref{wgthdhold} follows immediately from \eqref{c334}-\eqref{c335}. From \eqref{c333}, $\Psi(t,\mu)\leq (\nu t)^{-1/2}$, and $a=a^{\gamma}a^{1-\gamma}$ we get \eqref{c335n}.

To get uniform in $\epsilon$ H\"older estimates, we apply Duhamel's principle together with an integration by parts, to represent $\cv_{\epsilon}$ as 
\begin{align}
\cv_{\epsilon}(t,\xi)=&\intsol\Psi(t,\sigma)\cv_0(\xi-\sigma)d\sigma-\mu_1\int_0^tg(s)\intsol\partial_\sigma[\Psi(t-s,\xi-\sigma)]h_{\epsilon}(\sigma)\cv_{\epsilon}(s,\sigma)d\sigma ds\nonumber\\
&+\frac{C}{\beta}\int_0^tg(s)\intsol\Psi(t-s,\xi-\sigma)h_{\epsilon}'(\sigma)\cv_{\epsilon}(s,\sigma)d\sigma ds+\int_0^t\intsol\Psi(t-s,\xi-\sigma)d(s,\sigma)d\sigma ds.\label{c335n2}
\end{align}
For a given $\xi\in\mathbb{R}$, applying \eqref{wgthdholdgrad}-\eqref{c335n}, together with \eqref{uniformepsbd} guarantees the existence of a constant, 
\[
K_1=K_1(K_0,\xi,\beta,T,\gamma,\|g\|_{L^{\infty}},\mu_1,\|d\|_{L_t^{\infty}L_x^1})
\]
that blows up as $\gamma\rightarrow1^-$ (but is otherwise uniform in $\epsilon$) such that 
\begin{equation}\label{holdbdveps}
\sup_{t\in[0,T]}|\cv_{\epsilon}(t,\xi+\delta)-\cv_{\epsilon}(t,\xi)|\leq K_1\delta^{\gamma},\quad \gamma\in(0,1).
\end{equation}
To address Remark \ref{rmkforce}, note carefully that if the constant $K_0$ depends on $\|d\|_{L_t^{\infty}L_x^{1}}$ instead of $\|d\|_{L_t^{\infty}L_x^{\infty}}$ (which is the case if $d(t,\cdot)$ satisfies (S) from Theorem \ref{secondthm}), then so will $K_1$: from \eqref{c335n}
\[
\left|\int_0^t\intsol[\Psi(t-s,\xi-\sigma+\delta)-\Psi(t-s,\xi-\sigma)]d(s,\sigma)d\sigma ds\right|\leq \nu^{-\frac{\gamma+1}{2}}\delta^{\gamma}\int_0^t(t-s)^{-\frac{\gamma+1}{2}}\|d(s,\cdot)\|_{L^1}ds,
\] 
that is to say, the H\"older semi-norm of the volume potential corresponding to the external force is bounded in terms of $\|d\|_{L_t^qL_x^1}$, some $q>2$, regardless of symmetries, (contrary to the pointwise bound $K_0$ \eqref{uniformepsbd}).

Next, recall that $\Omega_\epsilon$ as defined in \eqref{defoeps} is odd in the spatial variable and solves
\begin{equation}\label{xtraeq}
	\partial_t\Omega_\epsilon(t,\xi)-4\nu\partial^2_\xi\Omega_{\epsilon}(t,\xi)-\mu_1g(t)h_{\epsilon}(\xi)\partial_\xi\Omega_\epsilon(t,\xi)=\frac{C}{\beta}g(t)\int_0^\xi h_{\epsilon}'(\eta)\partial_\eta\Omega_\epsilon(t,\eta)d\eta+a(t,\xi)
\end{equation}
for every $(t,\xi)\in(0,T]\times\rone$, together with 
\begin{align*}
	&\Omega_\epsilon(t,0)=0,\quad \forall t\in[0,T],\\
	&\Omega_\epsilon(0,\xi)=\Omega_0(\xi),\quad \forall \xi\in[0,\infty).
\end{align*}
Now, if we set $F_\epsilon(t,\xi)$ to be equal to the right-hand side of \eqref{xtraeq} then we must have  
\[
\Omega_{\epsilon}(t,\xi)=\intsol \Psi(t,\xi-\sigma)\Omega_0(\sigma)d\sigma+\int_0^t\intsol \Psi(t-s,\xi-\sigma)[\mu_1g(s)h_{\epsilon}(\sigma) \partial_\sigma\Omega_{\epsilon}(s,\sigma)+F_{\epsilon}(s,\sigma)]d\sigma ds.
\] 
Next, let us obtain estimates on $\partial_t\Omega_\epsilon$. Let $Q\in C([0,T]\times\rone)$ be given, and suppose that given $\xi\in\rone$, $\gamma\in(0,1)$, there exists an $M=M(\xi,\gamma,T)$ such that $|Q(s,\xi+\delta)-Q(s,\xi)|\leq M_{\xi}\delta^{\gamma}$ for any $\delta\in(0,1)$, and any $s\in[0,T]$. Let
\[
u(t,\xi):=\int_0^t\intsol \Psi(t-s,\xi-\sigma)Q(s,\sigma)d\sigma ds,
\]
and suppose that $Q$ doesn't grow faster than a Gaussian (in order for the integral to make sense): say $|Q(s,\xi)|\leq C_B(1+|\xi|^3)$. It is a classical fact that $u\in C_t^1C_x^2([0,T]\times\rone)$, solves $\partial_t u-\nu\partial^2_\xi u=Q(t,\xi)$, and one has the representation
\[
\partial_t u(t,\xi)=Q(t,\xi)+\int_0^t\intsol \partial_t\Psi(t-s,\xi-\sigma)(Q(s,\sigma)-Q(s,\xi))d\sigma ds,
\]
see, for instance, \cite[Chapter 5]{LSUbook1968}. We have 
\[
\left|\int_{|\sigma-\xi|\leq1/2} \partial_t\Psi(t-s,\xi-\sigma)(Q(s,\sigma)-Q(s,\xi))d\sigma\right|\lesssim M_\xi (t-s)^{\gamma/2-1},
\]
while 
\[
\left|\int_{|\sigma-\xi|\geq1/2} \partial_t\Psi(t-s,\xi-\sigma)(Q(s,\sigma)-Q(s,\xi))d\sigma\right|\lesssim \tilde {M}_\xi(t-s)^{-1}e^{\frac{-1}{(t-s)}}\lesssim\tilde{M}_\xi,
\]
for some $\tilde{M}_\xi>0$. From this, we clearly have $|\partial_t u(t,\xi)|\leq K_{\xi,T}$, for some $K_{\xi,T}>0$. It follows that courtesy of the uniform in epsilon bounds \eqref{uniformepsbd} and \eqref{holdbdveps}, Arzel\`a-Ascoli \cite[Theorem 4.44]{Folland1999book} guarantees the existence of an $\Omega\in C_t^1C_x^{1,\gamma}([0,T]\times\rone)$, any $\gamma\in(0,1)$, and a subsequence of $\{\Omega_\epsilon\}_{\epsilon>0}$ such that $\Omega_{\epsilon}\rightarrow \Omega$ uniformly on compact subsets of $[0,T]\times\rone$. The dominated convergence theorem tells that the limiting function satisfies
\begin{equation}\label{newlbl43}
	\Omega(t,\xi)=\intsol \Psi(t,\xi-\sigma)\Omega_0(\sigma)d\sigma+\int_0^t\intsol \Psi(t-s,\xi-\sigma)[g(s)h_0(\sigma) \partial_\sigma\Omega(s,\sigma)+F(s,\sigma)]d\sigma ds,
\end{equation}
where $h_0(\xi)$ is the odd extension of $\xi^{\beta}$ about $\xi=0$, and $F$ is the odd (in $\xi$ variable) function
\[
F(s,\xi)=g(s)C\int_0^{\xi}\eta^{\beta-1}\Omega(s,\eta)d\eta,\quad \xi\in\rone.
\]
It is clear then that $\Omega$ is in fact $\Omega\in C_t^1C_x^2((0,T]\times[0,\infty))$, courtesy of the above discussion regarding volume potentials, and so is a classical (point-wise) solution to \eqref{c31}. It is also clear that it inherits the qualitative properties of the approximating sequence $\{\Omega_\epsilon\}_{\epsilon>0}$ (meaning $\Omega(t,\cdot)$ is non-decreasing and concave on $[0,\infty)$ for every $t\in[0,T]$ if $\Omega_0$ and $a$ share those properties) and satisfies \eqref{BCcondomega}-\eqref{bddd2} as $\cv_\epsilon$ obeys (S) from Theorem \ref{secondthm} and satisfies the uniform in $\epsilon$ bounds \eqref{linftybd}-\eqref{implinftybd}. This concludes the proof.

\section{Proof of Theorem \ref{mainthm}}\label{pfmainthm}
Assume the hypotheses of Theorem \ref{mainthm}, that is: let $u_0$ be a given, smooth, divergence-free vector field, let $T_*>0$ be given, let $F:[0,T_*)\times\rd\rightarrow\rd$ be a smooth vector field, let $u:[0,T_*)\times\rd$ (and $p:[0,T_*)\times\rd\rightarrow\rone$) be the corresponding (smooth) solution to \eqref{NSE}, and suppose $u,F$ satisfy \eqref{hypoth}. Assume $u_0$ and $F$ satisfy one of conditions $A$ or $B$ stated in Theorem \ref{mainthm}, so that the solution $u$ inherits them. Let $T\in[0,T_*)$ be arbitrary, suppose $\Omega$ is a time-dependent modulus of continuity on $[0,T]$ (according to Definition \ref{deftimdepmod}), and assume that $\partial_\xi \Omega\in C([0,T]\times[0,\infty))$ (so that $\Omega(t,\cdot)$ is a Lipschitz modulus of continuity, and hence satisfies the requirement in Lemma \ref{gradbd}). Furthermore, suppose $u_0$ strictly obeys $\Omega(0,\cdot)$ as in Definition \ref{defobeymod}. Let us define 
\begin{equation}\label{deftau}
	\tau:=\sup\left\{t\in(0,T]:|u(s,x)-u(s,y)|<\Omega(s,|x-y|)\ \forall s\in[0,t],x\neq y\right\},\\
\end{equation}
and suppose for the moment that $\tau>0$. Our ultimate goal is to construct an $\Omega$ that forces $\tau=T$, which would imply that $|u(t,x)-u(t,y)|\leq \Omega(t,|x-y|)$ holds true for any $t\in[0,T]$ and any $x,y$ (and hence, any $t\in[0,T_*)$ as $T\in[0,T_*)$ was an arbitrary number). This would imply that $\|u(t,\cdot)\|_{L^{\infty}}\leq \|\Omega(t,\cdot)\|_{L^{\infty}}$: indeed, we have for any $x,y$, $|u(t,x)-u(t,y)|\leq \|\Omega(t,\cdot)\|_{L^{\infty}}$, and so in the periodic setting, there must always exist a $y$ such that $u(t,y)=0$ (owing to the fact that $u(t,\cdot)$ has zero average), from which the claimed bound follows by maximizing over $x$. In the whole space, since we assume $u(t,\cdot)$ vanishes at infinity, we first send $y$ to infinity to get $|u(t,x)|\leq \|\Omega(t,\cdot)\|_{L^{\infty}}$, so that the claimed bound follows by taking supremum in $x$. Bottom line is, in both cases, controlling the supremum norm of the solution $u$ boils down to bounding $\|\Omega(t,\cdot)\|_{L^{\infty}}$. Let us start by identifying the only possible scenario at time $\tau$ if $\tau<T$: one that is depicted by the solution vector-field violating its strict modulus of continuity away from the diagonal $x=y$, the so called ``breakthrough scenario'' first identified in \cite{KNS2008,KNV2007} in (see also \cite[Lemma~2.3]{Kiselev2011} for a time-dependent version). The proof of Proposition \ref{brkthru} below is similar to the proofs of Propositions 5.1 and 5.2 in our previous work \cite{Ibdah2022} (see also \cite[Lemma~2.3]{Kiselev2011}); nevertheless, we provide a proof since we slightly change the conditions that $\Omega$ need to satisfy.
\begin{prop}\label{brkthru}
 Let $T>0$ be given and let $u\in C([0,T]\times\rd)\cap C([0,T];W^{1,\infty}(\rd))$ be a vector field such that $u(t,\cdot)\in C^2(\rd)\cap W^{2,\infty}(\rd)$ for every $t\in[0,T]$ and for which either
 \begin{equation}\label{decay}
 \lim_{|x|\rightarrow\infty}|\nabla u(t,x)|=0,\quad \forall t\in[0,T],
 \end{equation}
or $u(t,\cdot)$ is periodic with some period $L>0$ in each direction. Let $\Omega\in C([0,T]\times[0,\infty))$ be such that $\Omega(t,\cdot)$ is a modulus of continuity for every $t\in[0,T]$, and suppose $\partial_\xi\Omega\in C([0,T]\times[0,\infty))$. Assume that 
\begin{enumerate}[label=\Alph*)]
\item There exists some $K\geq80$ such that 
\[
\inf_{t\in[0,T]}\Omega(t,\xi)\geq 3\sup_{t\in[0,T]}\|u(t,\cdot)\|_{L^{\infty}}, \quad\forall\xi\geq K,
\]
\item and for any $\xi>0$
\[
\inf_{t\in[0,T]}\Omega(t,\xi)>0.
\]
\end{enumerate}
Suppose $u(0,\cdot)$ strictly obeys $\Omega(0,\cdot)$, and let $\tau$ be as defined in \eqref{deftau}. It follows that $\tau$ is positive and if $\tau<T$, then there exists $x_0\neq y_0$ for which
\[
|u(\tau,x_0)-u(\tau,y_0)|=\Omega(\tau,|x_0-y_0|).
\]
\end{prop}
\begin{rmk}
Conditions $A$ and $B$ above are only required when working in the whole space. The easiest way to satisfy them is choosing $\Omega(t,\xi)=\tilde{\Omega}(t,\xi)+\delta\omega(\xi)$, where $\omega$ is an unbounded Lipschitz modulus of continuity, and $\tilde{\Omega}$ is bounded. This comes at the expense of having to deal with a forcing term of order $\delta$ in the condition \eqref{c21}. One then lets $\delta$ go to zero before using the $L^{\infty}$ bound of $\tilde{\Omega}$. This is done rigorously below after proving Proposition \ref{brkthru}.
\end{rmk}

\begin{proof}
We start by showing $\tau>0$. According to Lemma \ref{gradbd}, we have $\|\nabla u(0,\cdot)\|_{L^{\infty}}<\partial_\xi\Omega(0,0)$. For $(t,\xi)\in[0,T]\times[0,\infty)$, let us define $h(t,0):=\|\nabla u(t,\cdot)\|_{L^{\infty}}-\partial_\xi\Omega(t,0)$ and set
\[
h(t,\xi):=\|\nabla u(t,\cdot)\|_{L^{\infty}}-\frac{\Omega(t,\xi)}{\xi},\quad \xi>0.
\]
Note that by  our assumptions, we have $h\in C\left([0,T]\times[0,\infty)\right)$. Since  $h(0,0)<0$, continuity of $h$ guarantees the existence of an $\epsilon_0>0$ and a $\delta>0$ such that $h(t,\xi)<0$ whenever $(t,\xi)\in[0,\epsilon_0]\times[0,\delta]$. It follows that
\[
|u(t,x)-u(t,y)|\leq |x-y|\|\nabla u(t,\cdot)\|_{L^{\infty}}<|x-y|\left[\frac{\Omega(t,|x-y|)}{|x-y|}\right]=\Omega(t,|x-y|),
\]
provided $(t,|x-y|)\in[0,\epsilon_0]\times[0,\delta]$. This is true regardless whether we are in the periodic or whole space setting. Let us now assume $u(t,\cdot)$ is periodic with period $L>0$, which doesn't require $\Omega$ to satisfy conditions $A)$ and $B)$ above. We define $\mathcal{A}:=\left\{(x,y)\in[0,L]^d\times\mathbb{R}^d:|x-y|\in[\delta,2L\sqrt{d}]\right\}$, where $L>0$ is the period of $\theta(t,\cdot)$, and note that since the set $[0,\epsilon_0]\times \mathcal{A}$ is compact, the function $R(t,x,y):=|u(t,x)-u(t,y)|-\Omega(t,|x-y|)$ is uniformly continuous on it, and as $R(0,x,y)<0$, the same must be true on $[0,\epsilon]\times \mathcal{A}$, some $\epsilon\in(0,\epsilon_0]$. As $u(t,\cdot)$ is $L$ periodic and $\Omega(t,\cdot)$ is non-decreasing, we must have $\tau\geq\epsilon>0$.

Let us now assume $u(t,\cdot)$ satisfies \eqref{decay} rather than being periodic, and let us assume $\Omega$ satisfies $A)$ and $B)$. It follows that for any $(t,|x-y|)\in[0,\epsilon_0]\times[K,\infty)$, we have 
\[
|u(t,x)-u(t,y)|\leq2\sup_{t\in[0,T]}\|u(t,\cdot)\|_{L^{\infty}}<\Omega(t,|x-y|).
\]
Thus, it remains to show the existence of an $\epsilon_1\in(0,\epsilon_0]$ such that 
\[
|u(t,x)-u(t,y)|<\Omega(t,|x-y|),\quad \forall (t,|x-y|)\in[0,\epsilon_1]\times [\delta,K].
\]
To that extent, for any given $\mu>0$, \eqref{decay} tells us that we can find a large enough $K_1$ such that if we define the set $\mathcal{A}:=\{x\in\rd:|x|\geq K_1\}$, then $\|\nabla u(0,\cdot)\|_{L^{\infty}(\mathcal{A})}<\mu$. Since we assumed that $u\in C([0,T];W^{1,\infty}(\rd))$, the function $\gamma(t):=\|\nabla\theta(t,\cdot)\|_{L^{\infty}(\mathcal{A})}$ is continuous as well. Hence, we must have $\|\nabla u(t,\cdot)\|_{L^{\infty}(\mathcal{A})}<\mu$ for every $t\in[0,\epsilon_1]$, some $\epsilon_1=\epsilon_1(\mu)\in(0,\epsilon_0]$. Thus, we if we set 
\[
\kappa:=\inf_{t\in[0,\epsilon_0]}\Omega(t,\delta)>0,
\]
we can chose $K_1$ large enough such that 
\begin{equation}\label{small}
|\nabla u(t,x)|<\frac{\kappa}{2K},\quad \forall|x|\geq K_1, \quad t\in[0,\epsilon_1].
\end{equation}
Next, we split the set $\mathcal{B}:=\left\{(x,y)\in\rd\times\rd:|x-y|\in[\delta,K]\right\}$, into $\mathcal{B}_1\cup\mathcal{B}_2$ where $\mathcal{B}_2$ is the complement of $\mathcal{B}_1:=\left\{(x,y)\in\mathcal{B}:\min\{|x|,|y|\}>K+K_1\right\}$.  By the mean value theorem, whenever $(t,x,y)\in[0,\epsilon_1]\times\mathcal{B}_1$, we must have, for some $\sigma\in(0,1)$
\[
|u(t,x)-u(t,y)|\leq |x-y||\nabla u(t,\sigma(x-y)+y)|< \Omega(t,\delta)\leq \Omega(t,|x-y|),
\]
where we used $|y+\sigma(x-y)|\geq|y|-|x-y|\geq K_1$, $|x-y|\in[\delta,K]$, and \eqref{small} in the second inequality, while we used the fact that $\Omega(t,\cdot)$ is non-decreasing in the third one. Finally, since $\mathcal{B}_2$ is compact, we can certainly obtain a small enough $\epsilon>0$ such that the strict modulus of continuity is obeyed for $(t,x,y)\in[0,\epsilon]\times\mathcal{B}_2$. The same argument can be repeated if we assume for the sake of contradiction that $\tau<T$, with $u$ strictly obeying $\Omega$ at time $\tau$ and every $x\neq y$.
\end{proof}
Let us start by briefly addressing the forcing term $F$ (we will make more comments at the end of this section): it is assumed to be smooth in order to guarantee that the solution $u$ is $C_t^1C_x^2$. However, notice that the a-priori bound \eqref{mainbd} depends only on the $L_t^1L_x^{\infty}$ norm of $F$. We explain how to obtain such a-priori bound below, but let us for now assume the existence of a smooth, bounded, non-decreasing, and concave $\chi:[0,\infty)\rightarrow[0,\infty)$ satisfying $\chi(0)=0$ and a smooth $f:[0,T]\rightarrow[0,\infty)$ such that
\begin{equation}\label{holdcondf}
	|F(t,x)-F(t,y)|\leq f(t)\chi(|x-y|),\quad \forall (t,x,y)\in [0,T]\times\rd\times\rd.
\end{equation}
That being said, we now prove that if $\Omega\in C_t^1C_{\xi}^2((0,T]\times (0,\infty))$ satisfies
\begin{equation}\label{c44}
	\partial_t\Omega(t,\xi)-4\nu\partial^2_\xi\Omega(t,\xi)-g(t)\xi^{\beta}\partial_\xi\Omega(t,\xi)>\frac{C_d}{1-\beta}h(t)\int_0^{\xi}\frac{\partial_\eta\Omega(t,\eta)}{\eta^{1-\beta}}d\eta+f(t)\chi(\xi)
\end{equation}
for every $(t,\xi)\in(0,T]\times(0,\infty)$ on top of satisfying all the hypotheses required in Proposition \ref{brkthru}, then $\tau$ as defined in \eqref{deftau} is equal to $T$. For convenience, we summarize the conditions $\Omega$ needs to satisfy:
\begin{enumerate}
	\item Has to be a classical (pointwise, $C_t^1C_\xi^2$) solution to \eqref{c44}.
	\item Both $\Omega$ and $\partial_\xi\Omega$ have to be continuous on $[0,T]\times[0,\infty)$.
	\item For every fixed $t\in[0,T]$, $\Omega(t,0)=0$ and $\partial_\xi\Omega(t,\cdot)\geq0$.
	\item $u_0$ strictly obeys $\Omega(0,\xi)$ and $\partial^2_\xi\Omega(t,0^+)=-\infty$.
	\item Satisfies A) and B) from Proposition \ref{brkthru}.
\end{enumerate}
Assuming conditions 2-5 are met, we are guaranteed the positivity of $\tau$ by Proposition \ref{brkthru}, so let us suppose for the sake of contradiction that $\tau<T$, meaning $u$ must ``break'' the strict inequality at some $x_0\neq y_0$, regardless whether we are in the periodic setting or otherwise: $|u(\tau,x_0)-u(\tau,y_0)|=\Omega(\tau,|x_0-y_0|)$. We show this is not possible provided condition 1 is met. The analysis will be easier to carry out if we work with scalar functions and if $x_0-y_0=\xi e_1$, some $\xi>0$. This can be done without much trouble courtesy of the rotation invariance of \eqref{NSE}. Namely, we choose a rotational matrix $\mathcal{R}$ such that $\mathcal{R}(x_0-y_0)=\xi e_1$, where $\xi:=|x_0-y_0|$ and $e_1$ is the standard unit vector in the first coordinate, and we define $\tilde x_0:=\mathcal{R}x_0$, $\tilde y_0:=\mathcal{R}y_0$, $\widetilde u(t,x):=\mathcal{R} u(t,\mathcal{R}^{-1}x)$, $\widetilde p(t,x):=p(t,\mathcal{R}^{-1}x)$, and $\tilde{F}(t,x):=\mathcal{R}F(t,\mathcal{R}^{-1}x)$. It follows that there exists some unit vector $e\in\mathbb{S}^{d-1}$, possibly depending on the time $\tau$, but is otherwise a constant on $[0,\tau]\times\rd$ such that $\left[\widetilde u(\tau,\tilde x_0)-\widetilde u(\tau,\tilde y_0)\right]\cdot e=\Omega(\tau,\xi)$.
We define a scalar $\theta(t,x):=e\cdot\widetilde u(t,x)$ and study its evolution:
\begin{align}
	&\partial_t\theta(t,x)-\nu\Delta\theta(t,x)+(\widetilde u\cdot\nabla)\theta(t,x)+e\cdot\nabla\tilde{p}(t,x)=e\cdot\tilde{F}(t,x),&&\forall (t,x)\in[0,\tau]\times\rd,\label{evolthetadd}\\
	&\theta(\tau,\tilde x_0)-\theta(\tau,\tilde y_0)=\Omega(\tau,\xi), &&(\tilde x_0-\tilde y_0)=\xi e_1\label{violationDD}.
\end{align}
Since $\mathcal{R}$ is an orthogonal matrix, and since moduli of continuity are not sensitive to distance preserving maps, we have
\begin{align}
&|\theta(t,x)-\theta(t,y)|<\Omega(t,|x-y|), &&\forall t\in [0,\tau),\quad x\neq y,\label{modtheta}\\
&|\widetilde u(t,x)-\widetilde u(t,y)|\leq g(t)|x-y|^{\beta}, &&\forall (t,x,y)\in [0,\tau]\times\rd\times\rd.\label{holdcondb}
\end{align}
Next, we set $\rho(t):=\theta(t,\tilde x_0)-\theta(t,\tilde y_0)-\Omega(t,\xi)$, and note that courtesy of \eqref{modtheta}, we have $\rho(t)<0$ for any $t\in[0,\tau)$, so that if we choose $\Omega$ in a way that guarantees $\rho'(\tau)<0$ then \eqref{violationDD} is absurd, meaning $\tau=T$. Using the PDE \eqref{evolthetadd}, which holds true in the pointwise sense, we have 
\begin{align}
\rho'(\tau)=&\nu\Delta \theta(\tau,\tilde x_0)-\nu\Delta \theta(\tau,\tilde y_0)-\partial_t\Omega(\tau,\xi)+(\widetilde u\cdot\nabla)\theta(\tau,\tilde y_0)-(\widetilde u\cdot\nabla)\theta(\tau,\tilde x_0)\nonumber\\
&+e\cdot\left[\nabla\tilde{p}(\tau,\tilde{y}_0)-\nabla\tilde{p}(\tau,\tilde{x}_0)+\tilde{F}(\tau,\tilde{x}_0)-\tilde{F}(\tau,\tilde{y}_0)\right]\label{c40}.
\end{align}
From \eqref{modtheta} and continuity we see that at time $\tau$, $\theta$ still obeys $\Omega$, albeit not (necessarily) strictly. Nevertheless, we may use Lemma \ref{localmaxprincp} (as it doesn't require strict ``obedience'') to extract local dissipation from the Laplacian as well as to deal with the transport term. As $\tilde{x}_0\neq\tilde{y}_0$, inequality \eqref{localdissp} tells us that 
\begin{equation}\label{c41}
\Delta \theta(\tau,\tilde x_0)-\Delta \theta(\tau,\tilde y_0)\leq4\partial^2_\xi\Omega(\tau,\xi).
\end{equation}
Moreover, \eqref{derv} forces all first order derivatives of $\theta$ in directions other than $e_1$ to be zero while $\partial_1\theta(\tau,\tilde x_0)=\partial_1\theta(\tau,\tilde y_0)=\partial_\xi\Omega(\tau,\xi)\geq0$,
since $\Omega(t,\cdot)$ is non-decreasing. Thus, using \eqref{holdcondb}
\begin{equation}\label{c42}
(\widetilde u\cdot\nabla)\theta(\tau,\tilde y_0)-(\widetilde u\cdot\nabla)\theta(\tau,\tilde x_0)=(\widetilde u_1(\tau,\tilde{y}_0)-\widetilde u_1(\tau,\tilde{x}_0))\partial_1\theta(\tau,\tilde x_0)\leq g(\tau)\xi^{\beta}\partial_\xi\Omega(\tau,\xi).
\end{equation}
Next, we recall that $\nabla\tilde{p}$ has the same continuity estimates as $\nabla p$, and since $F$ is divergence-free, an application of Lemma \ref{pressest} with $b=u$, together with the fact that we have two moduli of continuity for $u$ at our disposal (the assumed H\"older one \eqref{holdcondb} as well as $\Omega$) yields
\begin{equation}\label{c43}
	|\nabla\tilde{p}(\tau,\tilde{y}_0)-\nabla\tilde{p}(\tau,\tilde{x}_0)|\leq C_d g(t)\left[\int_0^{\xi}\Omega(\tau,\eta)\eta^{\beta-2}d\eta+\frac{2\Omega(\tau,\xi)}{(1-\beta)\xi^{1-\beta}}\right]\leq\frac{2C_dg(t)}{1-\beta}\int_0^\xi\eta^{\beta-1}\partial_\eta\Omega(\tau,\eta)d\eta,
\end{equation}
where we used $\Omega(t,\xi)\geq0$, integrated by parts, and utilized $\Omega(t,0)=0$ together with $0\leq\partial_\xi\Omega(t,\cdot)<\infty$. Finally applying the continuity estimate we have on the forcing term \eqref{holdcondf} and bounding the right-hand side of \eqref{c40} by \eqref{c41}-\eqref{c43}, we get  
\[
\rho'(\tau)\leq 4\nu\partial^2_\xi\Omega(\tau,\xi)-\partial_t\Omega(\tau,\xi)+g(\tau)\xi^{\beta}\partial_\xi\Omega(\tau,\xi)+\frac{C_dg(\tau)}{1-\beta}\int_0^\xi\eta^{\beta-1}\partial_\eta\Omega(\tau,\eta)d\eta+f(\tau)\chi(\xi).
\]
Thus, if $\Omega\in C_t^1C_{\xi}^2((0,T]\times (0,\infty))$ satisfies \eqref{c44}  for every $(t,\xi)\in(0,T]\times(0,\infty)$, condition 1, together with conditions 2-5, then $\rho'(\tau)<0$, meaning $\tau=T$. 

We now construct such an $\Omega$, starting with defining
\begin{align*}
&\omega(\xi):=\frac{\xi}{1+\xi^{(1+\beta)/2}},\quad \xi\geq0,\\
&\Omega_0(\xi):=B_1\tanh(B_0\xi), \quad \xi\geq0,
\end{align*}
where $B_0$ and $B_1$ are both chosen large enough, depending on $\|u_0\|_{L^{\infty}}$ and $\|\nabla u_0\|_{L^{\infty}}$ such that $u_0$ strictly obeys $\Omega_0$. Utilizing the concavity of $\tanh$, we invite the reader to verify that choosing 
\[
B_0:=2\frac{\|\nabla u_0\|_{L^{\infty}}}{\|u_0\|_{L^{\infty}}},\quad B_1:=3\frac{\|u_0\|_{L^{\infty}}}{\tanh(1)},
\]
forces $u_0$ to obey $\Omega_0$ strictly. It is clear that $\omega\in C^1[0,\infty)\cap C^2(0,\infty)$ with $\omega''(0^+)=-\infty$. Let us set $C:=C_d/(1-\beta)\geq1$, and for any given $\epsilon\in(0,1]$, we choose $\mu_1:=C\beta^{-1}/\epsilon$. Define  
\[
\varphi(\xi):=\int_0^{\xi}\eta^{\beta-1}\omega'(\eta)d\eta,
\]
and notice that $\varphi$ is bounded uniformly: for small $\xi$, $\omega'\approx 1$, while for large $\xi$, $\omega'$ decays like $\xi^{-(\beta+1)/2}$. For any $\delta\in(0,1)$, let $\tilde{\Omega}$ be the solution to
\begin{align}
	&\partial_t\tilde{\Omega}(t,\xi)-4\nu\partial^2_\xi\tilde{\Omega}(t,\xi)-\mu_1g(t)\xi^{\beta}\partial_\xi\tilde{\Omega}(t,\xi)= Cg(t)\int_0^\xi\eta^{\beta-1}\partial_\eta\tilde{\Omega}(t,\eta)d\eta+f(t)\chi(\xi)+\delta \left(\beta+C\right)g(t) \varphi(\xi),\nonumber\\
	&\tilde{\Omega}(t,0)=0,\quad \forall t\in[0,T],\nonumber\\
	&\tilde{\Omega}(0,\xi)=\Omega_0(\xi),\quad \forall \xi\in[0,\infty),\label{c45}
\end{align}
given to us by Theorem \ref{thmbuildmod} (keeping in mind Remark \ref{rmkforce} about the smoothness of the forcing term(s)). As $u$ is smooth on $[0,T]$, $g$ is uniformly bounded and continuous on $[0,T]$, so that $\tilde{\Omega}$ has uniformly bounded first and second derivatives. The initial data in \eqref{c45} obviously obeys the symmetries listed in the hypotheses of Theorem \ref{thmbuildmod}, and so does the forcing term. Putting all this together with the fact that $\omega$ is a modulus of continuity according to Definition \ref{defmod} that happens to be unbounded, conditions 2 through 5 are met. We now claim that $\Omega$ satisfies \eqref{c44} (and hence, condition 1, as we already know $\Omega\in C_t^1C_\xi^2([0,T]\times(0,\infty))$). To see this, first note that since $\partial_\xi\tilde{\Omega}\geq0$ and $\mu_1\geq1$, we must have 
\begin{align*}
\partial_t\tilde{\Omega}(t,\xi)-4\nu\partial^2_\xi\tilde{\Omega}(t,\xi)=&\mu_1g(t)\xi^{\beta}\partial_\xi\tilde{\Omega}(t,\xi)+Cg(t)\int_0^\xi\eta^{\beta-1}\partial_\eta\tilde{\Omega}(t,\eta)d\eta+f(t)\chi(\xi)+\delta \left(\beta+C\right)\varphi(\xi)g(t)\\
&\geq g(t)\xi^{\beta}\partial_\xi\tilde{\Omega}(t,\xi)+Cg(t)\int_0^\xi\eta^{\beta-1}\partial_\eta\tilde{\Omega}(t,\eta)d\eta+f(t)\chi(\xi)+\delta \left(\beta+C\right)\varphi(\xi)g(t).
\end{align*}
Furthermore, since $\omega$ is strictly concave and $\omega'(\xi)\leq 1$, we have 
\[
\left(\beta+C\right)\varphi(\xi)=\left(\beta+C\right)\int_0^{\xi}\eta^{\beta-1}\omega'(\eta)d\eta>4\nu\omega''(\xi)+\xi^{\beta}\omega'(\xi)+C\int_0^\xi\eta^{\beta-1}\omega'(\eta)d\eta
\]
from which it becomes clear that $\Omega:=\tilde{\Omega}+\delta\omega$ satisfies \eqref{c44}. In other words, $\Omega$ is a time-dependent modulus of continuity that satisfies the hypothesis of Proposition \ref{brkthru}. It follows that $\tau=T$, and so for any $x,y$,
\[
|u(t,x)-u(t,y)|\leq\tilde{\Omega}(t,|x-y|)+\delta\omega(|x-y|),\quad \forall t\in[0,T].
\]
Sending $\delta$ to zero from above, using bound \eqref{maindbomega}, and noting that $T$ was any number in $[0,T_*)$ gives us bound \eqref{mainbd}, by recalling that $\lambda=1+\epsilon$. The bound claimed in Remark \ref{rmkscalednorm} follows from using bound \eqref{scaledopnorm} instead of \eqref{goodopnorm} in the proof of Proposition \ref{keyprop} (and by extension, in the proof of Theorem \ref{thmbuildmod}). 

Let us finally turn to the definition of the function $\chi$: the continuity estimate for the forcing term $F$ from \eqref{holdcondf}. Let us suppose that $\|F(t,\cdot)\|_{L^{\infty}}\leq f(t)$, suppose $F$ is smooth, and without loss in generality, assume $f\geq\kappa>0$ uniformly. It follows that if we set $K:=1+\|\nabla F\|_{L_t^{\infty}L_x^{\infty}}$, then for any $\alpha\in(0,1)$, we have $|F(t,x)-F(t,y)|\leq 2(K/\kappa)^{\alpha}[f(t)]^{1-\alpha}\chi(|x-y|)$ where
\[
\chi(\xi):=
\begin{cases}
 	\xi^{\alpha}, &\xi\in[0,1],\\
 	1, &\xi>1.
\end{cases}
\]
The a-priori bound \eqref{maindbomega} is independent of $\alpha$ (keeping in mind Remark \ref{rmkforce}), thus, we may let $\alpha\rightarrow0^+$ and then $\kappa\rightarrow 0^+$ to get \eqref{mainbd}. One can very easily smoothen out the singularity in $\chi'$ if necessary, see Remark \ref{rmkforce}.

\section*{Acknowledgments}
The author thanks Theodore Drivas, Tarek Elgindi, Gautam Iyer, Yu Gu, Alexander Kiselev, Huy Nguyen, and Stavros Papathanasiou for useful discussions, and Peter Constantin for helpful comments on an earlier version. The author extends a special thanks to Tarek for carefully reading an earlier version of this manuscript and for providing us with the counter-example in Appendix \ref{appce}.

\titleformat{\section}{\normalfont\Large\bfseries}{\appendixname~\thesection:}{1em}{}
\begin{appendices}
\section{$L^p$ Bounds and higher regularity}\label{applpbds}
In this section we obtain $L^p$ bounds on solutions to \eqref{c33} and then use them to control sub-critical H\"older norms in terms of supercritical ones. Notice that when $\lambda\in[0,1]$ in Proposition \ref{keyprop}, we have $L^{\infty}$ and $L^1$ bounds. Those can be interpolated to get any $L^p$ bound, so we focus on the more interesting case when $\lambda\in[1,2]$. We have the following:
\renewcommand\theprop{A}
\begin{prop}\label{proplpbds}
Let $\cv_0:\rone\rightarrow\rone$ be a given smooth function. Let $T>0$, $\mu_1\in(0,\infty)$, $\mu_2\in[0,\infty)$ be given constants, and let $g:[0,T]\rightarrow[0,\infty)$ be continuous and $d:[0,T]\times\rone\rightarrow\rone$ be smooth. Let $\cv_\epsilon$ solve 
\begin{equation}\label{a1}
	\partial_t\cv_\epsilon(t,\xi)-4\nu\partial^2_\xi\cv_{\epsilon}(t,\xi)-\mu_1g(t)h_{\epsilon}(\xi)\partial_\xi\cv_{\epsilon}(t,\xi)=\mu_2g(t)h_{\epsilon}'(\xi)\cv_{\epsilon}(t,\xi)+d(t,\xi),
\end{equation}
with $\cv_\epsilon(0,\xi)=\cv_0(\xi)$. Here $h_\epsilon$ is the smooth approximation to the odd extension of $\xi^{\beta}$ introduced in \S\ref{mainsec}. Suppose that $\lambda:=\mu_2/\mu_1\in[1,2]$, and let $G$ be as defined in \eqref{defG2}. Then the following a-priori bound holds true for every $t\in[0,T]$ and every $p\in[1,\infty)$ such that $p\lambda\leq 2$:
\begin{equation}\label{lpbd}
	\|\cv_\epsilon(t,\cdot)\|_{L^{p}}\leq \left[\|\cv_0\|_{L^{p}}+\int_0^t\|d(s,\cdot)\|_{L^p}ds\right]G^{\lambda-1/p}(0,t).
\end{equation} 
\end{prop}
\begin{proof}
The proof is essentially a slightly more cumbersome rendition of the one used to rigorously justify Proposition \ref{keyprop}, except we do $L^p$ estimates rather than $L^1$, employing Minkowski's integral inequality a couple of times. Dropping the subscript $\epsilon$ for notational convenience, recall that the Feynman-Kac formula \eqref{c317} permits us to represent the solution as $\cv=\cv_1+\cv_2$ where
\[
\cv_1(t,\xi):=\mathbb{E}\left[\cv_0(\ca_{t,\xi})\cb^{\lambda}(t,\xi)\right],\quad \cv_2(t,\xi):=\mathbb{E}\left[\cb^{\lambda}(t,\xi)\int_0^td(s,\Phi(s,\ca_{t,\xi}))\cb^{-\lambda}(s,\Phi(s,\ca_{t,\xi}))ds\right],
\]	
where $\ca_{t,\xi}=\Phi^{-1}(t,\xi)$, with $\Phi$ being the stochastic flow and $\cb=\partial_\xi\ca$ solves \eqref{c318}. An application of Minkowski's integral inequality followed by a change in variable $\sigma=\ca_{t,\xi}$ tells us that 
\[
\|\cv_1(t,\cdot)\|_{L^p}\leq \mathbb{E}\left[\left(\intsol\cb^{p\lambda}(t,\xi)|\cv_0(\ca_{t,\xi})|^pd\xi\right)^{1/p}\right]=\mathbb{E}\left[\left(\intsol\cb^{p\lambda -1}(t,\Phi_{t,\sigma})|\cv_0(\sigma)|^pd\sigma\right)^{1/p}\right].
\]
Applying H\"older's inequality followed by Fubini-Tonelli then yields
\[
\|\cv_1(t,\cdot)\|_{L^p}\leq\left(\intsol\mathbb{E}[\cb^{\lambda p-1}(t,\Phi_{t,\sigma})]|\cv_0(\sigma)|^pd\sigma\right)^{1/p}.
\]
Since $1\leq p\lambda\leq 2$, we have $0\leq p\lambda-1\leq1$, so that we may apply H\"older's inequality once more to get $\mathbb{E}[\cb^{\lambda p-1}(t,\Phi_{t,\sigma})]\leq \left(\mathbb{E}[\cb(t,\Phi_{t,\sigma})]\right)^{p\lambda-1}\leq \|\bar{\cb}(t,\cdot)\|_{L^{\infty}}^{p\lambda-1}$, with $\bar{\cb}=\mathbb{E}[\cb]$. Recalling the bound \eqref{cc319}: $\|\bar{\cb}(t,\cdot)\|_{L^{\infty}}\leq G(0,t)$ we get that $\|\cv_1(t,\cdot)\|_{L^p}\leq \|\cv_0\|_{L^p}G^{\lambda-1/p}(0,t)$. The volume potential $\cv_2$ is controlled similarly: Minkowski's integral inequality applied twice followed by the change in variable $\sigma=\ca_{t,\xi}$ and Fubini-Tonelli leads to 
\begin{align*}
\|\cv_2(t,\cdot)\|_{L^p}\leq&\mathbb{E}\left[\left(\intsol\left(\int_0^t\cb^{\lambda}(t,\xi)\left|d(s,\Phi(s,\ca_{t,\xi}))\right|\cb^{-\lambda}(s,\Phi(s,\ca_{t,\xi}))ds\right)^pd\xi\right)^{1/p}\right]\\
&\leq\mathbb{E}\left[\int_0^t\left(\intsol\cb^{p\lambda}(t,\xi)\left|d(s,\Phi(s,\ca_{t,\xi}))\right|^p\cb^{-p\lambda}(s,\Phi(s,\ca_{t,\xi}))d\xi\right)^{1/p}ds\right]\\
&=\mathbb{E}\left[\int_0^t\left(\intsol\cb^{p\lambda-1}(t,\Phi_{t,\sigma})\left|d(s,\Phi(s,\sigma))\right|^p\cb^{-p\lambda}(s,\Phi(s,\sigma))d\sigma\right)^{1/p}ds\right]\\
&=\int_0^t\mathbb{E}\left[\left(\intsol\cb^{p\lambda-1}(t,\Phi(t,\ca_{s,\mu}))|d(s,\mu)|^pd\mu\right)^{1/p}\right]ds.
\end{align*}
Invoking one more change in variable $\mu:=\Phi(s,\sigma)$ (or equivalently, $\sigma=\ca(s,\mu)$, meaning $d\sigma=\cb(s,\mu)d\mu$), using the fact that $p\lambda\in[1,2]$ by assumption, and the almost sure lower bound $\cb\geq1$ yields
\[
\intsol\cb^{p\lambda-1}(t,\Phi_{t,\sigma})\left|d(s,\Phi(s,\sigma))\right|^p\cb^{-p\lambda}(s,\Phi(s,\sigma))d\sigma\leq \intsol \cb^{p\lambda-1}(t,\Phi(t,\ca_{s,\mu}))|d(s,\mu)|^pd\mu,
\]
which upon a further application of H\"older's inequality followed by Fubini-Tonelli renders
\begin{align*}
\|\cv_2(t,\cdot)\|_{L^p}\leq\int_0^t\left(\intsol \mathbb{E}[\cb^{p\lambda-1}(t,\Phi(t,\ca_{s,\mu}))]|d(s,\mu)|^pd\mu\right)^{1/p}ds.
\end{align*}
As $p\lambda\in[1,2]$, we have $p\lambda-1\in[0,1]$, so that a final application of H\"older's inequality yields $\mathbb{E}[\cb^{p\lambda-1}(t,\Phi(t,\ca_{s,\mu}))]\leq\left(\mathbb{E}[\cb((t,\Phi(t,\ca_{s,\mu}))]\right)^{p\lambda-1}\leq \|\bar{\cb}(t,\cdot)\|_{L^{\infty}}^{p\lambda-1}$. Making use of \eqref{cc319} one more time allows us to conclude the proof:
\[
\|\cv_2(t,\cdot)\|_{L^p}\leq G^{\lambda-1/p}(0,t)\int_0^t\|d(s,\cdot)\|_{L^p}ds,
\]
\end{proof}
Having obtained $L^p$ bounds on $\cv$, those can be used to get $C_x^{0,\gamma}$ estimates on the solution $u$. Indeed, recall that by design, we have 
\begin{equation}\label{appholdbd}
	|u(t,x)-u(t,y)|\leq\Omega(t,|x-y|)=\int_0^{|x-y|}\cv(t,\eta)d\eta\leq |x-y|^{\frac{1}{q}}\|\cv(t,\cdot)\|_{L^p}\leq \|\cv_0\|_{L^p}|x-y|^{\frac{1}{q}}G^{\lambda-\frac{1}{p}}(0,t),
\end{equation}
where $\frac{1}{p}+\frac{1}{q}=1$, $1\leq p\lambda\leq2$. Recall that we can make $\lambda$ as close as we want to 1: $\lambda=1+\epsilon$, with $\epsilon\in(0,1)$ being arbitrary. One can use bound \eqref{appholdbd} to gain bounds on $u$ in $L_t^{p_0}C_x^{0,1/q}$ in terms of the $L_t^{p_1}C_x^{0,\beta}$ semi-norm: $[u(t,\cdot)]_{C_x^{0,1/q}}\lesssim G^{\epsilon+1/q}(0,t)$. For instance, if we wish to maximize the H\"older exponent $1/q$, the condition $p\lambda\leq 2$ tells us that $1/q$ can be at most $(1-\epsilon)/2$ (no larger than 1/2). Choosing this value of $q$ (with $\lambda=1+\epsilon$) we see that $[u(t,\cdot)]_{C_x^{0,1/q}}\leq \|\cv_0\|_{L^p}G^{\frac{\epsilon+1}{2}}(0,t)$. Choosing $p_0=2/(1+\epsilon)$, it is clear that $u\in L_t^{p_0}C_x^{0,1/q}$ provided $u\in L_t^1C_x^{0,\beta}$. Notice that the space $L_t^{p_1}C_x^{0,1/q}$, with $q=2/(1-\epsilon)$, is critical when $p_1=4/(3-\epsilon)$. Clearly $2/(1+\epsilon)>4/(3-\epsilon)$ when $\epsilon\in(0,1)$: a subcritical H\"older norm is controlled by a supercritical one. Of course one can trade of spatial regularity with temporal ones in \eqref{appholdbd} according to the rule $G\in L^p(0,T)$ provided $u\in L_t^pC_x^{0,\beta}$. The power $\epsilon+1/q$ is always less than 1, hence it allows us to break the criticality barrier. Finally, we point out that if we use $\tilde{G}$ as defined in Remark \ref{rmkscalednorm}, then \eqref{appholdbd} scales appropriately (as it should).
\section{A counter-example and sign changing solutions}\label{appce}
The purpose of this section is to show that the sign (and presence) of the drift term is absolutely crucial for breaking the criticality barrier: we show that Lemma \ref{lemspcase} (which is central to the entire analysis) is false if we flip the sign of (or eliminate) the transport term. We also show that solutions to the associated non-local one-dimensional problem in general do not obey a minimum principle. To that end, it is slightly more convenient to set $\nu=1$ and to approximate $\xi^{\beta}$ by means other than standard mollifiers. The following example in Proposition \ref{propce} is a slight modification of the one provided to us by Elgindi \cite{ElgindiPrivateComm}. Its application to constructing sign changing solutions is due to us.
\begin{prop}\label{propce}
Given $\epsilon\in(0,1)$, let 
\begin{equation}\label{detildeh}
	\bar{h}_{\epsilon}(\xi):=\int_0^{\xi}\frac{\beta}{\epsilon^{\frac{1-\beta}{2}}+|\eta|^{1-\beta}}d\eta,\quad \xi\in\rone.
\end{equation}
Then given any $p\in(0,2/(1+\beta))$ and $\mu\geq1$, there exists a non-negative smooth $g_\epsilon$ with $\|g_\epsilon\|_{L^p(0,T)}\leq1$ for any $T>0$ and an even solution to
\begin{equation}\label{eqvps2}
	\partial_t\cv_\epsilon(t,\xi)-4\partial^2_\xi\cv_\epsilon(t,\xi)+\mu_1g_\epsilon(t)\bar{h}_{\epsilon}(\xi)\partial_\xi\cv_\epsilon(t,\xi)=g_\epsilon(t)\bar{h}_{\epsilon}'(\xi)\cv_\epsilon(t,\xi),\quad \cv_\epsilon(0,\xi)=\exp\left(\frac{-\xi^2}{\epsilon}\right),
\end{equation}
such that for any $T\geq\epsilon$, there exists a positive constant $C$ (independent of $\epsilon$) for which 
\[
\int_0^T\|\cv_\epsilon(t,\cdot)\|_{L^{\infty}}dt\geq C\epsilon^{-1}.
\]
In particular, it simply is not possible to bound $\|\cv_\epsilon(t,\cdot)\|_{L_t^qL_x^{\infty}}$ uniformly in $\epsilon$ for any $q\geq1$, not unless $p\geq2/(1+\beta)$, i.e., not unless we make a critical or subcritical assumption.
\end{prop}
\begin{proof}
It is clear that $\bar{h}_{\epsilon}\in C^1(\rone)$ converges to the odd extension of $\xi^\beta$, and that it is concave on $(0,\infty)$. We ask the reader to readily verify that if we define
\[
g_\epsilon(t):=\frac{1}{\epsilon^p}e^{-t/\epsilon},\quad H(y):=\int_0^{y}\frac{\beta}{1+|\eta|^{1-\beta}}d\eta,\quad y\in\rone,
\]
and if $\cw_\epsilon$ solves 
\begin{equation}\label{eqweps}
	\partial_s\cw_\epsilon(s,y)-4\partial^2_y\cw_\epsilon(s,y)+\mu_1\epsilon^{\left(\frac{\beta+1}{2}-\frac{1}{p}\right)}g(s)H(y)\partial_y\cw_\epsilon(s,y)=\epsilon^{\left(\frac{\beta+1}{2}-\frac{1}{p}\right)}g(s)H'(y)\cw_\epsilon(s,y),
\end{equation}
with $\cw_{\epsilon}(0,y)=e^{-y^2}$, then $\cv_\epsilon(t,\xi):=\cw_\epsilon(t/\epsilon,\xi/\sqrt{\epsilon})$ solves \eqref{eqvps2}, so that our task reduces to bounding $\cw_\epsilon$ from below. Notice that solutions to \eqref{eqweps} satisfy the symmetry condition (S) from Theorem \ref{secondthm}: this follows just from the classical maximum/minimum principle and concavity of $H$ on $[0,\infty)$. It follows that, once again using the comparison principle, we only need to obtain a lower bound on solutions to
\[
\partial_s v_\epsilon(s,y)-4\partial^2_yv_\epsilon(s,y)=\epsilon^{\left(\frac{\beta+1}{2}-\frac{1}{p}\right)}g(s)H'(y)v_\epsilon(s,y),\quad v_{\epsilon}(0,y)=e^{-y^2}.
\]
Indeed: we have $\partial_y\cw_{\epsilon}\leq0$ on $[0,\infty)$, and by symmetry, it must be non-negative on $(-\infty,0]$. As $H$ is positive on $(0,\infty)$ (and negative on $(-\infty,0)$ by symmetry), the drift term is destabilizing: $-H(y)\partial_y\cw\geq0$ always. Thus, $\cw_\epsilon$ is bounded from below by $v_\epsilon$ (pointwise). By Duhamel's principle, we have 
\begin{equation}\label{duhrep}
	v_\epsilon(s,y)=\intsol\Psi(s,y-z)e^{-z^2}dz+\epsilon^{\left(\frac{\beta+1}{2}-\frac{1}{p}\right)}\int_0^sg(r)\intsol\Psi(s-r,y-z) H'(z)v_\epsilon(r,z)dzdr.
\end{equation}
As $v_\epsilon\geq0$, we have $v_\epsilon(s,y)\geq G_0(s,y)$, where 
\[
G_0(s,y):=\intsol\Psi(s,y-z)e^{-z^2}dz,
\] 
leading to 
\[
v_\epsilon(s,y)\geq G_0(s,y)+\epsilon^{\left(\frac{\beta+1}{2}-\frac{1}{p}\right)}\int_0^sg(r)\intsol\Psi(s-r,y-z) H'(z) G_0(r,z)dzdr.
\]
If we define 
\[
G_1(s,y):=\int_0^sg(r)\intsol\Psi(s-r,y-z) H'(z)G_0(r,z)dzdr,
\]
then $v_\epsilon(s,y)\geq G_0(s,y)+\epsilon^{\left(\frac{\beta+1}{2}-\frac{1}{p}\right)} G_1(s,y)$. Plugging this again into the Duhamel representation \eqref{duhrep}:
\[
v_\epsilon(s,y)\geq G_0(s,y)+\epsilon^{\left(\frac{\beta+1}{2}-\frac{1}{p}\right)}G_1(s,y)+\epsilon^{2\left(\frac{\beta+1}{2}-\frac{1}{p}\right)}\int_0^sg(r)\intsol\Psi(s-r,y-z) H'(z)G_1(r,z)dzdr,
\]
and so defining 
\[
G_2(s,y):=\int_0^sg(r)\intsol\Psi(s-r,y-z) H'(z)G_1(r,z)dzdr,
\]
we get $v_\epsilon(s,y)\geq G_0(s,y)+\epsilon^{\left(\frac{\beta+1}{2}-\frac{1}{p}\right)}G_1(s,y)+\epsilon^{2\left(\frac{\beta+1}{2}-\frac{1}{p}\right)} G_2(s,y)$. Defining $G_j$ iteratively, 
\[
G_j(s,y):=\int_0^sg(r)\intsol\Psi(s-r,y-z) H'(z)G_{j-1}(r,z)dzdr,\quad G_0(s,y):=\intsol\Psi(s,y-z)e^{-z^2}dz,
\]
we see that a straightforward inductive argument yields 
\[
v_\epsilon(s,y)\geq \sum_{j=0}^N\epsilon^{j\left(\frac{\beta+1}{2}-\frac{1}{p}\right)}G_j(s,y).
\]
The conditions $p\in(0,2/(\beta+1))$ and $\beta\in(0,1)$ guarantees the existence of a sufficiently large $N\in\mathbb{N}$ such that 
\[
N\left(\frac{\beta+1}{2}-\frac{1}{p}\right)\leq-3.
\]
The claim follows immediately by noting that $\|\cv_\epsilon(t,\cdot)\|_{L^{\infty}}=\cv_\epsilon(t,0)\geq v_\epsilon(t/\epsilon,0)$ and so
\[
\int_0^T\|\cv_\epsilon(t,\cdot)\|_{L^{\infty}}dt\geq\int_0^Tv_\epsilon\left(\frac{t}{\epsilon},0\right)dt=\epsilon\int_0^{T/\epsilon}v_\epsilon(s,0)ds\geq C\epsilon^{-1}.
\]
\end{proof}
An interesting consequence of this counter example is that one can use it to prove that solutions to the non-local problem
\begin{align}
&\partial_t\Omega-4\partial^2_\xi\Omega\geq g(t)\int_0^\xi\bar{h}_{\epsilon}'(\eta)\partial_\eta\Omega(t,\eta)d\eta,\quad &&(t,\xi)\in(0,\infty)\times(0,\infty),\label{pdeineqmin}\\
&\Omega(0,\xi)\geq0,\quad &&\xi\geq0\label{pdeidmin},\\
&\Omega(t,0)\geq0,\quad &&t\geq0,\label{pdebcmin}
\end{align}
in general do not satisfy a minimum principle. We start by showing that given any $k_1\geq0$ and $k_2>0$, one can always construct a solution to 
\begin{align}
&\partial_t\Omega-4\partial^2_\xi\Omega= g(t)\int_0^\xi\bar{h}_{\epsilon}'(\eta)\partial_\eta\Omega(t,\eta)d\eta+k_1,\quad &&(t,\xi)\in(0,\infty)\times(0,\infty),\label{pdeineqmin2}\\
&\Omega(0,\xi)\geq0,\quad &&\xi\geq0\label{pdeidmin2},\\
&\Omega(t,0)=tk_1+k_2,\quad &&t\geq0,\label{pdebcmin2}
\end{align}
that becomes negative for some $(t,\xi)\in(0,\infty)\times(0,\infty)$. Let $\cv_\epsilon$ be the sequence of solutions constructed in the proof of Proposition \ref{propce} corresponding to the case for $p=1$, $\beta=1/2$, and $\mu_1=0$:
\[
\partial_t\cv_\epsilon(t,\xi)-4\partial^2_\xi\cv_\epsilon(t,\xi)=g_\epsilon(t)\bar{h}_{\epsilon}'(\xi)\cv_\epsilon(t,\xi),\quad \cv_\epsilon(0,\xi)=e^{-\left(\frac{\xi}{\sqrt{\epsilon}}\right)^2},
\]
and recall the existence of a sequence of non-negative functions $\{G_j\}$, independent of $\epsilon$, such that  
\begin{equation}\label{lwrbd}
	\cv_\epsilon(t,\xi)\geq \sum_{j=0}^N\epsilon^{-\frac{j}{4}}G_j(t/\epsilon,\xi/\sqrt{\epsilon}).
\end{equation}
Thus, if we define 
\begin{equation}\label{negsolapp}
	\Omega(t,\xi):=tk_1+k_2-\int_0^\xi\cv_\epsilon(t,\eta)d\eta,\quad (t,\xi)\in[0,\infty)\times(0,\infty),
\end{equation}
then provided $\epsilon$ is small enough depending on $k_2$, we must have $\Omega(0,\xi)\geq0$:
\begin{equation}\label{icce}
\Omega(0,\xi)=k_2-\sqrt{\epsilon}\int_0^{\xi/\sqrt{\epsilon}}e^{-\eta^2}d\eta.
\end{equation}
Since $\cv_\epsilon$ is even in the spatial variable and constants solve the non-local PDE with homogenous forcing, $\Omega$ solves \eqref{pdeineqmin2}-\eqref{pdebcmin2}. It is clear from the lower bound on $\cv_\epsilon$ \eqref{lwrbd} that $\Omega$ becomes negative for some $(t,\xi)\in(0,\infty)\times(0,\infty)$. In fact, due to the lower bound, one can make $\Omega$ smaller than any given finite negative number by choosing $\epsilon$ small enough. This gives us a counter-example to the minimum principle for positive Dirichlet data. We now prove the falsehood of the minimum principle even with homogenous Dirichlet data. The main idea is to shift the previous example by a small amount $\delta$, introduce a jump discontinuity in the initial data, and argue by contraduction. To that end, suppose it is true, that is, suppose that any solution $\Omega$ to \eqref{pdeineqmin}-\eqref{pdebcmin} is non-negative. Let
\[
u_0(\xi):=
\begin{cases}
	\xi,&\xi\in[0,\delta],\\
	\Omega(0,\xi-\delta),&\xi>\delta,
\end{cases}
\]
where $\Omega(0,\cdot)$ is as defined in \eqref{icce} with $k_2=80$. Let $u$ solve 
\[
\partial_tu-4\partial^2_\xi u=g_\epsilon(t)\bar{h}_{\epsilon}'(\xi)u(t,\xi)-g_\epsilon(t)\int_0^\xi\bar{h}_{\epsilon}''(\eta) u(t,\eta)d\eta,\quad (t,\xi)\in(0,\infty)\times(0,\infty)
\]
with data $u(0,\xi)=u_0(\xi)$, $u(t,0)=0$. One can construct a solution by the standard method of reflection coupled with Duhamel's principle, converting the problem to a Volterra-integral equation of the second kind that can be solved by Neumann series for instance; we leave the details to the interested reader. Courtesy of classical parabolic regularity, the jump discontinuity in the initial data is smoothened out immediately, so that an integration by parts reveals
\[
\partial_tu-4\partial^2_\xi u=g_\epsilon(t)\int_0^\xi\bar{h}_{\epsilon}'(\eta) \partial_\eta u(t,\eta)d\eta,\quad (t,\xi)\in(0,\infty)\times(0,\infty),
\]
whence $u\geq0$ always as we assumed the validity of the minimum principle. Now, let $\Omega$ be \eqref{negsolapp} with $k_1=k_2=80$, and notice that for $\xi\geq\delta$, if we define
\[
F_\delta(t,\xi):= g_\epsilon(t)\int_0^\delta\bar{h}_{\epsilon}'(\eta)\partial_\eta u(t,\eta)d\eta+g_\epsilon(t)\int_0^{\xi-\delta}[\bar{h}_{\epsilon}'(\eta+\delta)-\bar{h}_{\epsilon}'(\eta)]\partial_\eta \Omega(t,\eta)d\eta
\] 
and set $w(t,\xi):=\Omega(t,\xi-\delta)-u(t,\xi)$ then $w$ solves 
\begin{align*}
&\partial_tw-4\partial^2_\xi w=g_\epsilon(t)\int_{\delta}^{\xi}\bar{h}_{\epsilon}'(\eta)\partial_\eta w(t,\eta)d\eta+80-F_\delta(t,\xi),\quad &&(t,\xi)\in(0,\infty)\times(\delta,\infty),\\
&w(0,\xi)\geq0,\quad &&\xi\geq\delta,\\
&w(t,\delta)\geq80-u(t,\delta),\quad &&t\geq0,
\end{align*}
Observe that $u(t,\delta)$ converges to the average of $u_0(\delta^+)$ and $u_0(\delta^-)$ as $t\rightarrow0^+$, while $\|\partial_\xi u(t,\cdot)\|_{L^{\infty}}$ is bounded uniformly in $t$ provided $t\gtrsim 1$ (it will be of order $\exp(\epsilon^{-1/4})$, so the bound is uniform in $\delta$ as well). Thus, if $\delta$ is small enough depending on $\epsilon$, we must have $w(t,\delta)\geq0$ always. Furthermore, a straightforward calculation reveals that $\|\partial_\xi u\|_{L_t^{\infty}L_x^{1}}+\|\partial_\xi \Omega\|_{L_t^{\infty}L_x^{1}}$ is controlled by $\exp(\epsilon^{-1/4})$, which makes $80-F_{\delta}(t,\xi)\geq 30$ if $\delta$ is small enough. That last statement follows from an integration by parts in the first integral in $F_\delta$, followed by the bound $\|u(t,\cdot)\|_{L^{\infty}}\leq\|\partial_\xi u(t,\cdot)\|_{L^1}$, together with using the inequality $|\bar{h}_{\epsilon}'(\eta+\delta)-\bar{h}_{\epsilon}'(\eta)|\leq\epsilon^{-1/2}\delta^{1/2}$ to handle the second integral. We now encourage the reader to readily verify that an implication of assuming the validity of the minimum principle is non-negativity of solutions to
\begin{align*}
&\partial_tw-4\partial^2_\xi w\geq g(t)\int_\delta^\xi\bar{h}_{\epsilon}'(\eta)\partial_\eta w(t,\eta)d\eta+30,\quad &&(t,\xi)\in(0,\infty)\times(\delta,\infty),\\
&w(0,\xi)\geq0,\quad &&\xi\geq\delta,\\
&w(t,\delta)\geq0,\quad &&t\geq0,
\end{align*}
provided $\|\partial_\xi w\|_{L_t^{\infty}L_x^{1}}$ is bounded independent of $\delta$ and the latter is chosen small enough. It is therefore clear that by choosing $\delta$ small enough depending on $\epsilon$, the assumption that the minimum principle holds implies that $w\geq0$, i.e. $\Omega(t,\xi-\delta)\geq u(t,\xi)$ whenever $\xi\geq\delta$. But this clearly means that $u(t,\xi)<0$ for some positive $t$ and $\xi$ since $\Omega$ is; absurd.
\section{A special case of Ladyzhenskaja-Prodi-Serrin}\label{applps}
Here, we give a simple proof of the fact that $L_t^qL_x^{\infty}$ solutions to the incompressible NSE \eqref{NSE} are regular provided $q>2$, for the sake of completeness. We start by proving a singular version of Gronwall's inequality. Such an inequality is more than likely to be proved somewhere (it is very hard to track all the different versions of Gronwall's inequality), nevertheless, we provide a proof here for the sake of convenience. Bound \eqref{gineqest}, below, is less than likely to be optimal, but it will do for our purposes.
\begin{lem}[Singular Gronwall]\label{singgron}
Let $f,h,g:[0,T]\rightarrow[0,\infty)$ be non-negative smooth functions such that 
\begin{equation}\label{gineqasum}
	f(t)\leq h(t)+\int_0^t(t-s)^{-\alpha}g(s)f(s)ds,\quad \forall t\in[0,T],
\end{equation}
for some $\alpha\in[0,1)$. It follows that for any $q\in(1/(1-\alpha),\infty)$, there exists a positive constant $C=C(q,\alpha,T)>0$ such that 
\begin{equation}\label{gineqest}
f(t)\leq h(t)+Ce^{C\|g\|^q_{L^q}}\left(\int_0^th^q(s)g^q(s)ds\right)^{1/q}, \quad \forall t\in(0,T].
\end{equation}
\end{lem}
\begin{proof}
	As $q>1/(1-\alpha)$, its H\"older conjugate $p$ satisfies $p\alpha<1$, so that by H\"older's inequality, there must exist a positive constant $C=C(q,\alpha,T)$ for which
	\[
	f(t)\leq h(t)+C(\eta(t))^{1/q}, \quad \eta(t):=\int_0^tg^q(s)f^q(s)ds,\quad \forall t\in[0,T].
	\]
	It follows that 
	\[
	\eta'(t)\leq 2^qh^q(t)g^q(t)+2^qC^qg^q(t)\eta(t),
	\]
	so that by the standard Gronwall inequality,
	\[
	\eta(t)\leq e^{C\|g\|_{L^q}^q}2^q\int_0^th^q(s)g^q(s)ds,
	\]
	from which the result follows. One can (potentially) obtain a sharper bound by plugging \eqref{gineqasum} into itself $n$ times (to absorb the singular kernel) before applying standard Gronwall inequality. This would still require $q>1/(1-\alpha)$, but the resulting bound could poossibly be sharper than the claimed one.
\end{proof}

Let $u_0$ be a given, smooth, divergence free vector field, and let $(u,p)$ be a smooth solution to \eqref{NSE}. Our aim is to show that if there exists a $q\in(2,\infty)$ and $M>0$ such that 
\begin{equation}\label{ac1}
	\int_0^{T_*}\|u(t,\cdot)\|^q_{L^{\infty}}dt\leq M,
\end{equation}
then there exists an $M_1$ such that
\[
\sup_{t\in[0,T_*)}\|\nabla u(t,\cdot)\|_{L^{\infty}}\leq M_1,
\]
from which regularity easily follows. To that end, we first obtain point-wise estimates on $\nabla p$ via the representation \eqref{reppress}
\[
\nabla p(t,x)=\int_{\rd} \left[u_i(t,x-z)-u_i(t,x)\right][u_j(t,x-z)-u_j(t,x)]\partial_i\partial_j\nabla\phi(z)dz,
\]
with $\phi(z):=C_d|z|^{2-d}$ being the fundamental solution to the Laplace equation. Notice that the kernel decays sufficiently rapidly at infinity, so that one can make sense of the integral even for periodic $u$. This was done rigorously in our previous work \cite[Lemma 4.1]{Ibdah2022}. From the above representation, it is straightforward to verify that for any $\rho>0$, we have
\[
|\nabla p(t,x)|\leq C_d\left[\|\nabla u(t,\cdot)\|_{L^{\infty}}^2\rho+\rho^{-1}\|u(t,\cdot)\|_{L^{\infty}}^2\right].
\]
Optimizing in $\rho$ and taking a supremum in $x$ on the left-hand side yields the bound 
\begin{equation}\label{ac2}
	\|\nabla p(t,\cdot)\|_{L^{\infty}}\leq C_d\|\nabla u(t,\cdot)\|_{L^{\infty}}\|u(t,\cdot)\|_{L^{\infty}}.
\end{equation}
An application of Duhamel's principle tells us:
\[
u(t,x)=\int_{\rd}\Psi(t,y)u_0(x-y)dy-\int_0^t\int_{\rd}\Psi(t-s,x-y)u(s,y)\cdot\nabla u(s,y)dyds-\int_0^t\int_{\rd}\Psi(t-s,x-y)\nabla p(s,y)dyds,
\]
where $\Psi$ is the heat kernel. Again, the above representation for the solution $u$ makes sense for both the periodic and whole space setting. Taking a derivative and bounding the nonlinearity by $\|u(s,\cdot)\|_{L^{\infty}}\|\nabla u(s,\cdot)\|_{L^{\infty}}$ and the pressure term by the same (owing to \eqref{ac2}) we get 
\[
|\nabla u(t,x)|\leq \int_{\rd}\Psi(t,y)|\nabla u_0(x-y)|dy+C_d\int_0^t\|u(s,\cdot)\|_{L^{\infty}}\|\nabla u(s,\cdot)\|_{L^{\infty}}\int_{\rd}|\nabla \Psi(t-s,x-y)|dyds.
\]
A simple calculation reveals that $\|\nabla \Psi(t-s,\cdot)\|_{L^1}\leq \nu^{-1/2}(t-s)^{-1/2}$, and so taking a supremum in $x$ on the left-hand side we arrive at 
\[
\|\nabla u(t,\cdot)\|_{L^{\infty}}\leq \|\nabla u_0\|_{L^{\infty}}+C_{d,\nu}\int_0^t(t-s)^{-1/2}\|u(s,\cdot)\|_{L^{\infty}}\|\nabla u(s,\cdot)\|_{L^{\infty}}ds.
\]
The claim follows by the singular Gronwall inequality, Lemma \ref{singgron} above, with $\alpha=1/2$, via utilizing \eqref{ac1}.
\end{appendices}
\bibliographystyle{abbrv}
\bibliography{mybib}
\end{document}